\documentclass[12pt]{amsart}

\usepackage{amsthm,amsmath,amssymb,verbatim,dsfont}                      
\usepackage[all]{xy}
\usepackage{mathdots}
\usepackage{mathtools}
\usepackage{tikz} \usetikzlibrary{matrix,arrows}\pgfrealjobname{dynkingrassmann2}  
\usepackage{extarrows}
\usepackage[linktoc=page,colorlinks,linkcolor={red!80!black},citecolor={red!80!black},urlcolor={blue!80!black}]{hyperref}
 
\usepackage{mathptmx}                                                              
\DeclareSymbolFont{cmletters}{OML}{cmm}{m}{it}                                     
\DeclareSymbolFont{cmsymbols}{OMS}{cmsy}{m}{n}
\DeclareSymbolFont{cmlargesymbols}{OMX}{cmex}{m}{n}
\DeclareMathSymbol{\myjmath}{\mathord}{cmletters}{"7C}     \let\jmath\myjmath 
\DeclareMathSymbol{\myamalg}{\mathbin}{cmsymbols}{"71}     
\DeclareMathSymbol{\mycoprod}{\mathop}{cmlargesymbols}{"60}\let\coprod\mycoprod
\DeclareMathSymbol{\myalpha}{\mathord}{cmletters}{"0B}     \let\alpha\myalpha 
\DeclareMathSymbol{\mybeta}{\mathord}{cmletters}{"0C}      \let\beta\mybeta
\DeclareMathSymbol{\mygamma}{\mathord}{cmletters}{"0D}     \let\gamma\mygamma
\DeclareMathSymbol{\mydelta}{\mathord}{cmletters}{"0E}     \let\delta\mydelta
\DeclareMathSymbol{\myepsilon}{\mathord}{cmletters}{"0F}   \let\epsilon\myepsilon
\DeclareMathSymbol{\myzeta}{\mathord}{cmletters}{"10}      \let\zeta\myzeta
\DeclareMathSymbol{\myeta}{\mathord}{cmletters}{"11}       \let\eta\myeta
\DeclareMathSymbol{\mytheta}{\mathord}{cmletters}{"12}     \let\theta\mytheta
\DeclareMathSymbol{\myiota}{\mathord}{cmletters}{"13}      \let\iota\myiota
\DeclareMathSymbol{\mykappa}{\mathord}{cmletters}{"14}     \let\kappa\mykappa
\DeclareMathSymbol{\mylambda}{\mathord}{cmletters}{"15}    \let\lambda\mylambda
\DeclareMathSymbol{\mymu}{\mathord}{cmletters}{"16}        \let\mu\mymu
\DeclareMathSymbol{\mynu}{\mathord}{cmletters}{"17}        \let\nu\mynu
\DeclareMathSymbol{\myxi}{\mathord}{cmletters}{"18}        \let\xi\myxi
\DeclareMathSymbol{\mypi}{\mathord}{cmletters}{"19}        \let\pi\mypi
\DeclareMathSymbol{\myrho}{\mathord}{cmletters}{"1A}       \let\rho\myrho
\DeclareMathSymbol{\mysigma}{\mathord}{cmletters}{"1B}     \let\sigma\mysigma
\DeclareMathSymbol{\mytau}{\mathord}{cmletters}{"1C}       \let\tau\mytau
\DeclareMathSymbol{\myupsilon}{\mathord}{cmletters}{"1D}   \let\upsilon\myupsilon
\DeclareMathSymbol{\myphi}{\mathord}{cmletters}{"1E}       \let\phi\myphi
\DeclareMathSymbol{\mychi}{\mathord}{cmletters}{"1F}       \let\chi\mychi
\DeclareMathSymbol{\mypsi}{\mathord}{cmletters}{"20}       \let\psi\mypsi
\DeclareMathSymbol{\myomega}{\mathord}{cmletters}{"21}     \let\omega\myomega
\DeclareMathSymbol{\myvarepsilon}{\mathord}{cmletters}{"22}\let\varepsilon\myvarepsilon
\DeclareMathSymbol{\myvartheta}{\mathord}{cmletters}{"23}  \let\vartheta\myvartheta
\DeclareMathSymbol{\myvarpi}{\mathord}{cmletters}{"24}     \let\varpi\myvarpi
\DeclareMathSymbol{\myvarrho}{\mathord}{cmletters}{"25}    \let\varrho\myvarrho
\DeclareMathSymbol{\myvarsigma}{\mathord}{cmletters}{"26}  \let\varsigma\myvarsigma
\DeclareMathSymbol{\myvarphi}{\mathord}{cmletters}{"27}    \let\varphi\myvarphi

\newcommand{\Sc}[2]{\langle #1,#2\rangle}
\newcommand{\Hom}{\mathrm{Hom}} 
\newcommand{\Ext}{\mathrm{Ext}}

\newcommand{\Rep}{\mathrm{Rep}}

\newcommand{\arst}{p\xlongrightarrow{v}q}
\newcommand{\ses}[3]{0\rightarrow #1\rightarrow #2\rightarrow#3\rightarrow 0}

\def\op{{\mathrm{op}}}

\theoremstyle{plain}
\newtheorem{thm}{Theorem}[section]
\newtheorem{cor}[thm]{Corollary}
\newtheorem{lemma}[thm]{Lemma}
\newtheorem{prop}[thm]{Proposition}

\newtheorem*{thmA}{Theorem A}
\newtheorem*{thmB}{Theorem B}

\newtheorem*{thmC}{Theorem C}
\theoremstyle{definition}
\newtheorem{df}[thm]{Definition}
\newtheorem{rem}[thm]{Remark}
\newtheorem{ex}[thm]{Example}

\numberwithin{equation}{section}

\def\A{{\mathbb A}}

\def\C{{\mathbb C}}
\def\F{{\mathbb F}}

\def\N{{\mathbb N}}
\def\P{{\mathbb P}}
\def\Q{{\mathbb Q}}

\def\Z{{\mathbb Z}}

\def\cB{{\mathcal B}}

\def\cE{{\mathcal E}}

\def\cV{{\mathcal V}}

\def\udim{{\underline{\dim}\, }}

\DeclareMathOperator{\Gr}{Gr}

\begin{document}

\title[Quiver Grassmannians of type $\widetilde{D}_n$]{Quiver Grassmannians of type $\widetilde{D}_n$,\\[10pt] Part 2: Schubert decompositions and $F$-polynomials}

\date{}

\author{Oliver Lorscheid}
\address{Instituto Nacional de Matem\'atica Pura e Aplicada, Estrada Dona Castorina 110, Rio de Janeiro, Brazil}
\email{oliver@impa.br}

\author{Thorsten Weist}
\address{Bergische Universit\"at Wuppertal, Gau\ss str.\ 20, 42097 Wuppertal, Germany}
\email{weist@uni-wuppertal.de}

\begin{abstract}
Extending the main result of \cite{LW}, in the first part of this paper we show that every quiver Grassmannian of an indecomposable representation of a quiver of type $\tilde D_n$ has a decomposition into affine spaces. In the case of real root representations of small defect, the non-empty cells are in one-to-one correspondence to certain, so called non-contradictory, subsets of the vertex set of a fixed tree-shaped coefficient quiver. In the second part, we use this characterization to determine the generating functions of the Euler characteristics of the quiver Grassmannians (resp. $F$-polynomials). Along these lines, we obtain explicit formulae for all cluster variables of cluster algebras coming from quivers of type $\tilde D_n$. 
\end{abstract}

\maketitle

\begin{small}  
 \tableofcontents 
\end{small}

\section*{Introduction} 
\noindent In this paper we continue the consideration of quiver Grassmannians of type $\tilde D_n$ initiated in \cite{LW}. Denoting the unique imaginary Schur root by $\delta$, there it is shown that every quiver Grassmannian of  a real root representation of dimension $\alpha$ with $\Sc{\delta}{\alpha}=-1$ has a cell decomposition into affine spaces. Moreover, it is also shown that this is true for every indecomposable representation lying in an exceptional tube and for every Schur representation of dimension $\delta$.

Passing to dual representations, this result can be easily extended to all indecomposable real root representations of dimension $\alpha$ of small defect, i.e. $|\Sc{\delta}{\alpha}|\leq 1$. 
We use this result to obtain the first main result of this paper, which says that this statement is in fact true for every indecomposable representation of type $\tilde D_n$. 

The focus of the second part of the paper is on the generating functions of the Euler characteristics of quiver Grassmannians (resp. $F$-polynomials) of indecomposable representations of $\tilde D_n$, i.e. for a fixed representation $M$ of $Q$, we consider 
\[F_M(x)=\sum_{e\in\N Q_0}\chi(\Gr_e(M))x^e.\] Thanks to the Caldero-Chapoton-formula, see \cite{cc} and \cite{calkel2}, this also builds the bridge to cluster algebras, which were introduced in \cite{Fomin-Zelevinsky02} and whose theory developed rapidly within the last ten years. We refer to the introduction of the first part (\cite{LW}) for more details.

Initially, we use the combinatorial description of the non-empty cells in terms of non-contra\-dictory subsets of a particular coefficient quiver in \cite{LW} to obtain explicit formulae for the $F$-polynomials of indecomposable representations of small defect. Surprisingly, it turns out that the $F$-polynomials of all preprojective (resp. preinjective) representations of small defect only depend on the $F$-polynomials of certain preprojective (resp. preinjective) representations, whose dimension vector is smaller than the imaginary Schur root $\delta$, and the $F$-polynomials of representations lying in homogeneous tubes. For the latter ones, we also have explicit descriptions which then gives explicit formulae for all $F$-polynomials. For representations lying in exceptional tubes an analogous phenomenon arises. 

Subsequently, the results of the first part can be used to obtain explicit formulae for the $F$-polynomials of all indecomposable representations of dimension $\alpha$ of large defect, i.e. $|\Sc{\delta}{\alpha}|=2$. 

Using the Caldero-Chapoton-formula, these results can now be used to obtain an explicit description of all cluster variables of mutation finite cluster algebras coming from quivers of type $\tilde D_n$. 

\subsection*{Connections to previous results}
The formulae for $F$-polynomials of representations of large defect can also be obtained by applying  the multiplication formula of \cite{calkel}.

The formulae we obtain for representations of small defect differ from those present in the literature as they state relations between $F$-polynomials from different components of the Auslander-Reiten quiver. Moreover, it is possible to state an explicit formula for the $F$-polynomial of any representation in terms of $F$-polynomials of indecomposables whose dimension is smaller than $\delta$.  Those formulae which are known to us and which are present in the literature are mostly of recursive nature and between $F$-polynomials of representations from one component of the Auslander-Reiten quiver, see \cite{KS11,dupont2012cluster}. But clearly, it would be interesting to investigate if there is a direct way to obtain our formulae from the present recursive formulae.

 As far as mutation finite cluster algebras are concerned, the approach of calculating the $F$-polynomials in order to determine cluster variables was mainly applied to cluster algebras of type $A$ and only partially for type $D$, see \cite{Cerulli11}, \cite{Cerulli-Esposito11}, \cite{Haupt12}, \cite{dupont2012cluster} and \cite{Dupont12b}. 

Finally, we note that cluster algebras of type $\tilde D_n$ also arise from surfaces. This yields a combinatorial description of the cluster variables in terms of perfect matchings of edge-weighted graphs coming from triangulations of the corresponding surface, see \cite{Musiker-Schiffler-Williams11}. Thus, in theory, the Euler characteristics could be determined using this approach, but this has not been carried out yet in the case of quivers of type $\tilde D_n$. In particular, it is not clear if our explicit formulae for the $F$-polynomials can be obtained using this description. From the representation-theoretic point of view, this approach would also be unsatisfactory as it does not use the geometry of the quiver Grassmannians themselves. We should point out that it is not clear at all how these two combinatorial descriptions fit together. Actually, this would be very interesting to investigate. However, since the shape of the formulae, which are obtained with our approach, are indeed easy, there is hope for a generalization to other mutation finite cluster algebras.  
\subsection*{Schubert decomposition}
For Schubert decompositions of quiver Grassmannians of indecomposable real root representations $M$ of small defect of a quiver of type $\tilde D_n$, which are in fact cell decompositions into affine spaces, we consider the coefficient quivers $\Gamma_M$ listed in \cite[Appendix B]{LW}.  Recall that every subset $\beta\subset(\Gamma_M)_0$ of cardinality $e$ defines a possibly empty Schubert cell $C_{\beta}^M\subset\Gr_e(M)$ induced by the Schubert decompositions of the product of usual Grassmannians $\prod_{q\in Q_0}\Gr_{e_q}(M_q)$. The first aim of this paper is to generalize Theorem 4.4 of \cite{LW} to all indecomposable representations of $\tilde D_n$, i.e.:
\begin{thmA}\label{tmt}
Let $M$ be an indecomposable representation of $\tilde D_n$. Then there exists a coefficient quiver $\Gamma_M$ of $M$ such that the Schubert decomposition $\Gr_e(M)=\coprod C_\beta^M$ is a decomposition into affine spaces and empty cells. Here $\beta$ runs through all subsets of $(\Gamma_M)_0$ of cardinality $e$.
\end{thmA}

The generalization of Theorem 4.4 of \cite{LW} to representations of large defect and to representations of the homogeneous tubes is subject of section \ref{sec2}. Since the quiver Grassmannians of representations lying in the homogeneous tubes behave similar to those of large defect, throughout the paper, we exclude them when referring to representations of small defect. While the construction of the cell decompositions of quiver Grassmannians of representations of small defect is highly combinatorial, in the cases of large defect the main idea is to consider exact sequences which are close to being almost split. It turns out that every indecomposable representation $B$ of large defect can be written as the middle term of such a sequence between indecomposables $M$ and $N$ of small defect. In particular, there exists a coefficient quiver of $B$ with vertex set $(\Gamma_M)_0\cup(\Gamma_N)_0$ where $\Gamma_M$ and $\Gamma_N$ are those considered in \cite{LW}. Generalizations of results of \cite{cc} can be used to show that this setup preserves cell decompositions in such a way that every pair of subsets $(\beta,\beta')$ of $(\Gamma_M)_0\times(\Gamma_N)_0$ determines a (possibly empty) cell $C^B_{\beta,\beta'}$ of a certain quiver Grassmannian of the middle term which turns out to be an affine space.
Since all cells can be obtained using this construction, this already proves that every quiver Grassmannian of indecomposables of large defect has a cell decomposition into affine spaces, see Theorem \ref{almost2} and also section \ref{noncontradictory} for the notion of non-contradictory subsets:
\begin{thmB}
Let $B$ be a real root representation of defect $-2$. Then there exist indecomposable representations $M$ and $N$ of defect $-1$ and respective coefficient quivers $\Gamma_M$ and $\Gamma_N$ such that the (induced) Schubert decomposition 
\[\Gr_e(B)=\coprod C_{(\beta,\beta')}^B\]
is a decomposition into affine spaces and empty cells. Here $(\beta,\beta')$ runs through all non-contra\-dictory subsets of $(\Gamma_M)_0\times(\Gamma_N)_0$ such that the cardinalities of $\beta$ and $\beta'$  sum up to $e$. 
\end{thmB}

There is also a very explicit description in terms of the Auslander-Reiten quiver of those pairs corresponding to an empty cell. As shown in \cite{cc}, in the case of almost split sequences there is only one such pair, consisting of the cokernel and the trivial subrepresentation.

With similar methods we can also show that every quiver Grassmannian of an indecomposable representation lying in one of the homogeneous tubes has a cell decomposition into affine spaces, see Theorem \ref{homogaff}.

In this paper and also in \cite{LW} we consider preprojective representations rather than preinjective ones. In section \ref{posdef}, we prove that passing to the opposite quiver and to dual representations Schubert decompositions are preserved. Thus all results can be transferred to the case of preinjective representations in the natural way.

Theorems A and B have strong implications on the geometry of $\Gr_e(M)$, see \cite[Section 6]{L14}. In particular, we can compute the Euler characteristic of $\Gr_e(M)$ as
\[
 \chi\bigl(\Gr_e(M)\bigr) \quad = \quad \# \bigl\{ \beta\subset(\Gamma_M)_0\text{ of type }e\text{ such that }C_\beta^M\text{ is not empty} \bigr\}.
\]
 If, in addition, $\Gr_e(M)$ is smooth, the closures of the non-empty Schubert cells form an additive basis for the singular cohomology ring of $\Gr_e(M)$ and it follows that the cohomology is concentrated in even degree.

The construction of the cell decompositions in terms of those of representations of small defect yields a description of the $F$-polynomial of the indecomposables of large defect, see Theorem \ref{Fdef2}. As already mentioned, in terms of cluster algebras this result translates to the well-known multiplication formula of \cite{calkel}. As far as cluster variables are concerned, we are thus left with the determination of $F$-polynomials of indecomposables of small defect. The investigation of $F$-polynomials and the derivation of explicit formulae is the main topic of sections \ref{reductionofgrass}, \ref{bgp} and \ref{secFpoly}. 
\subsection*{Calculation of $F$-poynomials} As already mentioned, the $F$-polynomials of representations of the homogeneous tubes play an important role in the formulae for $F$-polynomials of arbitrary indecomposable representations. They only depend on the dimension vector and are independent of the chosen tube, we denote them by $F_{r\delta}$ throughout the paper. This is straightforward with the methods of this paper, but also known for general affine quivers, see \cite[Lemma 5.3]{Dupont11}. With the results obtained there, also the $F$-polynomials of general representations (which means that they decompose into certain Schurian representations) of non-Schurian roots, can be determined recursively.

Considering the cell decompositions into affine spaces, we first obtain a recursive formula for $F_{r\delta}$ which can be used to obtain an explicit formula in terms of $F_{\delta}$ in Corollary \ref{homogform2}:
\[F_{r\delta}=\frac{1}{2z}(\lambda_+^{r+1}-\lambda_-^{r+1}),\text{ where }\, z=\frac{1}{2}\sqrt{F_{\delta}^2-4x^{\delta}},\quad\lambda_{\pm}=\frac{F_{\delta}}{2}\pm z.\]

Besides the combinatorial description of the non-empty cells, there are two other main ingredients which are used to obtain explicit, i.e. non-recursive, formulae for the $F$-polynomials of indecomposable representations. The first one is studied in section \ref{reductionofgrass}. The main idea is to reduce the determination of the $F$-polynomials to smaller quivers, i.e. $\tilde D_n$ for $n\leq 6$. Here we use that most linear maps corresponding to indecomposable representations of $\tilde D_n$ for large $n$ are isomorphisms. In combination with the reflection functor introduced in \cite{bgp}, it turns out that this is a powerful tool. In section \ref{bgp}, we review the reflection functor and its consequences for quiver Grassmannians, which were also studied in \cite{wol} and \cite{dwz2}. 


If $\tilde D_n$ is in subspace orientation, we are left with counting admissible subsets as defined in section \ref{admsub}. Since the coefficient quivers under consideration follow a certain recursion and since the description of these subsets is very easy (and again easier for $n\leq 6$), we get recursive formulae for the $F$-polynomials. It turns out that these recursive formulae can be used to obtain explicit formulae for all $F$-polynomials of real root representations of small defect in section \ref{secFpoly}. All these explicit formulae are in terms of the $F$-polynomials $F_{r\delta}$ from above and of certain indecomposable representations whose dimension is smaller than $\delta$. Since there also exists an explicit formula for $F_{r\delta}$ in terms of $F_{\delta}$, we are left with the easy task of calculating $F$-polynomials of representations of dimension $\alpha\leq\delta$. While these $F$-polynomials do depend on the orientation of $\tilde D_n$, the upshot is that the formulae for the remaining $F$-polynomials turn out to be independent of the orientation. 

In order to state the main result of the second part, we need some notation. If $\alpha$ is a real root, we denote by $M_{\alpha}$ the unique indecomposable representation of this dimension and by $F_{\alpha}$ the corresponding $F$-polynomial. In a tube of rank $t$ there exist $t$ chains of irreducible morphisms
\[M_{0,1}\hookrightarrow
 M_{0,2}\hookrightarrow\ldots\hookrightarrow M_{0,t-1}\hookrightarrow
M_{1,0}:=M_{\delta}\hookrightarrow M_{1,1}\hookrightarrow\ldots\]
where the $m_{l}(r):=\udim M_{r,l}$ are real roots and the indecomposable representations $M_{r,0}$ of dimension $r\delta$ are uniquely determined by this chain. Furthermore, for every real root $\alpha$ in the tube of rank $t$ there exists an exceptional root $m_l(0)$ such that $\alpha=r\delta+m_l(0)$. Under the convention that $F_{\alpha}=0$ if $\alpha\in\Z Q_0$ has at least one negative component and setting  $m_{t}(0):=\delta$, we obtain the second main result of this paper, see Theorems \ref{Fdef2}, \ref{Fpolyrank2}, \ref{Fpolyrankn2} and \ref{Fpolydef-1}:
\begin{thmC}
\begin{enumerate}
\item For the representations $M_{m_l(r)}$, where $l=0,\ldots, t-1$ (lying in the exceptional tube of rank $t$) and $r\geq 1$, we have
\[F_{m_{l}(r)}=F_{m_l(0)}F_{r\delta}+x^{m_{l+1}(0)}F_{m_{t-1}(0)-m_{l+1}(0)}F_{(r-1)\delta}.\]

\item Let $M$ be preprojective of defect $-1$ such that $t_M=\udim M-r\delta\leq\delta$. If $\delta-t_M$ is injective, we have
\[F_M=F_{t_M}F_{r\delta}-x^{\delta}F_{(r-1)\delta}.\]
 If $\delta-t_M$ is not injective, we have
\[F_M=F_{t_M}F_{r\delta}-x^{\tau^{-1}t_M}F_{\delta-\tau^{-1}t_M}F_{(r-1)\delta}.\]
Here $\tau$ is the Auslander-Reiten-translation.
\item 
Let $B$ be an indecomposable representation of defect $-2$ which is not projective. Then there exist indecomposable representations $M$ and $N$ of defect $-1$ such that
\[F_B=F_NF_M-x^{\udim\tau^{-1}M}F_{N/\tau^{-1}M}.\]
\item Passing to the dual, we obtain analogous formulae for indecomposable representations of positive defect.
\end{enumerate}
\end{thmC}
\subsection*{Acknowledgements}\noindent The authors would like to thank Giovanni Cerulli Irelli, Christof Gei{\ss}, Mar\-kus Reineke and Jan Schr\"oer for interesting discussions on the topic of this paper and for several helpful remarks.


\section{Schubert decomposition of quiver Grassmannians}\label{sec2}\label{sectionRecollections}
\noindent  One main result of this section is that every quiver Grassmannian of an indecomposable representation of a quiver of type $\tilde D_n$ of large defect has a cell decomposition into affine spaces, see section \ref{EC}. With similar methods, we can show that this is also the case for indecomposable representations of the homogeneous tubes, see section \ref{sechom}. This can be used later to obtain formulae for the $F$-polynomials. 

In order to prove this, we first introduce notation in sections \ref{quiverreps} and \ref{coeffquiver}. In section \ref{cluster}, we recall some results from the theory of cluster algebras which are linked to our considerations. In sections \ref{recollmain} and \ref{qgrer}, we state lemmas that are important for the proof of the main results.

Finally, we show in section \ref{posdef} how the results can be used to pass from representations of negative defect to representations of positive defect (resp. from preprojectives to preinjectives).

\subsection{Quiver representations}\label{quiverreps}
 
\noindent We fix $k=\C$ as our ground field. This suffices for the application of the results to cluster algebras. Actually, all results concerning the representation theory of quivers and quiver Grassmannians of representations remain true when passing to any algebraically closed field $k$. 

We shortly review some basics on quiver representations, see \cite{ass} and \cite{Crawley-Boevey92} for more details. Let $Q=(Q_0,Q_1)$ be a quiver with vertices $Q_0$ and arrows $Q_1$ denoted by $p\xlongrightarrow{v}q$ or $v:p\to q$ for $p,q\in Q_0$. We assume that $Q$ has no oriented cycles. In most parts of this paper, we consider quivers $Q$ of extended Dynkin type $\tilde D_n$, i.e. the underlying graph of $Q$ is \[
   \beginpgfgraphicnamed{tikz/fig58}
   \begin{tikzpicture}[>=latex]
  \matrix (m) [matrix of math nodes, row sep=-0.2em, column sep=2.5em, text height=1ex, text depth=0ex]
   {      & q_b &     &     &      &         &         &  q_c & \\
          &     & q_0 & q_1 & {\dotsb } & q_{n-5} & q_{n-4} &      & \\
      q_a &     &     &     &      &         &         &      &  q_d \\};
   \path[-,font=\scriptsize]
   (m-3-1) edge node[auto,swap] {$a$} (m-2-3)
   (m-1-2) edge node[auto] {$b$} (m-2-3)
   (m-2-3) edge node[auto] {$v_0$} (m-2-4)
   (m-2-6) edge node[auto] {$v_{n-5}$} (m-2-7)
   (m-2-7) edge node[auto] {$c$} (m-1-8)
   (m-2-7) edge node[auto,swap] {$d$} (m-3-9)
   (m-2-4) edge node[auto] {} (m-2-5)
   (m-2-5) edge node[auto] {} (m-2-6);
  \end{tikzpicture}
\endpgfgraphicnamed
\]

 A vertex $p\in Q_0$ is called sink if there does not exist an arrow $\arst\in Q_1$. A vertex $q\in Q_0$ is called source if there does not exist an arrow $\arst\in Q_1$. For an arrow $\arst$, let $s(v)=p$ and $t(v)=q$. We denote by $a(p,q)$ the number of arrows from $p$ to $q$.
For a vertex $p\in Q_0$, let 
\[N_p:=\{q\in Q_0\mid\exists \,p\xlongrightarrow{v}q\in Q_1\vee\exists\, q\xlongrightarrow{v}p\in Q_1\}\]
be the set of neighbors of $p$. Consider the abelian group
$\mathbb{Z}Q_0=\bigoplus_{q\in Q_0}\mathbb{Z}q$ and its monoid of dimension vectors $\mathbb{N}Q_0$. A finite-dimensional complex representation $M$ of $Q$ is given by a tuple
\[M=((M_q)_{q\in Q_0},(M_{v}:M_{s(v)}\rightarrow M_{t(v)})_{v\in Q_1})\]
of finite-dimensional  complex vector spaces and $\C$-linear maps between them. 

Let $\Rep(Q)$ denote the category of finite-dimensional representations of $Q$. The dimension vector $\udim M\in\mathbb{N}Q_0$ of $M$ is defined by
$\udim M=\sum_{q\in Q_0}\dim_kM_qq.$ Let $R_{\alpha}(Q)$ denote the affine space of representations of dimension $\alpha$. Moreover, we denote by $Q^\op$ the quiver obtained from $Q$ when turning around all arrows. Taking dual vector spaces and adjoint linear maps for each arrow, we obtain the dual representation $M^{\ast}$ of $Q^{\op}$ for every representation $M$ of $Q$.

On $\Z Q_0$ we have a non-symmetric bilinear form, the Euler form,
which is defined by
\[\Sc{\alpha}{\beta}=\sum_{q\in Q_0}\alpha_q\beta_q-\sum_{v \in Q_1}\alpha_{s(v)}\beta_{t(v)}\]
for $\alpha,\,\beta\in\Z Q_0$. 
Recall that for two representations $M$, $N$ of $Q$ we have
\begin{align}\label{HomExt}\Sc{\udim  M}{\udim  N}=\dim_k\Hom(M,N)-\dim_k\Ext(M,N)\end{align}
and $\Ext^i(M,N)=0$ for $i\geq 2$. For two representation $M$ and $N$, define $[M,N]:=\dim\Hom(M,N)$. As usual let $M^{\perp}=\{N\in\Rep(Q)\mid \Hom(M,N)=\Ext(M,N)=0\}$.

A dimension vector $\alpha$ is called a root if there exists an indecomposable representation of this dimension. It is called Schur root if there exists a representation with trivial endomorphism ring with this root as dimension vector. A representation $M$ with $\alpha=\udim M$ is called exceptional if we have $\Ext(M,M)=0$. In the case of real roots, i.e. if $\Sc{\alpha}{\alpha}=1$, there only exists one indecomposable representation $M$ up to isomorphism having $\alpha$ as dimension vector. We denote this representation by $M_{\alpha}$.
We denote by $S_q$ the simple representation corresponding to the vertex $q$ and by $s_q$ its dimension vector.

If $Q$ is of extended Dynkin type, we denote by $\delta$ the unique imaginary Schur root which is actually independent of the orientation. Following \cite[section 7]{Crawley-Boevey92}, the defect of a representation $M$ is defined as $\delta(M):=\Sc{\delta}{\udim  M}$. Clearly the defect is additive on dimension vectors. For indecomposables of quivers of extended Dynkin type $D$, we have $|\delta(M)|\leq 2$. We say that an indecomposable representation $M$ has small defect if $|\delta(M)|\leq 1$ and large defect if $|\delta(M)|=2$. As already mentioned, we exclude the representations from the homogeneous tubes when referring to representations of small defect.

\subsection{Quiver Grassmannians, Cluster algebras and $F$-polynomials}\label{cluster}
For a representation $M$ with $m=\udim M$, the quiver Grassmannian $\Gr_e(M)$ is the set of subrepresentations $U$ of $M$ with $\udim U=e$. It is a closed subvariety of the product $\prod_{q\in Q_0}\Gr(e_q,m_q)$ of the usual Grassmannians $\Gr(e_q,m_q)$. 

Let $\Q[x_q^{\pm 1}\mid q\in Q_0]$ be the $\Q$-algebra of Laurent polynomials in the variables $x_q$ for $q\in Q_0$. Denoting by $\chi$ the Euler characteristic in singular cohomology, as in \cite{cc}, we set
\[X_M=\sum_{e\in\N Q_0}\chi(\Gr_e(M))\prod_{q\in Q_0}x_q^{-\Sc{e}{s_q}-\Sc{s_q}{m-e}}.\]

With $Q$ we can associate a cluster algebra $\mathcal A(Q)$, which were introduced by \cite{Fomin-Zelevinsky02}, and its cluster category $\mathcal C_Q$ introduced in \cite{bmrrt}. We cite \cite[Theorem 4]{calkel2}:

\begin{thm}
The correspondence $M\mapsto X_M$ provides a bijection between the set of indecomposable objects of $\mathcal C_Q$ without self-extensions and the set of cluster-variables of $\mathcal A (Q)$.
\end{thm}

Actually, this bijection restricts to a bijection between indecomposable exceptional representations of $Q$ and cluster variables of $\mathcal A(Q)$ excluding the initial variables. In \cite[Theorem 2]{calkel}, which generalizes \cite[Proposition 3.10]{cc}, the following multiplication formula is shown:

\begin{thm}\label{multform}
Let $M$ and $N$ be indecomposable objects of $\mathcal C_Q$ such that $\dim\Ext_{\mathcal C_Q}(M,N)=1$. Then we have
\[X_MX_N=X_B+X_{B'}\]
where $B$ and $B'$ are up to isomorphism the unique middle terms of the non-split triangles
\[N\to B\to M\to SN,\quad M\to B'\to N\to SM.\]
\end{thm}

Note that we have 
\[\dim\Ext_{\mathcal C_Q}(M,N)=\dim\Ext(M,N)+\dim\Ext(N,M),\]
see \cite{bmrrt}. Moreover, if $\Ext(M,N)=k$, the middle term $B$ is the one induced by the non-splitting sequence in the module category. But since $\Ext(N,M)=0$ in this case, using the terminology of \cite{calkel2}, the middle term $B'$ is just an object of $\mathcal C_Q$. But it actually has a corresponding representation in the module category which can be determined explicitly.

In this paper, we mostly consider the generating function $F_M$ of the Euler characteristics of the corresponding quiver Grassmannians of $M$, also called $F$-polynomial, i.e.
\[F_M(x)=\sum_{e\in\N Q_0}\chi(\Gr_e(X))x^e\]
where $x^{e}=\prod_{q\in Q_0}x_q^{e_q}$ for $e\in\N Q_0$, see also \cite{dwz2}. It is closely related to the cluster variables $X_M$. Indeed, setting $$m'_q=\sum_{p\in Q_0}a(p,q)m_p-m_q$$ and considering the variable transformation $x_q\mapsto x_q'$ with 
$$x_q'=\prod_{p\in Q_0}x_p^{a(q,p)-a(p,q)},$$ it is straightforward to check that we have
\[X_M=x^{m'}F_M(x').\]

\subsection{Coefficient quivers and Schubert decomposition}\label{coeffquiver}
We introduce coefficient quivers and tree modules following the presentation given in \cite{rin1}. Let $Q$ be a quiver, $\alpha\in\N Q_0$ a dimension vector and $M$ with $\udim  M=\alpha$ a representation of $Q$. A basis of $M$ is a subset $\mathcal{B}$ of $\bigoplus_{q\in Q_0}M_q$ such that
\[\mathcal{B}_q:=\mathcal{B}\cap M_q\] is a basis of $M_q$ for all vertices $q\in Q_0$. For every arrow $\arst$, we may write $M_{v}$ as a $(\alpha_q\times \alpha_p)$-matrix $M_{v,\mathcal{B}}$ with coefficients in $k$ such that the rows and columns are indexed by $\mathcal{B}_q$ and $\mathcal{B}_p$ respectively. If
\[M_{v}(b)=\sum_{b'\in\mathcal{B}_q}\lambda_{b',b}b'\]
with $\lambda_{b',b}\in k$ and $b\in\mathcal B_p$, we obviously have $(M_{v,\mathcal{B}})_{b',b}=\lambda_{b',b}$.

\begin{df}
The coefficient quiver $\Gamma(M,\mathcal{B})$ of a representation $M$ with a fixed basis $\cB$ has vertex set $\mathcal{B}$ and arrows between vertices are defined by the condition: if $(M_{v,\mathcal{B}})_{b',b}\neq 0$, there exists an arrow $(v,b,b'):b\rightarrow b'$.
If $\mathcal B$ is ordered linearly, we say that $\Gamma(M,\mathcal B)$ is an ordered coefficient quiver.

A representation $M$ is called a tree module if there exists a basis $\mathcal{B}$ for $M$ such that the corresponding coefficient quiver is a tree.
\end{df}
Note that we obtain a natural map $F:\Gamma(M,\cB)\to Q$. In order to shorten notation, we sometimes denote an arrow $(v,b,b')$ by $v$ where $\arst$ is the corresponding  arrow of the original quiver. 
 
We shortly recall the notion of Schubert decompositions of quiver Grassmannians, see \cite{L14}. Let $M$ be a representation with ordered basis $\mathcal B$ and corresponding coefficient quiver $\Gamma_M$. By $\mu_{v,s,t}$, where $v:p\to q$ is an arrow in $Q_1$ and $s\in F^{-1}(p)$ and $t\in F^{-1}(q)$, we denote the corresponding matrix coefficient corresponding to the linear map $M_v$. The induced Schubert decomposition of the usual Grassmannian induces a Schubert decomposition of the corresponding quiver Grassmannians
\[\Gr_e(M)= \coprod_{\substack{\beta\subset\cB\\\text{of type }e}} \ C_\beta^M\]
where the affine varieties $C_\beta^M$ are obtained as subset of the matrix space $\mathrm{Mat}_{\cB\times\cB}$ with variables $w_{ij}$ for $i,j\in\cB$. Let $V(M,\cB)$ be the vanishing set of the polynomials
\[         \label{key formula}
 E(v,t,s) \quad = \quad \sum_{(v,s',t')\in \Gamma_1} \mu_{v,s',t'}  w_{t,t'} w_{s',s} \quad - \quad \sum_{(v,s',t)\in \Gamma_1}   \mu_{v,s',t} w_{s',s} 
\]
for all arrows $v:p\to q$ in $Q_1$ and all vertices $s\in F^{-1}(p)$ and $t\in F^{-1}(q)$. 
For a subset $\beta$ of $\cB$, a matrix $w\in\mathrm{Mat}_{\cB\times\cB}$ is in \emph{$\beta$-normal form}, if it satisfies 
\begin{enumerate}
 \item[(NF1)]\label{NF1} $w_{i,i}=1$ for all $i\in\beta$,
 \item[(NF2)]\label{NF2} $w_{i,j}=0$ for all $i,j\in\beta$ with $j\neq i$,
 \item[(NF3)]\label{NF3} $w_{i,j}=0$ for all $i\in\cB$ and $j\in\beta$ with $j<i$,
 \item[(NF4)]\label{NF4} $w_{i,j}=0$ for all $i\in\cB$ and $j\in\cB-\beta$, and
 \item[(NF5)]\label{NF5} $w_{i,j}=0$ for all $i\in\cB_p$ and $j\in\beta_q$ with $p\neq q$.
\end{enumerate}
Now the Schubert cell $C_\beta^M$ is the intersection of $V(M,\beta)$ with the solution set of (NF1)-(NF5). An arrow $(v,s,t)$ of $\Gamma_M$ is extremal if for all arrows $(v,s',t')\in(\Gamma_M)_1$ either $s<s'$ or $t'<t$. A subset $\beta$ of $\cB=\Gamma_0$ is extremal successor closed if for all extremal arrows $(v,s,t)\in(\Gamma_M)_1$, $s\in\beta$ implies $t\in\beta$. 

We can restrict to the reduced Schubert system if $\beta$ is extremal successor closed which is a necessary condition for the Schubert cell for being not empty, see \cite[section 2.3]{LW}. This means that we find the Schubert cell as vanishing set of the equations $\overline{E}(v,t,s)$ where $t\notin\beta$ and $s\in\beta$ :
 \begin{equation*}\label{eq: reduced form}
        \overline{E}(v,t,s) \quad = \hspace{-0pt} 
        \sum_{\substack{(v,s,t')\in \Gamma_1\\ t<t'}} \hspace{-5pt} \mu_{v,s,t'} w_{t,t'} \ \ + \hspace{-5pt}
        \sum_{\substack{(v,s',t')\in\Gamma_1\\ t<t'\text{ and }s'<s}}  \hspace{-5pt} \mu_{v,s',t'} w_{t,t'} w_{s',s} \ \ - \hspace{-5pt}
        \sum_{\substack{(v,s',t)\in \Gamma_1\\ s'< s}}   \mu_{v,s',t}w_{s',s} \ \ - \ \ 
        \mu_{v,s,t}
       \end{equation*}
       where $\mu_{v,s,t}=0$ if $\Gamma$ does not contain the arrow $(v,s,t)$. These equations are trivial if $t>s$ such that there is no arrow $s\xrightarrow{v} t$. Note that we have $\mu_{v,s,t}\in\{0,1\}$ if $M$ is a tree module.
       
\subsection{Non-contradictory subsets}\label{noncontradictory}


As it is used in this section and in section \ref{secFpoly}, we recall the notion of non-contra\-dictory subsets for those ordered bases of preprojective representations of defect $-1$ which are described in \cite[Appendix B]{LW}. 
In this case, it is possible to simplify the definition; for the general definition we refer to \cite[section 4]{LW}. 

Up to a permutation of the underlying graph $Q$, a preprojective representation $M$ of defect $-1$ has an ordered basis $\cB$ such that the associated coefficient quiver $\Gamma$ has the following shape.
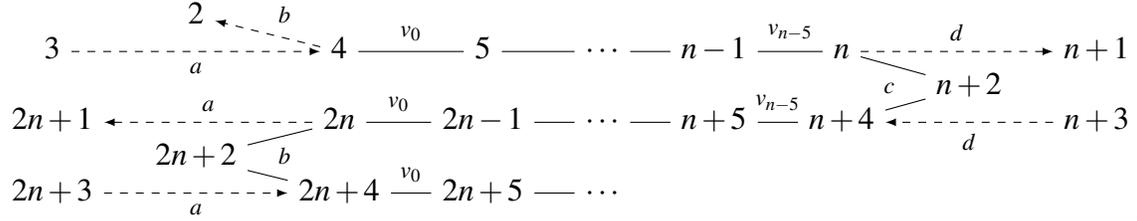
\[
   \begin{tikzpicture}[>=latex]
  \matrix (m) [matrix of math nodes, row sep=0em, column sep=1.3em, text height=1ex, text depth=0ex]
   {        &  2  &     &     &      &     &     &     &     \\   
         3  &     &  4  &  5  &\dotsb& n-1 &  n  &     & n+1 \\   
            &     &     &     &      &     &     & n+2 &     \\   
       2n+1 &     & 2n  & 2n-1&\dotsb& n+5 & n+4 &     & n+3 \\   
            & 2n+2&     &     &      &     &     &     &     \\   
       2n+3 &     & 2n+4&2n+5 &\dotsb                        \\   
};
   \path[-,font=\scriptsize]
   (m-2-3) edge node[auto] {$v_{0}$} (m-2-4)
   (m-2-6) edge node[auto] {$v_{n-5}$} (m-2-7)
   (m-2-7) edge node[auto,swap] {} (m-3-8)
   (m-3-8) edge node[auto,swap] {$c$} (m-4-7)
   (m-4-3) edge node[auto] {$v_0$} (m-4-4)
   (m-4-6) edge node[auto] {$v_{n-5}$} (m-4-7)
   (m-4-3) edge node[auto] {} (m-5-2)
   (m-5-2) edge node[auto] {$b$} (m-6-3)
   (m-6-3) edge node[auto] {$v_0$} (m-6-4)
   ;
   \path[->,dashed,font=\scriptsize]
   (m-2-3) edge node[auto,swap] {$b$} (m-1-2)
   (m-2-1) edge node[auto,swap] {$a$} (m-2-3)
   (m-2-7) edge node[auto] {$d$} (m-2-9)
   (m-4-9) edge node[auto] {$d$} (m-4-7)
   (m-4-3) edge node[auto,swap] {$a$} (m-4-1)
   (m-6-1) edge node[auto,swap] {$a$} (m-6-3)
   ;
   \path[-,font=\scriptsize]
   (m-2-4) edge node[auto] {} (m-2-5)
   (m-2-5) edge node[auto,swap] {} (m-2-6)
   (m-4-4) edge node[auto] {} (m-4-5)
   (m-4-5) edge node[auto] {} (m-4-6)
   (m-6-4) edge node[auto,swap] {} (m-6-5)
   ;
  \end{tikzpicture}
\]
Note that, depending on the orientation of the arrows $a$, $b$, $c$ and $d$, we find one of the following four situations at the ``ramifications'' of $\Gamma$ where $x$ stands for $b$ or $c$ and $y$ stands for $a$ or $d$ and where we label  the vertices with its residue class module $n$ for simplicity.
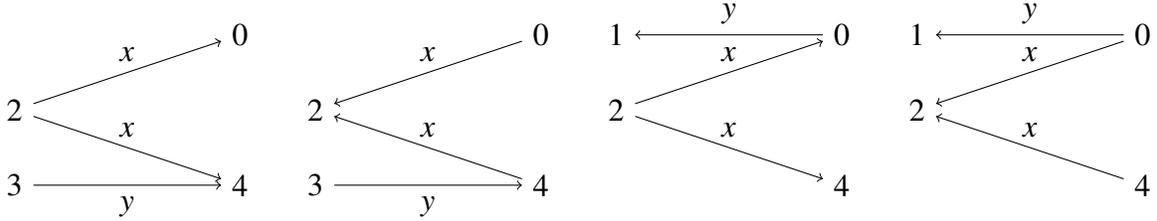
\[
\begin{tikzpicture}
\node (E) at (3,2) {$0$};
\node (F) at (0,1) {$2$};
\node (G) at (0,0) {$3$};
\node (H) at (3,0) {$4$};
\draw[->](F) to node[above]{$x$}(E);
\draw[->](F) to node[above]{$x$}(H);
\draw[->](G) to node[below]{$y$}(H);

\node (A) at (7,2) {$0$};
\node (B) at (4,1) {$2$};
\node (C) at (4,0) {$3$};
\node (D) at (7,0) {$4$};
\draw[<-](B) to node[above]{$x$}(A);
\draw[<-](B) to node[above]{$x$}(D);
\draw[->](C) to node[below]{$y$}(D);

\node (E1) at (11,2) {$0$};
\node (F1) at (8,1) {$2$};
\node (G1) at (8,2) {$1$};
\node (H1) at (11,0) {$4$};
\draw[->](F1) to node[above]{$x$}(E1);
\draw[->](F1) to node[above]{$x$}(H1);
\draw[<-](G1) to node[above]{$y$}(E1);

\node (A1) at (15,2) {$0$};
\node (B1) at (12,1) {$2$};
\node (C1) at (12,2) {$1$};
\node (D1) at (15,0) {$4$};
\draw[<-](B1) to node[above]{$x$}(A1);
\draw[<-](B1) to node[above]{$x$}(D1);
\draw[<-](C1) to node[above]{$y$}(A1);
\end{tikzpicture}
\]
A subset $\beta$ of $\cB=\Gamma_0$ is called \emph{non-contradictory} if $\beta$ is extremal successor closed and if the following conditions are satisfied, depending on the orientations of $x$ and $y$ in the above illustrations.
\begin{align*}
 \xrightarrow{x}, \xrightarrow{y}:& \quad \beta\cap\{0,2,3,4\}\neq\{2,3,4\}; \\
 \xleftarrow{x}, \xrightarrow{y}: & \quad \beta\cap\{0,2,3,4\}\neq\{3,4\};   \\
 \xrightarrow{x}, \xleftarrow{y}: & \quad \beta\cap\{0,1,2,4\}\neq\{2,4\};   \\
 \xleftarrow{x}, \xleftarrow{y}:  & \quad \beta\cap\{0,1,2,4\}\neq\{4\}. 
\end{align*}

Note that a non-contradictory subset $\beta$ is in particular extremal successor closed as defined in section \ref{coeffquiver} and Definition \ref{admsub}. 
\begin{rem}
This notion of non-contradictory subsets transfers to the coefficient quivers of the representations lying in the exceptional tube of rank $n-2$ and, moreover, to the case of the tubes of rank two for $n=4$, see \cite[section 4.3]{LW}. This means that the subsets need to be successor closed and satisfy the same condition at the ramification subgraphs. In particular, this notion suffices to determine the $F$-polynomials of all representations of quivers of $\tilde D_n$.
\end{rem}
\begin{rem}

 The more general definition of non-contradictory $\beta$-states $\Sigma_\beta$ in \cite{LW} is based on the internal logic of Schubert systems. While the above conditions on $\beta$ are always satisfied if $\Sigma_\beta$ is non-contradictory, the reverse conclusion does not hold in general, but it holds in special cases as the one considered above.
\end{rem}

\subsection{Short exact sequences and quiver Grassmannians.}\label{recollmain}
\noindent 
As already mentioned, the first aim of this paper is to prove that every quiver Grassmannian of an indecomposable representation of large defect has a cell decomposition into affine spaces. To do so, we write representations of large defect as the middle term of certain short exact sequences between indecomposables of small defect. Then we can combine \cite[Theorem 4.4]{LW} with the following observations relating the quiver Grassmannians of the middle term to those of the outer terms. In general, given two representations $M$, $N$ and an exact sequence
\[0\to M\xlongrightarrow{i}B\xlongrightarrow{\pi}N\to 0,\]
following \cite[section 3]{cc}, this yields a map 
\[\Psi_e:\Gr_e(B)\to\coprod_{f+g=e}\Gr_f(M)\times\Gr_g(N),\, U\mapsto(i^{-1}(U),\pi(U))\]
whose restrictions to $\Psi_e^{-1}(\Gr_f(M)\times\Gr_g(N))$ are morphisms of algebraic varieties for every $f$ and $g$ with $f+g=e$. Note that we have $i^{-1}(U)\cong U\cap M$ and $\pi(U)\cong (U+M)/M\cong U/U\cap M$.

The following is shown in the course of the proof of \cite[Lemma 3.11]{cc} in the case of almost split sequences. Actually, the proof for almost split sequences applies to arbitrary short exact sequences:

\begin{lemma}\label{afffib}
If $\Psi_e^{-1}(A,V)$ is not empty, we have $\Psi_e^{-1}(A,V)=\mathbb{A}^{[V,M/A]}$.
\end{lemma}

The next step is to restrict the map $\Psi_e$ to the preimage of products of Schubert cells $C^M_{\beta_1}\times C^N_{\beta_2}$ in $C_{\beta_1\cup\beta_2}^B$. This gives a morphism of affine varieties which we denote by $\Psi_{\beta_1,\beta_2}$. If $M$ has ordered basis $\cB_1$ and $N$ has ordered basis $\cB_2$, the middle term $B$ has ordered basis $\cB=\cB_1\cup\cB_2$ where we can assume that $p<q$ for all $p\in\cB_1$ and $q\in\cB_2$. Thus the non-vanishing variables $w_{ij}$ corresponding to the Schubert cell $C_{\beta_1\cup\beta_2}^B$ can be subdivided into three disjoint subsets $\cV_i$ for $i=1,2,3$ where $w_{ij}\in\cV_1$ if $i,j\in\cB_1$, $w_{ij}\in\cV_2$ if $i,j\in\cB_2$ and $w_{ij}\in\cV_3$ if $i\in\cB_1,j\in\cB_2$. 

From now on we assume that $\beta_1$, $\beta_2$ and $\beta_1\cup\beta_2$ are extremal successor closed which simplifies the following considerations. Actually, this is no restriction as we already mentioned that this is a necessary condition to the Schubert cell to be not empty. 

Also the set of non-trivial reduced equations $\overline{E}(v,t,s)$ can be subdivided, i.e. we have $\cE=\cE_1\cup \cE_2\cup\cE_3$ with $\overline{E}(v,t,s)\in\cE_1$ if $s,t\in\cB_1$, $\overline{E}(v,t,s)\in\cE_2$ if $s,t\in\cB_2$ and $\overline{E}(v,t,s)\in\cE_3$ if $s\in\cB_2,t\in\cB_1$. Note that the set of equations $\cE_1$ defines $C_{\beta_1}^M$ and the set $\cE_2$ defines $C_{\beta_2}^N$. Moreover, all variables appearing in $\cE_i$ are in $\cV_i$ for $i=1,2$. Finally, it is straightforward to see that, for fixed variables in $\cV_i$ for $i=1,2$, the equations $\overline{E}(v,t,s)\in\cE_3$ are linear in the variables from $\cV_3$. Indeed, we have $t<s$ with $s\in\cB_2$ and $t\in\cB_1$ if $\overline{E}(v,t,s)$ is non-trivial. The term $\mu_{v,s',t'}w_{t,t'}w_{s',s}$ can only be non-trivial if $t<t'$ and $s'<s$. Thus $w_{t,t'},w_{s',s}\in\cV_3$ would imply $t'\in\cB_2$ and $s'\in\cB_1$ which means $s'<t'$. But there is no arrow from a basis element of $s'\in\cB_1$ to a basis element $t'\in\cB_2$ in the coefficient quiver of $B$. Thus in terms of the variables $\cV$ the morphism $\Psi_{\beta_1,\beta_2}$ is given by setting the variables $w_{ij}\in\cV_3$ to zero.

\begin{prop}\label{bundle}
Fix $\beta_i\subset\cB_i$ for $i=1,2$ and let $\beta:=\beta_1\cup \beta_2$. Assume that the Schubert cells $C_{\beta_1}^M$ and $C_{\beta_2}^N$ are affine spaces and let $m:=\dim C_{\beta_1}^M+\dim C_{\beta_2}^N$. Consider 
\[\Psi_{\beta_1,\beta_2}:C_\beta^B\to C_{\beta_1}^M\times C_{\beta_2}^N.\]
If the fibres of $\Psi_{\beta_1,\beta_2}$ are affine spaces of constant dimension $n$ for some $n\geq 0$, we have $C_\beta^B\cong\A^{n+m}$.
\end{prop}
\begin{proof}
As a first step, we observe that the fibres $\Psi_{\beta_1,\beta_2}^{-1}(p)$ and $C_\beta^B$ are reduced schemes for the following reason. The Schubert cell $C_\beta^B$ is a subspace of a large matrix space $\mathrm{Mat}_{\cB\times\cB}$ and equals the intersection of the affine space $C_{\beta_1}^M\times C_{\beta_2}^N\times\mathrm{Mat}_{\cB_1\times\cB_2}$ with the hypersurfaces defined by the equations $\overline{E}(v,t,s)$ in $\cE_3$. Since these equations are linear in the variables $w_{i,j}$ in $\cV_3$, it is clear that each fibre of $\Psi_{\beta_1,\beta_2}$ is an affine space and henceforth reduced.

The equations $\overline{E}(v,t,s)$ in $\cE_3$ are also linear in the variables in $\cV_1\cup\cV_2$, which means that the fibres of $\Psi_{\beta_1,\beta_2}$ can be seen as the solutions to a system of affine linear equations whose coefficients are linear in the variables of $\cV_1\cup\cV_2$. Therefore there is an open subset $U$ in $C_{\beta_1}^M\times C_{\beta_2}^N$ such that at each point $x$ of $U$ the rank of this system of affine linear equations is maximal. 

Thanks to the Gauss algorithm, the solution space at a point $x$ in $U$ can be parametrized by a bijective affine linear map from an affine space that is rational in the coefficients of the affine linear equations $\overline{E}(v,t,s)$, i.e.\ those variables contained in $\cV_1\cup\cV_2$. As a rational function on $C_{\beta_1}^M\times C_{\beta_2}^N$, it is actually defined on an open subset and it specializes to a parametrization of the solution space for all points $x$ in a non-empty open subset $V$ of $U$.



This shows that the restriction of $\Psi_{\beta_1,\beta_2}$ to the inverse image of $V$ is a trivial vector bundle over $V$ and therefore reduced. Since $C_\beta^B$ is a closed subscheme of the ambient affine space of all variables in $\cV_1\cup\cV_2\cup\cV_3$, it must contain the reduced subscheme whose support is the closure of $\Psi_{\beta_1,\beta_2}^{-1}(V)$. By the semi-continuity of the fibre dimension and since all fibres of $\Psi_{\beta_1,\beta_2}$ have the same dimension, we conclude that $C_\beta^B$ is equal to this reduced subscheme.


The map $\Psi_{\beta_1,\beta_2}$ induces homomorphisms on the tangent spaces for each $q\in C_{\beta}^B$ which we denote by $d_q:=d_q\Psi_{\beta_1,\beta_2}$. As the fibres of $\Psi_{\beta_1,\beta_2}$ are affine spaces of dimension $n$, we have $\ker(T_q\Psi_{\beta_1,\beta_2})\cong\A^n$ for every $q\in C_\beta^B$, see \cite[Section II.8]{har}. 

A priori, it is not clear yet that $C_\beta^M$ is irreducible. But $C_\beta^M$ contains an irreducible component $Z$ of dimension $\dim Z=\dim C_\beta^M$ such that $\Psi_{\beta_1,\beta_2}(Z)$ is dense in $C_{\beta_1}^M\times C_{\beta_2}^N$. Therefore we can apply \cite[Theorem 25.3.1]{vak}, which asserts that there exists a dense open subset $U\subset Z$ such that $\Psi_{\beta_1,\beta_2}\mid_U$ is smooth of dimension $\dim Z-\dim C_{\beta_1}^M\times C_{\beta_2}^N$. Thus there exists a dense open subset $U\subset Z$ such that $d_q$ is surjective for all $q\in U$. This yields 
\[\dim Z = \dim T_q Z=\dim T_{d_{\Psi_{\beta_1,\beta_2}(q)}} C_{\beta_1}^M\times C_{\beta_2}^N+n\]
for all $q\in U$ which means that $U$ is contained in the smooth locus of $Z$. Thus for an arbitrary $q\in Z$, we have $$\dim Z\leq \dim T_qZ=\dim\bigl(\mathrm{im}(d_q)\bigr)+n\leq \dim T_{d_{\Psi_{\beta_1,\beta_2}(q)}} C_{\beta_1}^M\times C_{\beta_2}^N+n=\dim Z.$$
This shows that $Z$ is smooth and that $d_q$ is surjective for all $q\in Z$. Since all fibres of $\Psi_{\beta_1,\beta_2}$ are affine spaces of dimension $n$, this implies that every fibre is contained in $Z$. We conclude that $Z=C_\beta^M$.

Thus $\Psi_{\beta_1,\beta_2}$ is a morphism between smooth complex varieties whose fibres are affine spaces of constant dimension. As the induced maps on the tangent spaces are all surjective, it follows that $\Psi_{\beta_1,\beta_2}$ is smooth of relative dimension $n$, see \cite[Proposition III.10.4]{har}.
Thus we can cover $C_{\beta_1}^M\times C_{\beta_2}^N$ with open affine $W_i$ such that the following diagrams commute
\[\begin{xy}\xymatrix@R20pt@C35pt{C_\beta^B\ar[d]&\Psi_{\beta_1,\beta_2}^{-1}(W_i)\ar[r]^{\pi_i}\ar[d]\ar[l]&\A^{n}_{W_i}\ar[ld]\\C_{\beta_1}^M\times C_{\beta_2}^N&W_i\ar[l]}\end{xy}\]
for all $i$, where $\pi_i$ is \'etale. Since we know that the fibres of $\Psi_e$ are affine spaces of dimension $n$, each fibre of $\pi$ is a point. By the inverse function theorem for \'etale morphisms, $\pi$ is an \'etale locally trivial fibre bundle whose fibre is a point. In other words, $\pi$ is an \'etale vector bundle of rank $0$, and therefore by Serre's theorem (see \cite[section 4]{ser}) a Zariski vector bundle of rank $0$. We conclude that $\pi$ is an isomorphism.

This shows that $\Psi_{\beta_1,\beta_2}$ is a locally trivial $\A^n$-bundle.  A result of \cite{bcw} shows that every affine bundle over an affine space is already a vector bundle. Thus our claim follows by the Quillen-Suslin-Theorem \cite{qui,sus} as every vector bundle over an affine space is trivial.
\end{proof}

\subsection{Quiver Grassmannians of exceptional regular representations}\label{qgrer}
For the remaining part of this section, $Q$ is assumed to be of extended Dynkin type $\tilde D_n$. In order to prove the main result of section \ref{sec2}, we need some properties concerning  the quiver Grassmannians of exceptional regular representations.  More detailed, we need that the cell decomposition of \cite[Theorem 4.4]{LW} is compatible with the decomposition of subrepresentations into direct sums of regular and preprojective representations. To do so, we consider the coefficient quivers of the exceptional regular representations lying in the tubes of rank $2$ and $n-2$ respectively treated in \cite[Appendix B]{LW}. 
\begin{prop}\label{comptube}
Let $M$ be an exceptional regular representation of $\tilde D_n$ and let $C_\beta^M$ be a Schubert cell. If there exists $U\in C_{\beta}^M$ such that $U\cong R\oplus T$ with $R$ regular and $T$ preprojective, we have $U\cong R\oplus T'$ with $T'$ is preprojective for all $U\in C_{\beta}^M$. 
\end{prop}
\begin{proof}First recall the shape of the exceptional tubes listed in \cite[Appendix B]{LW}.
If $M$ lies in a tube of rank $2$, the statement is clearly true because $M$ has no regular subrepresentation. In the tube of rank $n-2$ there exist $n-2$ chains of irreducible inclusions
\[M^j_1\subset M^j_2\subset\ldots\subset M^j_{n-3}\]
of exceptional regular representations. Thus we have $M\cong M^j_i$ for some $i\in\{1,\ldots,n-3\}$ and $j\in\{1,\ldots, n-2\}$. We proceed by induction on $i$. If $i=1$, we have that $M^j_i$ has only preprojective subrepresentations and the claim follows. 

If $e$ is the dimension vector of a regular subrepresentation of $M$, we have that $e=\udim M^j_l$ with $l\leq i$. Moreover, $e$ is an exceptional root and thus by \cite[Corollary 4]{cr} we have
\[\dim\Gr_e(M)=\Sc{e}{\udim M-e}=\dim\Hom(M^j_l,M)-1=0.\]
In particular, there exists a unique subset of the vertex set of the coefficient quiver of $M$ corresponding to $\Gr_e(M)$. Since there exists a short exact sequence
\[\ses{M_l^j}{M}{M_k^{j'}}\]
with $k+l=i$, where $M_0^{j'}:=0$, the construction of the coefficient quivers in \cite[Appendix B]{LW} shows that the coefficient quiver of $M$ is obtained by glueing the one of $M_{k}^{j'}$ to the one of $M_l^{j}$ by an outgoing arrow.

Now assume that $U\cong M_l^j\oplus T$ is a subrepresentation of $M$ such that $T$ is preprojective. Since $\Ext(T,M_l^j)=0$, we obtain a commutative diagram
\[
\begin{xy} 
\xymatrix@R15pt@C20pt{0\ar[r] &M_l^j\ar[r]&\ar[r]M&M_k^{j'}\ar[r]&0\\0\ar[r] &M_l^{j}\ar[r]\ar@{=}[u]&\ar[r]\ar[u]M_l^{j}\oplus T&T\ar[r]\ar[u]\ar[u]^{i_T}&0
}\end{xy}\]
In particular, $T$ lies in a cell $C_{\beta}^{M_{k}^{j'}}$ defined by a subset $\beta$ of the vertex set of the coefficient quiver of $M_k^{j'}$. By induction hypothesis, we have that all representations in the cell $C_{\beta}^{M^{j'}_k}$ of $T$ decompose into preprojective representations. 

Now we have $\Psi_{\udim M_l^j+\udim T}^{-1}(M_l^j,T)=\mathbb A^0$ which shows that the lifted cell $\Psi_{\udim M_l^j+\udim T}^{-1}(C_{\beta}^{M^{j'}_k})$ has the same dimension. Thus every representation in the cell of $U$ decomposes into a direct sum of $M_l^{j}$ and a preprojective representation.
\end{proof}

\subsection{Representations of large defect}\label{EC}
\noindent The main aim of this section is to show that the Schubert decomposition of indecomposable representations of small defect obtained in \cite[Theorem 4.4]{LW} extends to a Schubert decomposition of indecomposable representations of large defect. In section \ref{sechom}, it turns out that similar methods can be applied to show that every  quiver Grassmannian coming along with a representation lying in one of the homogeneous tubes has a cell decomposition into affine spaces.

As already mentioned, we can restrict to the case of preprojective roots. Recall that the preprojectives of defect $-1$ are precisely the Auslander-Reiten translates of the projectives corresponding to the outer vertices, i.e. of $P_{q_a}, P_{q_b}, P_{q_c}$ and $P_{q_d}$. The preprojectives of defect $-2$ are Auslander-Reiten translates of projectives corresponding to the inner vertices, i.e. of $P_{q_0},\ldots, P_{q_{n-4}}$. 

\begin{rem}\label{reminout0}

\begin{enumerate}
\item An indecomposable preprojective representation $M$ has no proper factor $N$ such that $\delta(N)\leq\delta(M)$. This follows because the defect is additive on exact sequences and preprojective representations have only preprojective subrepresentations. 
\item If $M$ is preprojective with $\delta(M)=-1$ and $N$ is preprojective, then every non-zero morphism $f:M\to N$ is injective. Indeed, since $\mathrm{Im}(f)$ is a subrepresentation of $N$, we have $\delta(\mathrm{Im}(f))\leq -1=\delta(M)$. Since $\mathrm{Im}(f)$ is a factor of $M$, the first part yields $\mathrm{Im}(f)=M$.
\item If $M$ and $N$ are preprojective and there is no path from $M$ to $N$ in the Auslander-Reiten quiver we have $\Hom(M,N)=\Ext(N,M)=0$. The first statement is clear, the second follows by the Auslander-Reiten formula $\dim\Ext(N,M)=\dim\Hom(M,\tau N)$ keeping in mind that every morphism between $M$ and $N$ is a finite decomposition of irreducible ones. Indeed, there is a path from $\tau N$ to $N$ and thus no path from $M$ to $\tau N$ by assumption.
\end{enumerate}
\end{rem}
For $\tilde D_n$ in subspace orientation, there exist almost split sequences of the form
\[\ses{\tau^{-(l-1)}P_{p}}{\tau^{-l}P_{q_0}}{\tau^{-l}P_{p}},\quad \ses{\tau^{-(l-1)}P_{p'}}{\tau^{-l}P_{q_{n-4}}}{\tau^{-l}P_{p'}}\]
for $p\in\{q_a,q_b\}$, $p'\in\{q_c,q_d\}$ and $l\geq 1$.
In this case, the initial part of the preprojective component of the Auslander-Reiten quiver looks as follows

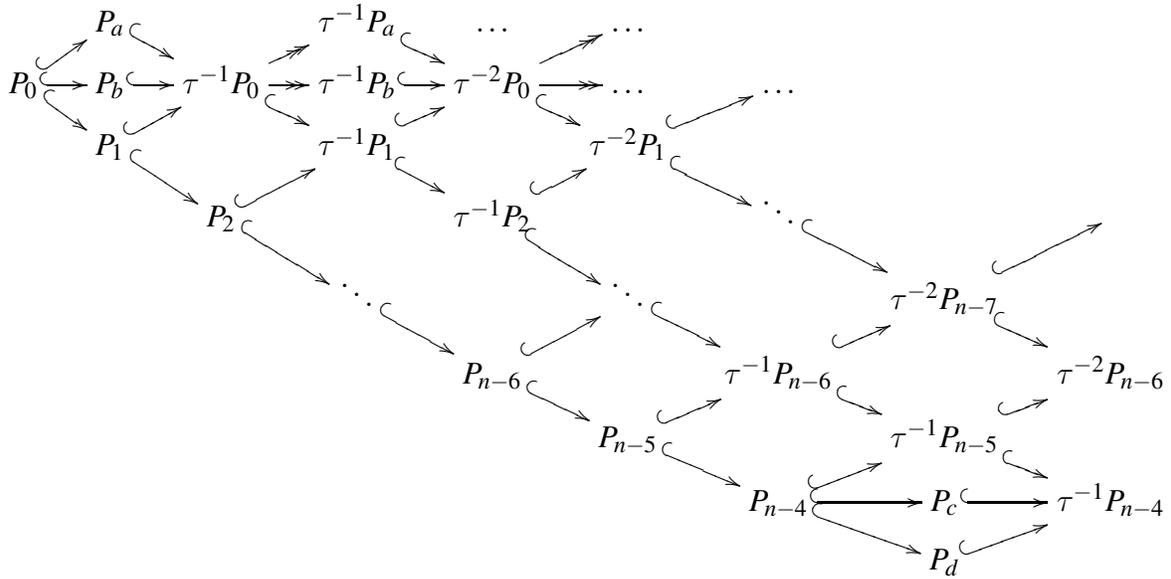
\[
\begin{xy}
\xymatrix@R4pt@C15pt{&P_{a}\ar@{^(->}[rd]&&\tau^{-1}P_{a}\ar@{^(->}[rd]&\dots&\dots&\\P_0\ar@{^(->}[rd]\ar@{^(->}[r]\ar@{^(->}[ru]&P_{b}\ar@{^(->}[r]&\tau^{-1}P_0\ar@{->>}[r]\ar@{->>}[ru]\ar@{^(->}[rd]&\tau^{-1}P_{b}\ar@{^(->}[r]&\tau^{-2}P_0\ar@{^(->}[rd]\ar@{->>}[ru]\ar@{->>}[r]&\dots&\dots\\&P_1\ar@{^(->}[ru]\ar@{^(->}[rd]&&\tau^{-1}P_{1}\ar@{^(->}[ru]\ar@{^(->}[rd]&&\tau^{-2}P_1\ar@{^(->}[rd]\ar@{^(->}[ru]\\&&P_2\ar@{^(->}[ru]\ar@{^(->}[rd]&&\tau^{-1}P_2\ar@{^(->}[rd]\ar@{^(->}[ru]&&\ddots\ar@{^(->}[rd]&&\\&&&\ddots\ar@{^(->}[rd]&&\ddots\ar@{^(->}[rd]&&\tau^{-2}P_{n-7}\ar@{^(->}[rd]\ar@{^(->}[ru]\\&&&&P_{n-6}\ar@{^(->}[ru]\ar@{^(->}[rd]&&\tau^{-1}P_{n-6}\ar@{^(->}[ru]\ar@{^(->}[rd]&&\tau^{-2}P_{n-6}\\&&&&&P_{n-5}\ar@{^(->}[ru]\ar@{^(->}[rd]&&\tau^{-1}P_{n-5}\ar@{^(->}[rd]\ar@{^(->}[ru]&\\&&&&&&P_{n-4}\ar@{^(->}[rd]\ar@{^(->}[ru]\ar@{^(->}[r]&P_{c}\ar@{^(->}[r]&\tau^{-1}P_{n-4}\\&&&&&&&P_{d}\ar@{^(->}[ru]
}
\end{xy}
\]
where we use the abbreviations $P_i:=P_{q_i}$. If $n$ is even, the remaining part of the preprojective component is obtained from this and looks for every orientation as follows:

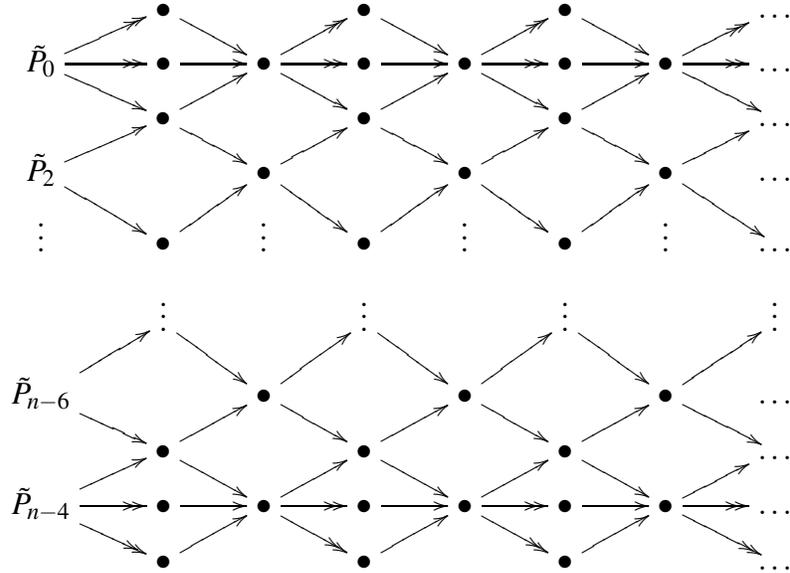
\[
\begin{xy}
\xymatrix@R5pt@C25pt{&\bullet\ar[rd]&&\bullet\ar[rd]&&\bullet\ar[rd]&&\dots\\\tilde P_0\ar@{->>}[r]\ar@{->>}[ru]\ar[rd]&\bullet\ar[r]&\bullet\ar@{->>}[r]\ar@{->>}[ru]\ar[rd]&\bullet\ar[r]&\bullet\ar@{->>}[r]\ar@{->>}[ru]\ar[rd]&\bullet\ar[r]&\bullet\ar@{->>}[r]\ar@{->>}[ru]\ar[rd]&\dots\\&\bullet\ar[rd]\ar[ru]&&\bullet\ar[rd]\ar[ru]&&\bullet\ar[rd]\ar[ru]&&\dots\\\tilde P_2\ar[ru]\ar[rd]&&\bullet\ar[ru]\ar[rd]&&\bullet\ar[ru]\ar[rd]&&\bullet\ar[ru]\ar[rd]&\dots\\\vdots&\bullet\ar[ru]&\vdots&\bullet\ar[ru]&\vdots&\bullet\ar[ru]&\vdots&\dots\\&\vdots\ar[rd]&&\vdots\ar[rd]&&\vdots\ar[rd]&&\vdots\\\tilde P_{n-6}\ar[rd]\ar[ru]&&\bullet\ar[ru]\ar[rd]&&\bullet\ar[ru]\ar[rd]&&\bullet\ar[ru]\ar[rd]&\dots\\&\bullet\ar[rd]\ar[ru]&&\bullet\ar[rd]\ar[ru]&&\bullet\ar[rd]\ar[ru]&&\dots\\\tilde P_{n-4}\ar@{->>}[rd]\ar@{->>}[r]\ar[ru]&\bullet\ar[r]&\bullet\ar@{->>}[r]\ar@{->>}[rd]\ar[ru]&\bullet\ar[r]&\bullet\ar@{->>}[r]\ar@{->>}[rd]\ar[ru]&\bullet\ar[r]&\bullet\ar@{->>}[r]\ar@{->>}[rd]\ar[ru]&\dots\\&\bullet\ar[ru]&&\bullet\ar[ru]&&\bullet\ar[ru]&&\dots
}
\end{xy}
\]
where the usual arrows indicate monomorphisms and the $\tilde P_i$ are Auslander-Reiten translates of $P_i$. In subspace orientation we have 
\[\tilde P_0=\tau^{-\frac{n+3}{2}}P_0,\quad \tilde P_2 =\tau^{-\frac{n+1}{2}}P_2,\quad\ldots\quad \tilde P_{n-6}=\tau^{-2}P_{n-6},\quad\tilde P_{n-4}=\tau^{-1}P_{n-4}.\]

If $n$ is odd, the remaining part of the preprojective component is obtained from this and looks for every orientation as follows:

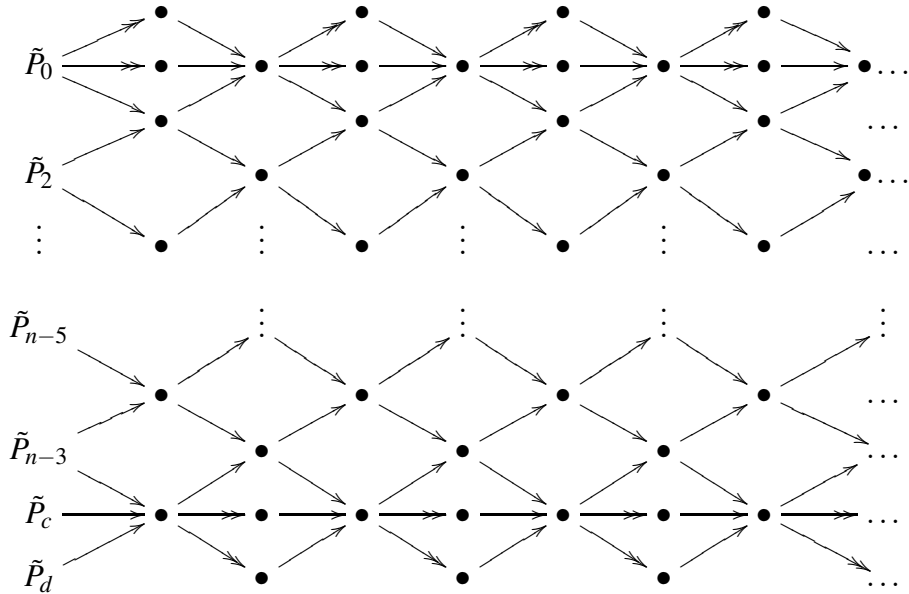
\[
\begin{xy}
\xymatrix@R5pt@C25pt{&\bullet\ar[rd]&&\bullet\ar[rd]&&\bullet\ar[rd]&&\bullet\ar[rd]&\\\tilde P_0\ar@{->>}[r]\ar@{->>}[ru]\ar[rd]&\bullet\ar[r]&\bullet\ar@{->>}[r]\ar@{->>}[ru]\ar[rd]&\bullet\ar[r]&\bullet\ar@{->>}[r]\ar@{->>}[ru]\ar[rd]&\bullet\ar[r]&\bullet\ar@{->>}[r]\ar@{->>}[ru]\ar[rd]&\bullet\ar[r]&\bullet\dots\\&\bullet\ar[rd]\ar[ru]&&\bullet\ar[rd]\ar[ru]&&\bullet\ar[rd]\ar[ru]&&\bullet\ar[ru]\ar[rd]&\dots\\\tilde P_2\ar[ru]\ar[rd]&&\bullet\ar[ru]\ar[rd]&&\bullet\ar[ru]\ar[rd]&&\bullet\ar[ru]\ar[rd]&&\bullet\dots\\\vdots&\bullet\ar[ru]&\vdots&\bullet\ar[ru]&\vdots&\bullet\ar[ru]&\vdots&\bullet\ar[ru]&\dots\\\tilde P_{n-5}\ar[rd]&&\vdots\ar[rd]&&\vdots\ar[rd]&&\vdots\ar[rd]&&\vdots\\&\bullet\ar[rd]\ar[ru]&&\bullet\ar[ru]\ar[rd]&&\bullet\ar[ru]\ar[rd]&&\bullet\ar[ru]\ar[rd]&\dots\\\tilde P_{n-3}\ar[ru]\ar[rd]&&\bullet\ar[rd]\ar[ru]&&\bullet\ar[rd]\ar[ru]&&\bullet\ar[rd]\ar[ru]&&\dots\\\tilde P_c\ar[r]&\bullet\ar@{->>}[r]\ar@{->>}[rd]\ar[ru]&\bullet\ar[r]&\bullet\ar@{->>}[r]\ar@{->>}[rd]\ar[ru]&\bullet\ar[r]&\bullet\ar@{->>}[r]\ar@{->>}[rd]\ar[ru]&\bullet\ar[r]&\bullet\ar@{->>}[r]\ar@{->>}[rd]\ar[ru]&\dots\\\tilde P_d\ar[ru]&&\bullet\ar[ru]&&\bullet\ar[ru]&&\bullet\ar[ru]&&\dots
}
\end{xy}
\]
where the usual arrows indicate monomorphisms and the $\tilde P_i$ are Auslander-Reiten translates of $P_i$. In subspace orientation we have 
\[\tilde P_0=\tau^{-\frac{n+3}{2}}P_0,\quad \tilde P_2 =\tau^{-\frac{n+1}{2}}P_2,\,\,\ldots\quad\tilde P_{n-5}=\tau^{-2}P_{n-5},\quad\tilde P_{n-3}=\tau^{-1}P_{n-3},\quad \tilde P_c=P_c,\quad \tilde P_d=P_d.\]

If $N$ is preprojective with $\delta(N)=-1$,  we denote by $\hat N$ its neighbor in the Auslander-Reiten quiver satisfying $\delta(\hat N)=-1$, which means by definition that they are Auslander-Reiten translates of $P_a$ and $P_b$ (resp. $P_c$ and $P_d$). Note that, for $n\geq 5$, this means that they are direct summands of the middle term of the same Auslander-Reiten sequence.
In subspace orientation, for $\tau^{-s} P_{a}$ we have $$\widehat{\tau^{-s} P_{a}}=\tau^{-s} P_{b}\text{ and }\widehat{\tau^{-s} P_{c}}=\tau^{-s} P_{d}$$
and the corresponding relations when permuting $a,b$ and $c,d$ respectively. 

For a representation $C$ with $\delta(C)=-1$, we define $\rho C:=\widehat{\tau^{-1}C}$ and $\kappa C:=\widehat{\rho C}$. Then we get a chain of irreducible inclusions
\[C\subset \rho C\subset\rho^2C\subset\ldots\subset\rho^i C\subset\ldots.\]

Let $\mathcal C C$ denote the full subcategory of $\Rep(Q)$ which contains the objects $\rho^lC/\rho^k C$ for $l>k\geq 1$ and which is closed under exact sequences and images.
\begin{lemma}\label{verschoben}The following holds:
\begin{enumerate}
\item The category $\mathcal CC$ is equivalent to the full subcategory of $\Rep(Q)$ whose objects are direct sums of representations of the tube of rank $n-2$. 
\item For $1\leq l\leq n-3$ and $m\geq 0$, we have
\[\dim\Ext(\kappa^{m(n-2)+l} C,C)=m+1,\quad\dim\Hom(C,\kappa^{m(n-2)+l} C)=m.\]
\end{enumerate}
\end{lemma}
\begin{proof}For two representations $M,N\in \Rep(Q)$, we have $\Hom(M,N)=\Hom(\tau^{-1}M,\tau^{-1}N)$ and $\Ext(M,N)=\Ext(\tau^{-1}M,\tau^{-1}N)$ by the Auslander-Reiten formulae, see \cite[Theorem IV.2.13]{ass}. Thus we can assume that $C=P_a$ and the first part of the lemma is straightforward.

For the second part,  one observes that $\udim\kappa^{m(n-2)} C=\udim C+m\delta$ for $m\geq 0$. Since we have $\Ext(M,N)=0$ or $\Hom(M,N)=0$ for two preprojective representations and since $\delta(C)=-1$, the claim follows by formula (\ref{HomExt}). 
\end{proof}

Given two representations $M$, $N$ and an exact sequence
\[0\to M\xlongrightarrow{i}B\xlongrightarrow{\pi}N\to 0,\]
we consider the map
\[\Psi_e:\Gr_e(B)\to\coprod_{f+g=e}\Gr_f(M)\times\Gr_g(N),\, U\mapsto(i^{-1}(U),\pi(U))\]
defined in section \ref{recollmain}. If non-empty, the dimension of a fibre depends on the direct sum decomposition of the subrepresentations of $M$ and $N$. In general, it is already difficult to say in which cases the fibre is empty. But in the case of representations of large defect there are sequences which are close to being almost split so that the fibres can be determined in any case. This extends the following result, see \cite[Lemma 3.11]{cc} and \cite[Proposition 2]{cr}.

\begin{thm}\label{almostsplit}
Let $M$ be a representation of $Q$ and $\tau M$ be its Auslander-Reiten translate. Consider the almost split sequence 
\[\ses{\tau M}{B}{ M}.\]
Then we have
\[\Psi_e^{-1}(A,V)=\begin{cases}\emptyset\text{ if } (A,V)=(0,M)\\\mathbb{A}^{[V,\tau M/A]}\text{ otherwise.}\end{cases}\]
In particular, we have \[\chi(\Gr_e(B))=\begin{cases}\sum_{f+g=e}\chi(\Gr_f(M))\chi(\Gr_g(\tau M))\text{ if } e\neq \udim M\\\sum_{f+g=e}\chi(\Gr_f(M))\chi(\Gr_g(\tau M))-1\text{ if } e=\udim M.\end{cases}\]
\end{thm}
In general, a representation of large defect cannot be written as the middle term of an almost split sequence. But we can modify the preceding statement to make it applicable for our purposes. If $B$ is indecomposable preprojective of defect $-2$, which is not projective, in the Auslander-Reiten quiver exists the following subquiver
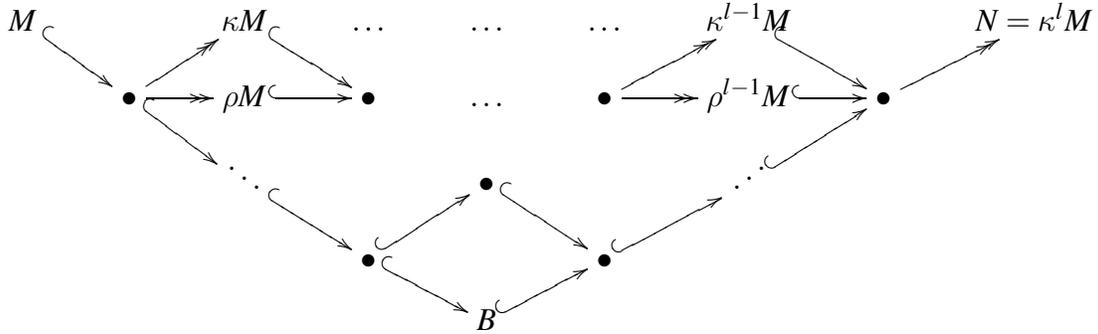
\[
\begin{xy}
\xymatrix@R10pt@C25pt{M\ar@{^(->}[rd]&&\kappa M\ar@{^(->}[rd]&\dots&\dots&\dots&\kappa^{l-1}M\ar@{^(->}[rd]&&N=\kappa^lM\\&\bullet\ar@{->>}[ru]\ar@{^(->}[rd]\ar@{->>}[r]&\rho M\ar@{^(->}[r]&\bullet&\dots&\bullet\ar@{->>}[ru]\ar@{->>}[r]&\rho^{l-1}M\ar@{^(->}[r]&\bullet\ar@{->>}[ru]&\\&&\ddots\ar@{^(->}[rd]&&\bullet\ar@{^(->}[rd]&&\iddots\ar@{^(->}[ru]\\&&&\bullet\ar@{^(->}[ru]\ar@{^(->}[rd]&&\bullet\ar@{^(->}[ru]\\&&&&B\ar@{^(->}[ru]
}
\end{xy}
\] 
Here $M$ and $N$ are two indecomposable preprojective representations of $\tilde D_n$ of defect $-1$. More precisely, we have:
\begin{lemma}\label{def2}
Every indecomposable preprojective representation $B$ with $\delta(B)=-2$ which is not projective is obtained as the middle term of an exact sequence
\[\ses{M}{B}{N}\]
such that $N=\kappa^{l} M\in M^{\perp}$ with $l\leq n-3$, $\Ext(N,M)=k$ and $\Hom(N,M)=0$. 
\end{lemma}

\begin{proof}
The representation $B$ is the $n^\mathrm{th}$ Auslander-Reiten translate of a projective representation corresponding to an inner vertex $q_i$ where $n\geq 1$. As $B$ is not projective, it has a subrepresentation $M$ which is the $m^{\mathrm{th}}$ Auslander-Reiten translate of a projective representation of defect $-1$ where $n-1\leq m\leq n$. This can be seen when considering the preprojective component of the Auslander-Reiten quiver. By applying Auslander-Reiten translations, we can assume that $M=P_{a}$. The claim follows by Lemma \ref{verschoben} together with the Auslander-Reiten-formulae.
\end{proof}

Note that, if $B$ is projective, it might happen that such a triangle does not exist. This is for instance the case if $\tilde D_n$ is in subspace orientation. But as our quiver is acyclic, there is at most one path between each two vertices of the quiver which means that for the dimension vectors $\alpha$ of projective representations, we have $\alpha_q\in\{0,1\}$ for all vertices $q$. In turn, the quiver Grassmannians trivially have cell decompositions.

 Actually, we can assume that $N=\kappa^{l}M$ with $l\leq n/2-1$. This is because the two lower rows of the Auslander-Reiten quiver behave dual to the upper ones. In particular, if $n=4,5$ we can assume that $N=\tau^{-1}M$ and, moreover, we are in the situation of Theorem \ref{almostsplit}. 

From now on, we assume that $B$ is not projective. In the following, we refer to the indecomposable representations lying properly in the above triangle or corresponding to a point $T\neq B$ on the path from $B$ to $N$ as $(M,N)$-inner representation. The remaining ones, i.e. those which are outside the triangle or on the path from $M$ to $B$, as $(M,N)$-outer representations. We drop $(M,N)$ if it is clear which representations are considered. Finally, if a $(M,N)$-inner (resp. $(M,N)$-outer) representation is also a subrepresentation of $N$, we call it inner (resp. outer) subrepresentation of $N$ if we fixed a triangle.
In order to investigate such a triangle, the Auslander-Reiten formulae assure that we can mostly without loss of generality assume that $M=P_{a}$ and $N=\kappa^lM$.

%

In order to prove the following essential lemma, we use the notation of Lemma \ref{def2} and use some well-known facts concerning the Auslander-Reiten theory of extended Dynkin quivers, see for instance \cite[sections IV, VIII.2]{ass}:

\begin{lemma}\label{inout}The following holds:
\begin{enumerate}
\item\label{1} For every inner representation $C$, we have $\Hom(C,B)=\Hom(C,M)=0$.
\item\label{2} The representation $N$ has no subrepresentation which is isomorphic to a representation which lies on the border of the triangle. 
\item\label{3}
If $Z$ is an outer representation, an injective morphism $f\in\Hom(Z,N)$ with $f\neq 0$ factors through $B$. 
\item\label{4} The inner subrepresentations $C$ of $N$ are precisely the representations $C=\kappa^i M$ for $i=1,\ldots,l$. In particular, $N$ has no inner subrepresentation of defect $-2$.
\item\label{5} If $A$ is a non-zero subrepresentation of $M$ and $V$ is any subrepresentation of $N$ we have $\Ext(V,M/A)=0$.

\end{enumerate}

\end{lemma}
\begin{proof} The first part follows because there is no path from an inner representation $C$ to $B$ and $M$ respectively, see also part three of Remark \ref{reminout0}.

We get (\ref{2}) as follows: for each representation $C\neq N$ which lies on the border, there exists an injection $M\hookrightarrow C$. Now we can use that $\Hom(M,N)=0$. 

To see (\ref{3}), one first observes that every such mono factors through an outer representation $N'$ which lies on the path from $M$ to $B$ and we thus have $M\subset N'\subset B$. Since $\Hom(M,N)=0$, we have $N'\neq M$ and thus $\delta(N')=-2$ . As we get an exact sequence with the same properties as those for $M$ and $B$, which has $N'$ as the middle term and $M$ as the kernel, it is straightforward that $ M\in N'^{\perp}$ which yields $\Hom(N',B)\cong\Hom(N',N)$. Thus if an injection factors through $N'$ it already factors through $B$.

Every representation of defect $-2$ which lies in the triangle has $M$ as a subrepresentation. As $\Hom(M,N)=0$, the representation $N$ has no such subrepresentations. Moreover, there exists the following chain of inner representations 
\[\kappa M\subset\kappa^2M\subset\ldots\subset\kappa^lM=N.\]
Apart from these representations the only representations of defect $-1$ in the triangle are of the form $\rho^i M$ for $i=1,\ldots,l$. But since $M\subset \rho^i M$, they are no subrepresentations of $N$. Thus we get (\ref{4}).

Since $N$ is preprojective, $V$ is also preprojective. By the first part of Remark \ref{reminout0}, it follows that $M/A$ has no preprojective direct summand. Thus we get (\ref{5}).

\end{proof}
We use this lemma to prove the following proposition classifying subrepresentations of $N$.
\begin{prop}\label{inout2}The following holds:
\begin{enumerate}

\item\label{6} If $0\neq V\subseteq N$, the corresponding injection either factors through $B$ or $V\cong C\oplus L$ where $C$ is an inner subrepresentation of $N$.
\item\label{7} If $C\oplus L\subseteq N$ such that $C$ is an inner subrepresentation, we have that $L\subseteq N/C$ is preprojective, 
and, moreover, with respect to the Schubert decomposition induced by the bases considered in \cite[Appendix B]{LW}, we have that $$\{V'\in\Gr_{\udim C+\udim L}(N)\mid V'\cong C\oplus L',\,L'\text{ preprojective} \}\cong\{L'\in\Gr_{\udim L}(N/C)\mid L'\text{ preprojective}\}$$ is a union of cells $C_{\beta}^N$  of the cell decomposition of $\Gr_{\udim C+\udim L}(N)$ into affine spaces.
\end{enumerate}
\end{prop}
\begin{proof}

Part (\ref{6}) is just a reformulation of statements of Lemma \ref{inout} taking into account Remark \ref{reminout0}. Note that we have $\Hom(C,B)=0$ and thus $B$ has no subrepresentation isomorphic to $C\oplus L$ where $C$ is an inner subrepresentation.

The second part can be obtained as follows. For a fixed inner subrepresentation $C\subset N$, setting $e:=\udim C+\udim L$, we again obtain a map of quiver Grassmannians $$\Psi_e:\Gr_{e}(N)\to \prod_{ f+g=e}\Gr_{ f}(C)\times \Gr_{g}(N/C).$$ Our task is to analyze the fibres in the case $ f=\udim C$ and the consequences for the Schubert decompositions.

By Lemma \ref{verschoben}, we have that $N/C$ is an exceptional regular representation. In addition, it follows that $\Hom(C,N/C)=\Ext(C,N/C)=0$. Now \cite[Corollary 4]{cr} yields $\Gr_{\udim C}(N)=\{\mathrm{pt}\}$. Moreover, since we have $N=\rho^m C$ for some $m\leq n-4$, by construction of the coefficient quiver $\Gamma_N$ of $N$ in \cite[Appendix B]{LW}, this subrepresentation corresponds to the full subquiver $\Gamma_C$ of $\Gamma_N$ which consists of the first $\udim C$ vertices. Since $C$ is a subrepresentation of $N$ and $N/C$ a factor, the full subquiver $\Gamma_{N/C}$ consisting of the vertices $(\Gamma_N)_0\backslash(\Gamma_C)_0$ is connected to $\Gamma_C$ by an outgoing arrow.

If $L\subset N/C$ has a regular direct summand, we have $\Hom(L,N)=0$. In particular, $L\oplus C$ is no subrepresentation of $N$. Note that, if $L$ is regular indecomposable, we can consider the inclusion $\Hom(L,N/C)\hookrightarrow\Ext(L,C)$. Then we even have that the middle term $C'$ of the corresponding sequence is also an inner representation, see also Remark \ref{inner}.

Thus assume that $L\subseteq N/C$ is preprojective. By Proposition \ref{comptube}, we have that there exist cells $C_{\beta_1}^{N/C},\ldots,C_{\beta_n}^{N/C}$ such that the subsets $\beta_i$ have cardinality $\udim L$ and such that all representations in these cells are preprojective. If $\Ext(L,C)=0$, we obtain a commutative diagram
\[
\begin{xy} 
\xymatrix@R15pt@C20pt{0\ar[r] &C\ar[r]&\ar[r]N&N/C\ar[r]&0\\0\ar[r] &C\ar[r]\ar@{=}[u]&\ar[r]\ar[u]V&L\ar[r]\ar[u]\ar[u]^{i_L}&0
}\end{xy}\] 
with $V\cong C\oplus L$.  In particular, the subrepresentation $L$ lifts to a subrepresentation of $N$. 

We claim that, actually, $\Ext(L,C)=0$ for all preprojective subrepresentations $L\subset N/C$. We can without loss of generality assume that $L$ is indecomposable, $N\in\{\tau^{-l}P_a,\tau^{-l}P_b\}$ and that $C=\tau^{-r}P_a$ where $1\leq r\leq l\leq n-3$. In particular, we have $\delta>\udim N\geq\udim C$. If $L=\kappa^l C$ for some $n-3\geq l\geq 1$, the representation $V$ would be an indecomposable inner representation with $\delta(V)=-2$. Indeed, in this case we have $C\in{^\perp} L$ by Lemma \ref{verschoben}. But since $N$ has no inner subrepresentations of defect $-2$ by Lemma \ref{inout}, this is not possible. Also $l\geq n-2$ is not possible because then we had $\udim L>\delta>\udim N/C$.

We have $\dim \Ext(L,C)=\dim \Hom(C,\tau L)$ by the Auslander-Reiten formulae. If $\Ext(L,C)\neq 0$, since $C$ is of defect $-1$, by the second part of Remark \ref{reminout0}, it follows that $C\subset \tau L$. If $\delta(L)=-2$, there exists a chain of  inclusions
\[\tau L\subset L\subset \tau^{-1}L\subset\ldots\]
In particular, we have $\udim L\geq \udim\tau L\geq\udim C$. But for $C=\tau^{-r}P_a$ we either have $\dim C_{q_a}\geq 1 $ or $\dim C_{q_b}\geq 1$ and thus $\dim V_{q_a}\geq 2> \delta_{q_a}\geq\dim N_{q_a}$ or $\dim V_{q_b}\geq 2> \delta_{q_b}\geq\dim N_{q_b}$. But this is not possible because $V$ is a subrepresentation of $N$ and $\udim N<\delta$.

Thus it remains to deal with the case if $L\in\{\tau^{-l}P_c,\tau^{-l}P_d\}$ for some $l\geq 0$. There exists a $r\geq 1$ such that $C\subset \tau^{-s} P_{n-4}$ and $C\not\subset \tau^{-t} P_{n-4}$ for $s\geq r$ and $r-1\geq t\geq 0$. Moreover, there exists almost-split sequences

\[\ses{\tau^{-k}P_c}{\tau^{-r} P_{n-4}}{\tau^{-(k+1)}P_c},\quad \ses{\tau^{-k'}P_d}{\tau^{-r} P_{n-4}}{\tau^{-(k'+1)}P_d}\]
for some $k,k'\geq 0$. By the choice of $r$, we have $\Hom(C,\tau^{-k}P_c)=\Hom(C,\tau^{-k'}P_d)=0$ because otherwise there were a path of irreducible morphism from $C$ to $\tau^{-k}P_c$ which were forced to factor through $\tau^{-r+1}P_{n-4}$.
Thus, keeping in mind the second part of Remark \ref{reminout0}, it follows that  $C\subset \tau^{-l}P_c,\,\tau^{-l}P_d$ for all $l> k, k'$ and the claim follows as in the case $\delta(L)=-2$.

Since we have $\Psi_{e}^{-1}(C,L)=\mathbb{A}^0$ if the fibre is not empty, as a consequence, we obtain an isomorphism 
$$\{V'\in\Gr_{\udim C+\udim L}(N)\mid V'\cong C\oplus L',\,L'\text{ preprojective} \}\cong\{L'\in\Gr_{\udim L}(N/C)\mid L'\text{ preprojective}\}.$$
As the right hand side is compatible with the Schubert decomposition by Proposition \ref{comptube}, this restricts to an isomorphism between the respective Schubert cells.  
\end{proof}
\begin{rem}\label{inner}
By the results of this section, it follows that every subrepresentation $C\oplus L\subset N$ such that the fibre of $(0,C\oplus L)$ is empty are in bijection with the subrepresentations of $N/\tau^{-1}M$. To check this it suffices to keep in mind that every inner representation is obtained as the middle term of an exact sequence between a regular subrepresentation of $N/\tau^{-1}M$ and $\tau^{-1}M$. 
\end{rem}

The considerations of this section enable us to prove several properties concerning the morphism $$\Psi_e:\Gr_{e}(B)\to \prod_{ f+g=e}\Gr_{ f}(M)\times \Gr_{g}(N)$$ induced by short the exact sequence under consideration:

\begin{prop} Let $A\subset M$ and $V\subset N$ be subrepresentations such that $e=\udim A+\udim V$.

\begin{enumerate}
\item The fibre $\Psi_e^{-1}(A,V)$ is empty if and only if $A=0$ and $V\cong C\oplus L$ where $C$ is an inner subrepresentation and $L=0$ or $L\subset N/C$ is preprojective.

\item We have $[V,M/A]=\Sc{\udim V}{\udim M/A}$. In particular, we have $\Psi_e^{-1}(A,V)=\mathbb{A}^{\Sc{\udim V}{\udim M/A}}$ if $\Psi_e^{-1}(A,V)$ is not empty, i.e. the fibre dimensions only depend on the dimension vectors of $V$ and $A$.
\end{enumerate}
\end{prop}

\begin{proof}

The strategy of the proof is adopted from \cite[Lemma 3.11]{cc}. 

If $A=0$, by Proposition \ref{inout2}, the inclusion $V\hookrightarrow N$ does not factor through $B$ if and only if $V\cong C\oplus L$ where $C$ is an inner subrepresentation and $L=0$ or $L\subset N/C$ is preprojective. This yields one direction of the first part. 

From now on let $A\neq 0$. If $V$ is direct sum of outer representations, by Proposition \ref{inout2}, every injection $V\hookrightarrow N$ factors through $B$. Thus the fibre is not empty. 

Next let $V= C$ be an inner subrepresentation of $N$. Then we have $C=\kappa^m M$ with $1\leq m\leq l\leq n-3$. We have $\Hom(N/C,M)=0$ because $N/C$ is exceptional regular by Lemma \ref{verschoben}. As $C=\kappa^m M$, considering the appropriate long exact sequence, we obtain $\Ext(N/C,M)=0$ and thus $\Ext(C,M)\cong\Ext(N,M)=k$. Since the representation $C$ is preprojective and $M/A$ has no preprojective direct summand, we have $\Ext(C,M/A)=0$ and thus we get a surjection
\[\Ext(C,A)\twoheadrightarrow\Ext(C,M)\cong\Ext(N,M).\]
In particular, we get a commutative diagram
\[
\begin{xy}
\xymatrix@R15pt@C20pt{0\ar[r] &M\ar[r]&\ar[r]B&N\ar[r]&0\\0\ar[r] &A\ar[r]\ar[u]^{i_A}&\ar[r]\ar[u]U&C\ar[r]\ar[u]\ar[u]^{i_C}&0
}

\end{xy}
\]
showing that $\Psi_{e}^{-1}(A,C)\neq\emptyset$ where $e=\udim U$. 

Finally, if $V= C\oplus L$ is a direct sum of an inner subrepresentation and  outer subrepresentations, we can combine both cases in order to show that the fibre is not empty. In summary, we obtain the first statement.

The second part is obtained as follows. By Lemma \ref{afffib}, we have $\Psi_e^{-1}(A,V)=\mathbb{A}^{[V,M/A]}$. If $A\neq 0$, we have $\Ext(V,M/A)=0$ because $M/A$ has no preprojective direct summand by the preceding considerations. Thus the second part follows in this situation. If $A=0$ and if the fibre is not empty, by the first part, $V$ is forced to be a direct sum of outer subrepresentation. Let $V\cong V_1\oplus\dots V_r$ be its direct sum decomposition. Since $V_i$ does also not lie on the border, by the second part of Lemma \ref{inout},  there is no path from $M$ to $V_i$ in the Auslander-Reiten quiver for all $1\leq i\leq r$. It follows that we have $\Ext(V_i,M)=0$ and thus $\dim\Hom(V_i,M)=\Sc{\udim V_i}{\udim M}$. 
\end{proof}

Keeping in mind Proposition \ref{bundle}, the considerations of this section together with \cite[Theorem 4.4]{LW} now yield that there exists a cell decomposition for every quiver Grassmannian attached to preprojective representations (resp. preinjective representations).

\begin{thm}\label{almost2} Let $B\in\mathrm{Rep}(Q)$ be an indecomposable preprojective representation with $\delta(B)=-2$. Then there exist two preprojective representations $M$ and $N=\kappa^{l} M$ with $\delta(M)=\delta(N)=-1$ and a short exact sequence
\[\ses{M}{B}{N}\] such that for $\Psi_e$ we have
\[\Psi_e^{-1}(A,V)=\begin{cases}\emptyset\text{ if } A=0,\,V\cong C\oplus L,\, C\text{ an $(M,N)$-inner subrepresentation }\\\mathbb{A}^{\Sc{\udim V}{\udim M/A}}\text{ otherwise}\end{cases}.\]
Moreover, the fibres $\Psi_e^{-1}(A,V)$ are constant over $C_{\beta}^M\times C_{\beta'}^N\subseteq \Gr_{\udim A}(M)\times\Gr_{\udim V}(N)$ for each pair $(\beta,\beta')$ of type $(\udim A,\udim V)$ and $\Gr_e(B)$ has a cell decomposition into affine spaces.
\end{thm}

Using Theorem \ref{dual}, we obtain:

\begin{cor}
Let $B\in\Rep(Q)$ be a preinjective representation with $\delta(B)=2$. Then every quiver Grassmannian $\Gr_e(B)$ has a cell decomposition into affine spaces.
\end{cor}

The results of this section can now be used to obtain the $F$-polynomials of indecomposable representations of large defect. It is straightforward to check that, in terms of cluster variables, this corresponds to the multiplication formula, see Theorem \ref{multform}:

\begin{thm}\label{Fdef2}
Let $B$ be an indecomposable representation with $\delta(B)=-2$. If $\ses{M}{B}{N}$ is a short exact sequence as in Theorem \ref{almost2} we have
\[F_B=F_NF_M-x^{\udim\tau^{-1}M}F_{N/\tau^{-1}M}.\]
\end{thm}

\begin{proof}
First recall that $\Gr_{\udim C}(N)=\{\text{pt}\}$ for every inner representation $C\in\mathcal C$. Every regular subrepresentation $V$ of $N/C$ gives rise to a subrepresentation $C'$ where $C'$ is also an inner subrepresentation of $N$ such that fibre of $(0,C')$ is empty. Moreover, every preprojective subrepresentation $L$ gives rise to a subrepresentation $C\oplus L$ of $N$ such that the fibre of $(0,C\oplus L)$ is empty. We can also combine both cases in the natural way. Choosing $C=\tau^{-1} M$ as in Remark \ref{inner}, these observations can be summarized to
\[\chi(\Gr_e(B))=\sum_{f+g=e}\chi(\Gr_f(M))\chi(\Gr_g(N))-\sum_{f=e-\udim\tau^{-1}M}\chi(\Gr_f(N/\tau^{-1}M)).\]
Now it is straightforward that, in terms of $F$-polynomials, this translates to the claim.
\end{proof}

Clearly, the analogous statement holds for preinjective representations $B$ with $\delta(B)=2$.

\subsection{Representations of the homogeneous tubes}\label{sechom}
\noindent In this section, we consider quiver Grassmannians of indecomposable representations lying in one of the homogeneous tubes. It turns out that they are independent of the chosen tube because the quiver Grassmannians of indecomposable representations of dimension $\delta$ are independent of the chosen homogeneous tube, see \cite[Lemma 5.3]{Dupont11} and \cite[Theorem 4.4]{LW}. Note that this can also be checked by hand, see section \ref{homotube}. We fix a homogeneous tube and denote by $M_{r\delta}$ the indecomposable representation of dimension $r\delta$ which lies in this tube where $r\geq 1$. There exists a chain of irreducible inclusions
\[\{0\}\xhookrightarrow{i_1} M_{\delta}\xhookrightarrow{i_2} M_{2\delta}\xhookrightarrow{i_3}\ldots\xhookrightarrow{i_{r}}M_{r\delta}\xhookrightarrow{i_{r+1}}\ldots\]
Actually, we can recursively construct all representations $M_{r\delta}$ by considering non-splitting short exact sequences
\[0\to M_{(r-1)\delta}\xrightarrow{i_r} M_{r\delta}\xrightarrow{\pi_r}M_{\delta}\to 0.\]

The idea is to proceed along the lines of section \ref{EC}. Thus we start with considering the morphism
\[\Psi^r_e:\Gr_e(M_{r\delta})\to\coprod_{f+g=e}\Gr_f(M_{(r-1)\delta})\times\Gr_g(M_\delta),\, U\mapsto(i_r^{-1}(U),\pi_r(U)).\]

\begin{rem}\label{homog}
We have $\Gr_{m\delta}(M_{r\delta})=\{\mathrm{pt}\}$ if $m\leq r$. Indeed, the only subrepresentations of $M_{r\delta}$ are preprojective or contained in the tube of $M_{r\delta}$. Since $\delta(m\delta)=0$ and since the defect is additive, a subrepresentation of dimension $m\delta$ cannot contain a preprojective direct summand.
\end{rem}

\begin{lemma}\label{hilfs4} Let $A$ be a subrepresentation of $M_{(r-1)\delta}$ such that $0\neq A':=\pi_{r-1}(A)\subsetneq M_{\delta}$. Then we have $\Ext(M_{\delta},M_{(r-1)\delta}/A')=0$. 
\end{lemma}

\begin{proof}Since $A'$ is a proper subrepresentation of $M_{\delta}$, we have that $A'$ is preprojective and $M_{\delta}/A'$ is preinjective. Thus we have $\Ext(M_{\delta}, M_{\delta}/A')=0$. But since $A'$ is preprojective, the inclusion $A'\hookrightarrow M_{\delta}$ factors through $M_{(r-1)\delta}$. In particular, we get (using the universal property of the cokernel of $A'\hookrightarrow M_{(r-1)\delta}$) a commutative diagram
\[
\begin{xy}
\xymatrix@R15pt@C20pt{0\ar[r] &M_{(r-2)\delta}\ar[r]&\ar[r]M_{(r-1)\delta}/A'&M_{\delta}/A'\ar[r]&0\\0\ar[r] &M_{(r-2)\delta}\ar[r]\ar@{=}[u]&\ar[r]\ar[u]M_{(r-1)\delta}&M_{\delta}\ar[r]\ar[u]&0
}
\end{xy}
\]
But since the lower sequence does not split and, moreover, since $\Ext(M_{\delta},M_{(r-2)\delta})=k$, this shows that $\Hom(M_{\delta},M_{\delta}/A')\to\Ext(M_{\delta},M_{(r-2)\delta})$ is surjective. Thus we have $$\Ext(M_{\delta}, M_{(r-1)\delta}/A')\cong \Ext(M_{\delta}, M_{\delta}/A')=0,$$
which completes the proof.
\end{proof}

\begin{lemma}\label{hilfs5}
The fibre $(\Psi^r_e)^{-1}(A,V)$ is empty if and only if $V=M_{\delta}$  and $i_{r-1}^{-1}(A)\cong A$, i.e. $A$ is already a subrepresentation of $M_{(r-2)\delta}$.
\end{lemma}

\begin{proof}
If $V=0$, the fibre of $(A,V)$ is clearly not empty because every subrepresentation of $M_{(r-1)\delta}$ is already a subrepresentation of $M_{r\delta}$. 

If $0\neq V\subsetneq M_{\delta}$, we have that $V$ is preprojective and thus the canonical inclusion factors through $M_{r\delta}$. In particular, the fibre of $(A,V)$ is not empty. 

Thus assume that $V=M_{\delta}$. If $A=0$, the fibre is empty because the sequence does not split. For general $A\subset M_{(r-1)\delta}$, we consider the long exact sequence
\[\begin{xy}\xymatrix@R5pt@C40pt{0\ar[r]&\Hom(M_{\delta},A)\ar[r]&\Hom(M_{\delta},M_{(r-1)\delta})\ar[r]&\Hom(M_{\delta},M_{(r-1)\delta}/A)\ar[r]&\\\ar[r]&\Ext(M_{\delta},A)\ar[r]^(.43){f^A_r}&\Ext(M_{\delta},M_{(r-1)\delta})\ar[r]&\Ext(M_{\delta},M_{(r-1)\delta}/A)\ar[r]&0}\end{xy}\]

Since $\dim\Ext(M_{\delta},M_{(r-1)\delta})=1$, we have that the fibre is empty if and only if $f^A_r=0$. This is obviously the case if and only if $\Ext(M_{\delta},M_{(r-1)\delta}/A)=k$. In turn the fibre is not empty if and only if $f_r^A$ is surjective. First assume that $i_{r-1}^{-1}(A)\cong A$. If $A=M_{(r-2)\delta}$, we have $M_{(r-1)\delta}/A\cong M_{\delta}$ by Remark \ref{homog}. Then every vector space in the above sequence is isomorphic to $k$. In particular, we have $f^A_r=0$. 

Furthermore, if $A\subset M_{(r-2)\delta}$, we have that $f^A_r$ is the composition 
\[\Ext(M_{\delta},A)\xrightarrow{f_{r-1}^A}\Ext(M_{\delta},M_{(r-2)\delta})\xrightarrow{f_r^{M_{(r-2)\delta}}}\Ext(M_{\delta},M_{(r-1)\delta}).\]
Thus also in this case the fibre is empty.

If $A'':=i_{r-1}^{-1}(A)\ncong A$, we have $A'=\pi_{r-1}(A)\neq 0$. Then we have the following commutative diagram
\[
\begin{xy}
\xymatrix@R15pt@C25pt{0\ar[r] &M_{(r-2)\delta}\ar[r]&\ar[r]M_{(r-1)\delta}&M_{\delta}\ar[r]&0\\0\ar[r] &A''\ar[r]\ar[u]&\ar[r]\ar[u]A&A'\ar[r]\ar[u]&0
}
\end{xy}
\]
inducing a diagram
\[
\begin{xy}
\xymatrix@R15pt@C45pt{\Ext(M_{\delta},A'')\ar[d]\ar[r]^(.45){f_{r-1}^{A''}}&\ar[d]\Ext(M_{\delta},M_{(r-2)\delta})\ar[r]&\Ext(M_{\delta},M_{(r-2)\delta}/A'')\ar[r]\ar[d]&0\\\Ext(M_{\delta},A)\ar[d]\ar[r]^(.45){f_{r}^{A}}&\ar[d]\Ext(M_{\delta},M_{(r-1)\delta})\ar[r]&\Ext(M_{\delta},M_{(r-1)\delta}/A)\ar[r]\ar[d]&0\\\Ext(M_{\delta},A')\ar[r]&\Ext(M_{\delta},M_{\delta})\ar[r]&\Ext(M_{\delta},M_{\delta}/A')\ar[r]&0
}
\end{xy}
\]
If $A'=M_{\delta}$, we clearly have $\Ext(M_{\delta},M_{\delta}/A')=0$. By induction, we have that $f^{A''}_{r-1}$ is surjective and thus it follows that $\Ext(M_{\delta},M_{(r-1)\delta}/A)=0$. Thus the fibre is not empty. 

Finally, assume that $0\neq A'=\pi_{r-1}(A)\neq M_{\delta}$. By Lemma \ref{hilfs4}, we have $\Ext(M_{\delta},M_{(r-1)\delta}/A')=0$. Thus $f^{A'}_r$ is surjective. But since $A'$ is a factor of $A$, we have $\Ext(M_{\delta},A)\twoheadrightarrow\Ext(M_{\delta},A')$. Thus the fibre of $(A,M_{\delta})$ is not empty. 
\end{proof}

\begin{lemma}\label{hilfs6}
If $(\Psi^r_e)^{-1}(A,V)\neq\emptyset$, we have $(\Psi^r_e)^{-1}(A,V)=\mathbb{A}^{\Sc{\udim V}{\udim M_{(r-1)\delta}/A}}$.
\end{lemma}

\begin{proof}
By Lemma \ref{afffib}, we have
$$(\Psi^r_e)^{-1}(A,V)=\mathbb{A}^{[\udim V,\udim M_{(r-1)\delta}/A]}$$
if it is not empty. The statement is clearly true for $V=0$. If $0\subsetneq V\subsetneq M_{\delta}$, it is preprojective. Since $M_{(r-1)\delta}/A$ has no preprojective direct summand, we have $\Ext(V,M_{(r-1)\delta}/A)=0$ in this case. If $V=M_{\delta}$ and the fibre is not empty, by the considerations in the proof of Lemma \ref{hilfs5}, we have that $f_r^A$ is surjective. Thus we have $\Ext(M_{\delta},M_{(r-1)\delta}/A)=0$.
\end{proof}

The preceding lemmas together with Lemma \ref{afffib} and Proposition \ref{bundle} can now be used to prove the main result of this section: 

\begin{thm}\label{homogaff}
Every quiver Grassmannian $\Gr_e(M_{r\delta})$ has a cell decomposition into affine spaces.  Moreover, this decomposition is compatible with the decomposition
\begin{align}\label{decomp} \Gr_e(M_{r\delta})=\{U\in\Gr_e(M_{r\delta})\mid \pi_r(U)= 0\}\cup\{U\in\Gr_e(M_{r\delta})\mid \pi_r(U)\neq 0\}.\end{align}
\end{thm}

\begin{proof}
We proceed by induction on $r$. If $r=1$, the claim follows by \cite[Theorem 4.4]{LW}. Alternatively, it is straightforward to check by hand that every quiver Grassmannian has a cell decomposition. Since we clearly have $\pi_1(U)\cong U\neq 0$ for every subrepresentation $0\neq U\subset M_{\delta}$, also the compatibility follows. 

Thus let $r\geq 2$. By Lemma \ref{hilfs5}, the fibre of $(A,V)$ is empty if and only if $A$ is a subrepresentation of $M_{(r-2)\delta}$ and $V=M_{\delta}$. Since $\Gr_f(M_{(r-1)\delta})$ and $\Gr_g(M_{\delta})$ have cell decompositions, by Lemma \ref{hilfs6} together with Proposition \ref{bundle}, it follows that
\[(\Psi_e^r)^{-1}(\Gr_f(M_{(r-1)\delta})\times \Gr_g(M_{\delta}))\]
has a cell decomposition if $g\neq\delta$. If $g=\delta$, the fibre is empty if $\pi_{r-1}(A)= 0$. Since the cell decompositions of the quiver Grassmannians $\Gr_f(M_{(r-1)\delta})$ are compatible with the decomposition (\ref{decomp}) by induction hypothesis, the claim follows in this case in the same way. 

Since we have $\pi_r((\Psi_e^r)^{-1}(A,V))=0$ if and only if $V=0$, it follows that
\[\{U\in\Gr_e(M_{r\delta})\mid \pi_r(U)= 0\}=(\Psi_e^r)^{-1}(\Gr_e(M_{(r-1)\delta})\times\{0\})\]
and 
\[\{U\in\Gr_e(M_{r\delta})\mid \pi_r(U)\neq 0\}=(\Psi_e^r)^{-1}(\coprod_{\substack{f+g=e\\g\neq 0}}\Gr_f(M_{(r-1)\delta})\times \Gr_g(M_{\delta})).\]
This already shows that the cell decompositions of the quiver Grassmannians $\Gr_e(M_{r\delta})$ are also compatible with decomposition (\ref{decomp}).
\end{proof}

We define $F_{r\delta}:=F_{M_{r\delta}}$. Now the following Corollary, which is also obtained in \cite[Theorem 7.1]{dupont2012cluster} is straightforward:

\begin{cor}\label{homogform1} We have
\[F_{r\delta}=F_{\delta}F_{(r-1)\delta}-x^{\delta}F_{(r-2)\delta}\]
for $r\geq 1$ where $F_{0}=1$ and $F_{-\delta}:=0$.
\end{cor} 
In section \ref{homotube}, we use this recursive formula to obtain an explicit formula for the $F$-polynomial $F_{r\delta}$. 

\subsection{Representations of positive defect}\label{posdef}

\noindent For indecomposable representations $M$ of positive defect, we can deduce that $\Gr_e(M)$ decomposes into affine spaces from the corresponding fact for indecomposable representations of negative defect when passing to dual representations. 

Let $M$ be a representation of $Q$ with ordered basis $\mathcal B$ and $d:=\udim M$.
It defines the dual basis $\cB^\ast$ of $M^\ast$, which consists of the linear maps $f_b:M\to \C$ with $f_b(b')=\delta_{b,b'}$. We endow $\cB^\ast$ with the inverse order of $\cB$. Note that the coefficient quiver $\Gamma(M^\ast,\cB^\ast)$ is obtained from the coefficient quiver $\Gamma(M,\cB)$ by inverting the arrows.

If $M=s_{p_1}\dotsb s_{p_t} S_p$ is preprojective, then $M^\ast=s_{p_1}\dotsb s_{p_t} S_p^\ast$ is preinjective, and vice versa. Therefore the dual $(-)^\ast$ establishes a correspondence between the preprojective representations of $Q$ and the preinjective representations of $Q^\op$. If $Q$ is of extended Dynkin type $\tilde D_n$, the absolute value of the defect depends on whether $p\in\{q_a,q_b,q_c,q_d\}$. Thus this correspondence restricts to a correspondence between defect $-1$ (or defect $-2$) preprojectives and defect $1$ (or defect $2$) preinjectives. 

For a subrepresentation $U$ of $M$ with dimension vector $e$, we define $U^\ast=(U^\ast_q)_{q\in Q}$ as the collection of subspaces
\[
 U^\ast_q \ = \ \{ \, f\in M^\ast_q \, | \, f(n)=0\text{ for all }n\in U_q \, \}
\]
of $M^\ast_q$. For a subset $\beta$ of $\cB$, we define its dual as
\[
 \beta^\ast \ = \ \{ \, f\in\cB^\ast \, | \, f(b)=0\text{ for all }b\in\beta \,\},
\]
which is of type $e^\ast=d-e$. Note that $\beta^\ast$ is the complement of the set of dual elements $f_b\in\cB^\ast$ of basis vectors $b\in\beta$.

\begin{thm}\label{dual}
 The association $U\mapsto U^\ast$ defines an isomorphism $\Gr_e(M)\stackrel\sim\to\Gr_{e^\ast}(M^\ast)$, which restricts to an isomorphism $C^M_\beta\stackrel\sim\rightarrow C^{M^\ast}_{\beta^\ast}$ between Schubert cells for every subset $\beta$ of $\cB$. 
\end{thm}

\begin{proof}
 Let $\tilde d=\sum d_p$ be the total dimension of $d=(d_p)$, $\tilde e$ the total dimension of $e$ and $\tilde e^\ast$ the total dimension of $e^\ast$. Then $M$ has dimension $\tilde d$ as a $\C$-vector space, a subrepresentation $U$ with dimension vector $e$ has dimension $\tilde e$ over $\C$ and a subrepresentation $U^\ast$ of $M^\ast$ with dimension vector $e$ has dimension $\tilde e^\ast$ over $\C$.

 The canonical isomorphism $\Lambda^{\tilde e} M\stackrel\sim\longrightarrow\Lambda^{\tilde e^\ast}M^\ast$ that sends $n_1\wedge\dotsb\wedge n_{\tilde e}$ to the unique element $f_1\wedge\dotsb\wedge f_{\tilde e^\ast}$ with $f_j(n_i)=0$ for all $i=1,\dotsc,\tilde e$ and $j=1,\dotsc,\tilde e^\ast$ induces an isomorphism $\Phi:\P\bigl(\Lambda^{\tilde e} M\bigr)\stackrel\sim\longrightarrow\P\bigl(\Lambda^{{\tilde e}^\ast}M^\ast\bigr)$ between the corresponding projective spaces.

 A subrepresentation $U$ of $M$ of type $e$ corresponds to a point $\iota(U)$ of $\P\bigl(\Lambda^{\tilde e} M\bigr)$ and $U^\ast$ corresponds to a point $\iota^\ast(U^\ast)$ of $\P\bigl(\Lambda^{{\tilde e}^\ast}M^\ast\bigr)$ where $\iota$ and $\iota^\ast$ denote the respective Pl\"ucker embeddings. It is clear from the definitions that $\iota^\ast(U^\ast)=\Phi\bigl(\iota(U)\bigr)$.

 The following calculation shows that $U^\ast$ is a subrepresentation of $M^\ast$, i.e.\ that $M_{v^\ast}^\ast(U_q^\ast)\subset U_p^\ast$ for all arrows $v:p\to q$ of $Q$ where $v^\ast:q\to p$ denotes the dual arrow of $Q^\op$. For $f\in U_q^\ast$, we have $f(M_v(n))=0$ for all $n\in U_p$ since $M_v(U_p)\subset U_q$. By the defining property of the adjoint map $M^\ast_{v^\ast}$ of $M_v$, we have $f(M_v(n))=M^\ast_{v^\ast}(f)(n)$. Thus $M^\ast_{v^\ast}(f)(n)=0$ for all $n\in U_p$, which shows that the functional $M^\ast_{v^\ast}(f)$ is indeed an element of $U_p^\ast$.

 We conclude that the isomorphism $\Phi:\P\bigl(\Lambda^{\tilde e} M\bigr)\stackrel\sim\longrightarrow\P\bigl(\Lambda^{\tilde e^\ast}M^\ast\bigr)$ restricts to a morphism $\Phi':\Gr_e(M)\to\Gr_{e^\ast}(M^\ast)$. By the same arguments as above applied to the respective dual spaces $M^\ast$ and $(M^\ast)^\ast=M$, we see that the inverse $\Psi$ of $\Phi$ restricts to a morphism $\Psi':\Gr_{e^\ast}(M^\ast)\to\Gr_{e}(M)$, which is inverse to $\Phi'$. This shows that the association $U\mapsto U^\ast$ defines an isomorphism $\Phi':\Gr_e(M)\to\Gr_{e^\ast}(M^\ast)$, which establishes the first claim of the theorem.

 For a subset $\beta$ of $\cB$, the Schubert cell $C_\beta^M$ consists of all subrepresentations $U=(U_p)$ of $M$ such that $\Delta_{\beta_p}(U_p)\neq 0$ and $\Delta_{\beta'_p}(U_p)=0$ for all $p\in Q_0$ and $\beta'>\beta$. The isomorphism $\Phi':\Gr_e(M)\to\Gr_{e^\ast}(M^\ast)$ restricts to an isomorphism $C_\beta^M\stackrel\sim\to C_{\beta^\ast}^{M^\ast}$ of Schubert cells since it is compatible with the vanishing of Pl\"ucker coordinates, i.e.\ we have $\Delta_{\beta_p}(U_p)=0$ if and only if $\Delta_{\beta^\ast}(U_p^\ast)=0$ where we make use of the notation from the Introduction of \cite{LW}. This shows the second claim of the theorem.
%
%
\end{proof}

\section{Reductions of quiver Grassmannians}\label{reductionofgrass}
\noindent In this section, we show how we can simplify the determination of quiver Grassmannians by passing to smaller quivers and smaller roots respectively. Together with BGP-reflections reviewed in section \ref{bgp}, it turns out that these methods are very useful when calculating generating functions of representations of $\tilde D_n$ in section \ref{secFpoly}.
\subsection{Reduction of type one}\label{redone}
Assume that $Q$ has a full subquiver of the form
\[p_0\xleftarrow{\rho_1} p_{1}\xleftarrow{\rho_2} p_{2}.\]
Let $\alpha$ be a dimension vector of $Q$ and let $M$ be a representation of dimension $\alpha$ such that the linear maps $M_{\rho_i}$ for $i=1,2$ have maximal rank. Note that this is true for all linear maps of real root representation. This follows because they are even of maximal rank type, see \cite{wie}.

Let $e$ be a second dimension vector such that $e\leq\alpha$ and such that $\mathrm{Gr}_{e}(M)\neq\emptyset$. Thus, in terms of quiver Grassmannians, we consider a commutative diagram
\[
\begin{xy}
\xymatrix@R20pt@C20pt{k^{\alpha_0}&k^{\alpha_{1}}\ar[l]&k^{\alpha_{2}}\ar[l]\\k^{e_0}\ar@{^{(}->}[u]&k^{e_{1}}\ar@{^{(}->}[u]\ar[l]&k^{e_{2}}\ar[l]\ar@{^{(}->}[u]}
\end{xy}
\]
where $\alpha_i:=\alpha_{p_i}$ and $e_i:=e_{p_i}$.

Let $Q(p_{1})$ be the quiver of type $\tilde D_{n-1}$ resulting from $Q$ when deleting the vertex $p_{1}$ and the two corresponding arrows and, moreover, when adding an extra arrow $p_{2}\to p_{0}$. Moreover, let $\hat e,\hat\alpha$ be the corresponding dimension vectors and $\hat M$ the induced representation. 

If $\alpha_2\leq \alpha_1=\alpha_0$, all maps in the diagram are injective and we have $e_2\leq e_1\leq e_0$. It is easy to check that $\hat M$ is indecomposable if and only if $M	$ is indecomposable. For a vector space $V$, we denote by $\Gr(l,V)$ the usual Grassmannian (resp. $\Gr(l,r)$ in the case $V=k^r$). Recall that, for a fixed vector space $V$ and a subspace $U$ with $\dim U\leq l$, we have
\[\Gr(l,U,V):=\{W\in\Gr(l,V)\mid U\subset W\}\cong\Gr(l-\dim U,\dim V-\dim U).\]
Thus it is straightforward that $ \mathrm{Gr}_{e}(M)$ is $\Gr(e_1-e_2,e_0-e_2)$-bundle over $\Gr_{\hat e}(\hat M)$. In particular, we have 
$$\chi(\mathrm{Gr}_{e}(M))=\chi(\Gr(e_1-e_2,e_0-e_2))\chi(\Gr_{\hat e}(\hat M)).$$
Note that, since $\alpha_1=\alpha_0$ every subspace of $M_{p_0}\cong k^{\alpha_0}$ can be identified with a subspace of $M_{p_1}\cong k^{\alpha_1}$ and vice versa.
Moreover, note that, in case of a subquiver
\[p_2\leftarrow p_1\leftarrow p_0\]
with dimension vector $\alpha$ satisfying $\alpha_2\leq\alpha_1=\alpha_0$, we  can turn around all arrows in $Q$ to obtain the situation treated above. 

We want to consider a similar case: using the same notation, we assume that we are faced with the following situation
\[
\begin{xy}
\xymatrix@R15pt@C20pt{&k^{e_0}&&&&k^{\alpha_0}\\&k^{e_1}\ar[u]&&\subseteq&&k^{\alpha_1}\ar@{=}[u]\\k^{e_{2}}\ar[ru]&&k^{e_{3}}\ar[lu]&&k^{\alpha_{2}}\ar@{^{(}->}[ru]&&k^{\alpha_{3}}\ar@{_(->}[lu]}
\end{xy}
\]
where $\alpha_2+\alpha_3\leq \alpha_1=\alpha_0$ and $e\leq\alpha$. Assume that $k^{\alpha_2}\cap k^{\alpha_3}=\{0\}$ when understanding these two vector spaces as subspaces of $k^{\alpha_1}$. This is again true for real root representations and representations of maximal rank type respectively. 

Since all maps in the diagram are injective, similar to the preceding case, we can reduce this situation to the case
\[
\begin{xy}
\xymatrix@R15pt@C20pt{&k^{e_0}&&&k^{\alpha_0}\\k^{e_{2}}\ar[ru]&&k^{e_{3}}\ar[lu]&k^{\alpha_{2}}\ar[ru]&&k^{\alpha_{3}}\ar[lu]}
\end{xy}
\]

Note that it is again straightforward to check that $\alpha$ is a root if and only if $\hat\alpha$ is a root (resp. that the corresponding representation $M$ is indecomposable if and only if $\hat M$ is indecomposable). As before we get a subrepresentation of $M$ for a fixed subrepresentation of $\hat M$ together with a subspace $U\in\Gr(e_1-e_2-e_3,e_0-e_2-e_3)$. As above, we obtain that $\Gr_{e}(M)$ is a $\Gr(e_1-e_2-e_3,e_0-e_2-e_3)$-bundle over $\Gr_{\hat e}(\hat M)$ and we have
$$\chi(\Gr_{e}(M))= \chi(\Gr_{\hat e}(\hat M))\chi(\Gr(e_1-e_2-e_3,e_0-e_2-e_3)).$$
Note that there is again a dual case obtained when turning around all arrows. 
In the sequel we will refer to these two procedures as reduction of type one.
\subsection{Application to real root representations and examples}
When calculating the generating function of the Euler characteristics of quiver Grassmannians, it turns out that the reduction of type one is a very powerful tool when combining it with BGP-reflections. Actually, we can start with $\tilde D_n$ in subspace orientation where we can apply reductions of type one. Then we can show that the obtained formulae are invariant under BGP-reflections. Note that the reductions of type one preserve possible cell decompositions into affine spaces whence it is not clear under which conditions this is true for BGP-reflections.

Considering the Auslander-Reiten quiver of $\tilde D_n$ in subspace orientation, it can be seen easily that the preprojectives of defect $-1$ can be reduced to 
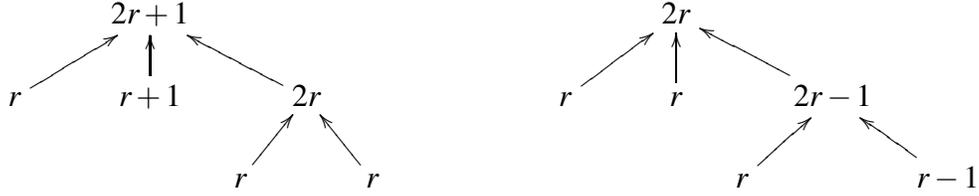
\[\begin{xy}\xymatrix@R15pt@C10pt{&&2r+1&&&&&&&&&2r&&\\r\ar[rru]&&r+1\ar[u]&&2r\ar[llu]&&&&&r\ar[rru]&&r\ar[u]&&2r-1\ar[llu]\\&&&r\ar[ru]&&r\ar[lu]&&&&&&&r\ar[ru]&&r-1\ar[lu]}\end{xy}\]

and 
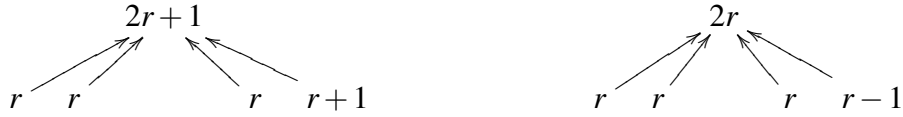
\[\begin{xy}\xymatrix@R15pt@C10pt{&&2r+1&&&&&&&&&2r&&\\r\ar[rru]&r\ar[ru]&&r\ar[lu]&r+1\ar[llu]&&&&&r\ar[rru]&r\ar[ru]&&r\ar[lu]&r-1\ar[llu]}\end{xy}\]
where the numbers indicate the dimension vector.

Note that in the same way we can reduce the calculation of quiver Grassmannians of indecomposables of preinjective roots to the case of $\tilde D_5$. Alternatively, we can consider the opposite quiver and restrict to preprojective roots.

Next we consider the non-exceptional real roots. It is easy to check that we can actually reduce all the real roots of the tubes of rank two to cases of the form 
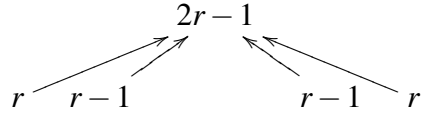
\[\begin{xy}\xymatrix@R15pt@C10pt{&&2r-1\\r\ar[rru]&r-1\ar[ru]&&r-1\ar[lu]&r\ar[llu]}\end{xy}\]
Most of the real roots of the exceptional tubes of rank $n-2$ can be reduced to the case $n=5$, the remaining to $n=6$. Roughly speaking, the worst roots to consider are the following:
 
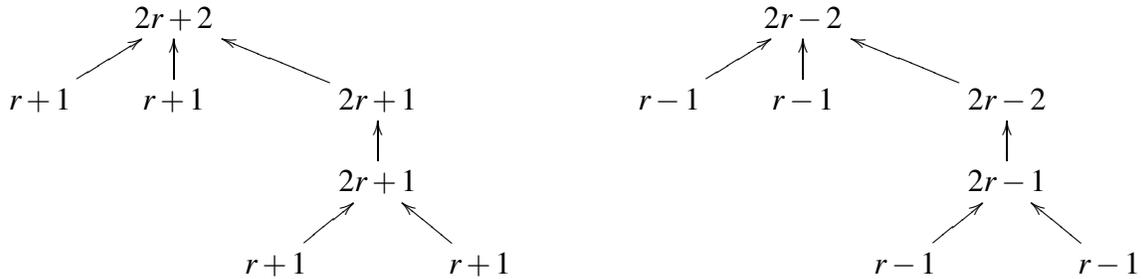
\[\begin{xy}\xymatrix@R15pt@C5pt{&&2r+2&&&&&&&&&2r-2&&\\r+1\ar[rru]&&r+1\ar[u]&&2r+1\ar[llu]&&&&&r-1\ar[rru]&&r-1\ar[u]&&2r-2\ar[llu]\\&&&&2r+1\ar[u]&&&&&&&&&2r-1\ar[u]\\&&&r+1\ar[ru]&&r+1\ar[lu]&&&&&&&r-1\ar[ru]&&r-1\ar[lu]}\end{xy}\]

In summary, as far as subspace orientation is concerned, by the introduced reductions steps, we can stick to the exceptional roots of $\tilde D_5$ and to the non-exceptional roots of $\tilde D_6$. As far as the calculation of $F$-polynomials is concerned, it suffices to consider the case $n=5$.

Note that one has to be careful when applying these methods to imaginary root representations lying in the exceptional tubes because not all linear maps are of maximal rank.
\section{BGP-reflections and quiver Grassmannians}\label{bgp}
\noindent Another method to get morphisms and connections between quiver Grassmannians and the corresponding generating functions is to consider the reflection functor introduced by Bernstein, Gelfand and Ponomarev in \cite{bgp}, see \cite[section 5]{wol} and \cite[section 5]{dwz2}. For a quiver $Q$, consider the matrix $A=(a_{p,q})_{p,q\in Q_0}$ with $a_{p,p}=2$ and $a_{p,q}=a_{q,p}$ for $p\neq q$, in which $a_{p,q}=|\{v\in Q_1\mid v:p\rightarrow q\vee  v:q\rightarrow p\}|$. Fixed some $q\in Q_0$ define $\sigma_q:\mathbb{Z}Q_0\rightarrow\mathbb{Z}Q_0$ as
\[\sigma_q(p)=p-a_{q,p}q.\] 
Let $Q$ be a quiver and $q\in Q_0$ a sink (resp. a source). Then by $\sigma_qQ$ we denote the quiver which is obtained from $Q$ by turning around all arrows with tail (resp. head) $q$. In both cases we denote the reflection functors, which are additive functors,  by $\sigma_q:\Rep(Q)\rightarrow \Rep(\sigma_qQ)$. If $M$ is a representation of $Q$ and $q$ is a sink (resp. a source) we consider the linear maps
\[\phi_q^M:\bigoplus_{p\xrightarrow{v}q}M_p\xlongrightarrow{M_v} M_q\text{  (resp.}\quad\phi_q^M: M_q\xlongrightarrow{M_v}\bigoplus_{q\xrightarrow{v}p}M_p).\]
Recall that in both cases we have $(\sigma_qM)_{p}=M_{p}$ if $p\neq q$. Moreover, we have $(\sigma_qM)_{q}=\mathrm{Ker}(\phi_q^M)$ (resp. $(\sigma_qM)_{q}=\mathrm{coker}(\phi_q^M)$). Now the linear maps $(\sigma_qM)_v$ for $v:p\to q$ (resp. $(\sigma_qM)_v$ for $v:q\to p$) are the natural ones. Moreover, the maps $M_{v'}$ for the remaining arrows $v'\in Q_1$ do not change. The functors have the following properties:
\begin{enumerate}
\item If $M\cong S_q$, then $\sigma_q(S_q)=0$.
\item If $M\ncong S_q$ is indecomposable, then $\sigma_q(M)$ is indecomposable such that $\sigma_q^2(M)\cong M$ and $\udim \sigma_q(M)=\sigma_q(\udim M)$.
\end{enumerate}

In order to investigate the behavior of quiver Grassmannians and the corresponding generating functions under the reflection functor, we review and re-prove some results of \cite{wol} and \cite{dwz2}. Note that in \cite{dwz2} the more general case of mutations is treated. We define 
\[\mathrm{Gr}_e(M, q^r)=\{U\in\mathrm{Gr}_e(M)\mid\dim\Hom(U,S_q)=r\}.\]
and 
\[\mathrm{Gr}_e(q^r,M)=\{U\in\mathrm{Gr}_e(M)\mid\dim\Hom(S_q,U)=r\}.\]

In order to simplify notations, we assume that $q$ is a sink and, moreover, that $M$ is an indecomposable representation of $Q$ with $\alpha:=\udim  M$. The case when $q$ is a source can be obtained analogously or simply by considering the isomorphisms $\Gr_e(M)\cong \Gr_{\alpha-e}(M^\ast)$. 
Following \cite[section 5]{wol}, we consider the following map
\[\pi^r_q:\mathrm{Gr}_e(M, q^r)\rightarrow\mathrm{Gr}_{e-rs_q}(M,q^0)\]
where $\pi^r_q$ is defined by $\pi^r_q(U)_q=\mathrm{Im}\phi_q^U$ and $\pi^r_q(U)_{p}=U_{p}$ if $p\neq q$. Note that, indeed, $\pi^r_q(U)$ is a subrepresentation of dimension $e-rs_q$ of $M$ such that $\Hom(U,S_q)=0$. By \cite[Theorem 5.11]{wol}, we have for sinks $q$:

\begin{thm}\label{grassrefl}
The morphism $\pi^r_q$ is surjective with fibres isomorphic to $\Gr(r,\alpha_q-e_q+r)$. Moreover, there exists an isomorphism of varieties
\[\sigma_q:\mathrm{Gr}_e(M, q^0)\rightarrow\mathrm{Gr}_{\sigma_qe}(q^0,\sigma_q M),\,U\mapsto \sigma_qU.\]
The analogous statement holds if $q$ is a source.
\end{thm}
For every dimension vector $e\in\N Q_0$, there exists some $0\leq t\leq e_q$ such that $\Gr_{e-ls_q}(M,q^r)=\emptyset$ for $r\geq 1$ and $l\geq t$. Then we have $\Gr_{e-ts_q}(M,q^0)=\Gr_{e-ts_q}(M)$. Fix $t\in\N$ minimal with this property.
Thus, for $e'=e-ls_q$ with $l\geq t $ we have $\Gr_{e'}(M)\cong\Gr_{\sigma_qe'}(\sigma_qM)$. Then we have the following statement:

\begin{prop}\label{euler}
Assume that $q$ is a sink and that $\Gr_e(M)=\Gr_e(M,q^0)$. Then we have
\[\chi(\Gr_{e+ms_q}(M,q^0))=\sum_{i=0}^m(-1)^{m-i}{\alpha_q-e_q-i \choose m-i}\chi(\Gr_{e+is_q}(M)).\]
\end{prop}

\begin{proof}
First recall that $\chi(\Gr(k,n))={n\choose k}$ for $k\leq n$.
We proceed by induction on $m$. The statement is satisfied for $m=0$. By Theorem \ref{grassrefl}, we have 
\[\chi(\Gr_{e+ms_q}(M))=\sum_{i=0}^m {\alpha_q-e_q-(m-i) \choose i}\chi(\Gr_{e+(m-i)s_q}(M,q^0)).\]
Applying the induction hypothesis, we get
{\allowdisplaybreaks\begin{eqnarray*}
\chi(\Gr_{e+ms_q}(M,q^0))&=& \chi(\Gr_{e+ms_q}(M))-\sum_{i=1}^m {\alpha_q-e_q-(m-i) \choose i}\chi(\Gr_{e+(m-i)s_q}(M,q^0))\\
&=& \chi(\Gr_{e+ms_q}(M))-\sum_{i=1}^m {\alpha_q-e_q-(m-i) \choose i}\\&&\cdot\sum_{j=0}^{m-i}(-1)^{m-i-j}{\alpha_q-e_q-j \choose m-i-j}\chi(\Gr_{e+js_q}(M))\\
&=&\chi(\Gr_{e+ms_q}(M))-\sum_{j=0}^{m-1}\chi(\Gr_{e+js_q}(M))\\&&\cdot\sum_{i=1}^{m-j}(-1)^{m-i-j} {\alpha_q-e_q-(m-i) \choose i}{\alpha_q-e_q-j \choose m-i-j}\\
&=&\chi(\Gr_{e+ms_q}(M))-\sum_{j=0}^{m-1}(-1)^{m-j}{\alpha_q-e_q-j \choose m-j}\chi(\Gr_{e+js_q}(M))\\&&\cdot\sum_{i=1}^{m-j}(-1)^{i} {m-j\choose i}.
\end{eqnarray*}}Since we have 
\[\sum_{i=1}^{m-j}(-1)^{i} {m-j\choose i}=-1\]
the claim follows.
\end{proof}

We need the following identities which can be proved by induction where
 $${n\choose k}:=\frac{n(n-1)\ldots (n-k+1)}{k!}$$
for $n\in\Z$.

\begin{lemma}\label{binolem}The following holds:

\begin{enumerate}
\item For natural numbers $m,t,n$ with $n\geq t\geq m$, we have
\[\sum_{r=0}^m(-1)^r{m\choose r}{n-m+r\choose n-t}=(-1)^m{n-m\choose t}.\]
\item For natural numbers $m,t,n$ with $m\leq n<t$, we have
\[\sum_{r=0}^m(-1)^r{m\choose r}\frac{(n-m+r)!}{(t-m+r)!}=\frac{(n-m)!(t-n+m-1)!}{t!(t-n-1)!}.\]
\end{enumerate}
\end{lemma}

\begin{proof}
We proceed by induction where the result is checked easily for $m=0$. Applying the induction hypothesis and the well-known formula ${n\choose k}={n-1\choose k-1}+{n-1\choose k}$ have
{\allowdisplaybreaks\begin{eqnarray*}\sum_{r=0}^{m+1}(-1)^r{m+1\choose r}{n-(m+1)+r\choose n-t}&=&\sum_{r=0}^{m+1}(-1)^r\left({m\choose r}+{m\choose r-1}\right){n-(m+1)+r\choose n-t}\\
&=&\sum_{r=0}^{m}(-1)^r{m\choose r}{(n-1)-m+r\choose (n-1)-(t-1)}\\&&-\sum_{r=0}^{m}(-1)^{r}{m\choose r}{n-m+r\choose n-t}\\
&=&(-1)^{m+1}\left(-{n-1-m\choose t-1}+{n-m\choose t}\right)\\
&=&(-1)^{m+1}{n-(m+1)\choose t}.
\end{eqnarray*}}

The second statement can be proved similarly. We proceed by induction where the result is checked easily for $m=0$. Applying the induction, hypothesis we have
{\allowdisplaybreaks\begin{eqnarray*}\sum_{r=0}^{m+1}(-1)^r{m+1\choose r}\frac{(n-m-1+r)!}{(t-m-1+r)!}&=&\sum_{r=0}^{m}(-1)^r{m\choose r}\frac{((n-1)-m+r)!}{((t-1)-m+r)!}\\&&-\sum_{r=0}^{m}(-1)^{r}{m\choose r}\frac{(n-m+r)!}{(t-m+r)!}\\
&=&\frac{(n-1-m)!(t-n+m-1)!}{(t-1)!(t-n-1)!}-\frac{(n-m)!(t-n+m-1)!}{t!(t-n-1)!}\\
&=&\frac{(n-m-1)!(t-n+m)!}{t!(t-n-1)!}
\end{eqnarray*}}
This completes the proof of the lemma.
\end{proof}

Applying the preceding statements, we see that the Euler characteristic of a quiver Grassmannian of a representation, which is obtained by reflecting at a source or sink, is already determined by the Euler characteristics of quiver Grassmannians of the original representation:

\begin{thm}\label{bgpeuler}
Let $M$ be a representation of dimension $\alpha$. Let $q$ be a sink and $e\in\N Q_0$ such that $\Gr_e(M)=\mathrm{Gr}_e(M,q^0)$. Let $n:=(\sigma_qe)_q$ and $t:=\alpha_q-e_q$. We have 
\[\chi(\Gr_{\sigma_qe-ms_q}(\sigma_qM))=\sum_{j=0}^m \chi(\Gr_{e+js_q}(M)){n-t\choose m-j}.\]
\end{thm}

\begin{proof}
Let $q$ be a source and assume that $\Gr_{e}(M)=\Gr_{e}(q^0,M)$. This is the case if and only if $\Gr_{\alpha-e}( M^\ast)=\Gr_{\alpha-e}( M^\ast,q^0)$. Since $\chi(\Gr_{e}(M))=\chi(\Gr_{\alpha-e}( M^\ast))$, by Theorem \ref{grassrefl} (for a source $q$), we get
\begin{eqnarray*}
\chi(\Gr_{e-ms_q}(M))=\chi(\Gr_{\alpha-e+ms_q}( M^\ast))&=&\sum_{i=0}^m {e_q-(m-i) \choose i}\chi(\Gr_{\alpha-e+(m-i)s_q}( M^\ast,q^0))\\&=&\sum_{i=0}^m {e_q-(m-i) \choose i}\chi(\Gr_{e+(i-m)s_q}(q^0,M)).
\end{eqnarray*}

Thus if $\Gr_e(M)=\mathrm{Gr}_e(M,q^0)$, applying successively this statement, Theorem \ref{grassrefl} and Proposition \ref{euler}, we have
{\allowdisplaybreaks\begin{eqnarray*}
\chi(\Gr_{\sigma_qe-ms_q}(\sigma_qM))&=& \sum_{i=0}^m {(\sigma_qe)_q-(m-i) \choose i}\chi(\Gr_{\sigma_qe+(i-m)s_q}(q^0,\sigma_qM))\\
&=&\sum_{i=0}^m {(\sigma_qe)_q-(m-i) \choose i}\chi(\Gr_{e+(m-i)s_q}(M,q^0))\\
&=&\sum_{i=0}^m {(\sigma_qe)_q-(m-i) \choose i}\sum_{j=0}^{m-i}(-1)^{m-i-j}{\alpha_q-e_q-j \choose m-i-j}\chi(\Gr_{e+js_q}(M))\\
&=&\sum_{j=0}^m \chi(\Gr_{e+js_q}(M))\sum_{i=0}^{m-j}(-1)^{m-i-j}{(\sigma_qe)_q-(m-i) \choose i}{\alpha_q-e_q-j \choose m-i-j}.
\end{eqnarray*}}
Set $n=(\sigma_qe)_q$ and $t=\alpha_q-e_q$. First assume that $n\geq t$. Applying Lemma \ref{binolem}, we obtain:
{\allowdisplaybreaks\begin{eqnarray*}
\chi(\Gr_{\sigma_qe-ms_q}(\sigma_qM))&=& \sum_{j=0}^m \frac{(t-j)!(n-t)!}{(n-m)!(m-j)!}(-1)^{m-j}\chi(\Gr_{e+js_q}(M))\\&&\cdot\sum_{i=0}^{m-j}(-1)^{i}{m-j\choose i}{n-j-(m-j-i) \choose n-j-(t-j)}\\
&=&\sum_{j=0}^m \chi(\Gr_{e+js_q}(M))(-1)^{2(m-j)}\frac{(t-j)!(n-t)!}{(n-m)!(m-j)!}\frac{(n-m)!}{(t-j)!(n-t-m+j)!}\\
&=&\sum_{j=0}^m \chi(\Gr_{e+js_q}(M)){n-t\choose m-j}.
\end{eqnarray*}}
If $n<t$, again applying Lemma \ref{binolem}, we obtain
{\allowdisplaybreaks\begin{eqnarray*}
\chi(\Gr_{\sigma_qe-ms_q}(\sigma_qM))&=& \sum_{j=0}^m \frac{(t-j)!}{(n-m)!(m-j)!}(-1)^{m-j}\chi(\Gr_{e+js_q}(M))\\&&\cdot\sum_{i=0}^{m-j}(-1)^{i}{m-j\choose i} \frac{(n-j-(m-j-i))!}{(t-(m-j-i))!}\\
&=&\sum_{j=0}^m \frac{(-1)^{m-j}(t-j)!}{(n-m)!(m-j)!}\frac{(n-m)!(t-n+m-j-1)!}{(t-j)!(t-n-1)!}\chi(\Gr_{e+js_q}(M))\\
&=&\sum_{j=0}^m (-1)^{m-j}\chi(\Gr_{e+js_q}(M)){t-n+m-j-1\choose m-j}\\&=&\sum_{j=0}^m \chi(\Gr_{e+js_q}(M)){n-t\choose m-j}.
\end{eqnarray*}}
This completes the proof of the theorem.
\end{proof}

We want to investigate how the $F$-polymonial of a representation changes when applying BGP-reflections. Assume that $\Gr_e(M)=\Gr_e(M,q^0)$. By Theorem \ref{bgpeuler}, we know that $$\chi(\Gr_{e+rs_q}(M)){n-t\choose i}$$ contributes to the coefficient of $x^{\sigma_q(e)-(r+i)s_q}$ for $r=0,\ldots, t$ and $i=0,\ldots n-t$. In other words for the coefficient of $x^{\sigma_q(e)-(r+i)s_q}$ in $F_{\sigma_q M}(x)$ we get 
\[\sum_{i=0}^{n-t}{n-t\choose i}\chi(\Gr_{e+rs_q}(M))x^{\sigma_q(e)-(r+i)s_q}=x^{\sigma_q(e+rs_q)}\chi(\Gr_{e+rs_q}(M))(1+x_q^{-1})^{n-t}.\]
Define
\[\sigma_q(x^e):=x^{\sigma_q(e)}(1+x_q^{-1})^{\sigma_q(e)_q+e_q}.\]
Recall that $a(p,q)$ was defined as the number of arrows from $p$ to $q$. 
Finally, we re-obtain \cite[Lemma 5.2]{dwz2} in the case where $q$ is a sink:

\begin{thm}\label{refpoly}The following holds:
\begin{enumerate}
\item Let $q$ be a sink. We have 
\[F_{\sigma_qM}(x)=(1+x_q^{-1})^{-\dim M_q}\sum_{e\in\N Q_0}\chi(\Gr_e(M))\sigma_q(x^e)=(1+x_q^{-1})^{-\dim M_q}F_M(x')\]
where
\[x'_i=\begin{cases}x_i^{-1}\text{ if } i=q\\x_ix_q^{a(i,q)}(1+x_q^{-1})^{a(i,q)}\text{ if }i\neq q\end{cases}\]
\item Let $q$ be a source. We have 
\[F_{\sigma_qM}(x)=(1+x_q^{-1})^{(\sigma_q\udim M)_q}F_M(x')\]
where
\[x'_i=\begin{cases}x_i^{-1}\text{ if } i=q\\x_ix_q^{a(q,i)}(1+x_q)^{-a(q,i)}\text{ if }i\neq q\end{cases}\]
\end{enumerate}

\end{thm}

\begin{proof}
The first part follows from the results of this section.
Moreover, we have $F_{\sigma_qM}=F_{(\sigma_qM^\ast)^\ast}$. Since $q$ is a sink of $Q^\op$ and keeping in mind that $F_{M^\ast}(x_p\mid p\in Q_0)=x^{\udim M}F_{M}(x_p^{-1}\mid p\in Q_0)$, the statement is straightforward consequence of the first part.
\end{proof}
\begin{rem} 
If we consider the quiver with one vertex and if $M$ is the semi-simple representation of dimension vector $n$ we get the generating function of the usual Grassmannian
\[F_{M}(x)=\sum_{k=0}^n{n\choose k}x^k.\]
We have $\sigma_q(x^k)=x^{-k}(1+x^{-1})^{-n}$ and thus
\[F_{\sigma_qM}(x)=F_0=\sum_{k=0}^n{n\choose k}x^{-k}(1+x^{-1})^{-n}=1.\]
\end{rem}

\section{Generating functions of Euler characteristics of quiver Grassmannians of type $\tilde D_n$}\label{secFpoly}
\noindent The main aim of this section is to develop explicit formulae for the generating functions of Euler characteristics of quiver Grassmannians (resp. $F$-polynomials) of representations of quivers of type $\tilde D_n$. This reduces to counting certain subsets of the vertex set of coefficient quivers of the respective representations. We first derive formulae for the generating functions of indecomposable representations of small defect. To do so, we initially restrict to subspace orientation and generalize the obtained formulae by applying BGP-reflections. By Theorem \ref{Fdef2}, this can be used to obtain formulae for all indecomposable representations of a quiver of type $\tilde D_n$. Since we have $F_{M\oplus N}=F_MF_N$ for two representations $M$ and $N$, see \cite[Corollary 3.7]{cc}, we obtain formulae for all representations of $Q$. Throughout this section, we frequently use the notation of section \ref{bgp}.

\begin{rem}
The following observation is trivial, but crucial for the considerations of this section: if $F,G\in k[x_q\mid q\in Q_0]$ and $x_q\mapsto x_q'$, where $x_q'\in k[x_q\mid q\in Q_0]$, is a variable transformation, we have
\[(FG)(x')=F(x')G(x')\text{ and }F(x')+G(x')=(F+G)(x').\]
In many cases, we can use this to transfer a factorization or a formula for the generation function of a representation $M$ to one of $M'$ which is obtained from $M$ by applying the reflection functor or the methods from section \ref{redone}.
\end{rem}

Recall that for a sink (resp. source) $q\in Q_0$ we defined $\sigma_q x^d=(x')^d$ where $x'$ is obtained by the variable transformation of Theorem \ref{refpoly}. Moreover, this extends to $F_{\sigma_qM}=(1+x_q^{-1})^{-\dim M_q}\sigma_qF_M$. We frequently use the following lemma:

\begin{lemma}\label{hilfslemma}The following holds: 
\begin{enumerate}
\item Let $d\in\N Q_0$. For a sink $q\in Q_0$, we have 
\[\sigma_qx^d=(1+x_q^{-1})^{\sum_{p\in Q_0}a(p,q)d_p}x_q^{(\sigma_qd)_q}\prod_{\substack{p\in Q_0\\p\neq q}}x^{d_p}=(1+x_q^{-1})^{\sum_{p\in Q_0}a(p,q)d_p}x^{\sigma_qd}.\]
\item Let $q$ be sink of $Q$. Then for every indecomposable representation $M$, we have 
\[\dim(\tau M)_q=\sum_{p\in Q_0}a(p,q)\dim M_p-\dim M_q.\]
\end{enumerate}
\end{lemma}

\begin{proof}
The first part is just a reformulation of the definition. The second statement follows because the Auslander-Reiten translate can be obtained by any admissible sequence of BGP-reflections at sinks. 
\end{proof}

\begin{rem}
If a dimension vector $\alpha\in\N Q_0$ is a root of $Q$, it is a root with respect to all orientations of the arrows. Though the $F$-polynomial $F_\alpha$ depends on the chosen orientation of $Q$, we opt to suppress it from the notation. Note that also the $F$-polynomial $F_\delta$ of a representation from a homogeneous tube depends on the orientation.
\end{rem}


\subsection{Reduction steps and generating functions}\label{redF}
\noindent In this section, we analyze the behavior of the generating functions under reduction of type one, i.e. we have an indecomposable representation $M$ of dimension $\alpha$ such that $\alpha_{i}=\alpha_{i+1}=\alpha_{i+2}+1$. Let $M$ be of maximal rank type and $\hat M$ and $\hat e$ the induced representation and induced dimension vector, respectively. According to section \ref{reductionofgrass},  when removing the vertex $i+1$, for the Euler characteristic
we get \[\chi(\Gr_{e}(M))=\chi(\Gr(e_{i}-e_{i+1},e_{i}-e_{i+2}))\chi(\Gr_{\hat e}(\hat M))={e_{i}-e_{i+2}\choose e_{i}-e_{i+1}}\chi(\Gr_{\hat e}(\hat M)).\]
This yields the following easy relation between the corresponding $F$-polynomials.

\begin{lemma}\label{verlaengern}
Let $M$ be a representation of $\tilde D_n$ which can be reduced to a representation $\hat M$ by reduction of type one. Then we have
\[F_M(x)=\sum_{\hat e\in\N Q_0} \chi(\Gr_{\hat e}(\hat M))x^{\hat e}x_{i+1}^{\hat e_{i+2}}(1+x_{i+1})^{\hat e_{i}-\hat e_{i+2}}.\]
In other words, considering the variable transformation $x_q\mapsto x_q'$ where
\[x_i':=x_i(1+x_{i+1}),\,x_{i+2}':=x_{i+2}x_{i+1}(1+x_{i+1})^{-1},\,x_q'=x_q\,\text{   for all }q'\notin\{i,i+2\},\]
we have $F_M(x)=F_{\hat M}(x').$ 

Moreover, we obtain an analogous statement for the second instance of reduction of type one.
\end{lemma}

Finally, in order to pass to preinjective representations, we can pass to the opposite quiver and dual representations. On the level of $F$-polynomials this can be described by the following formula:

\begin{lemma}Let $M$ be a representation of $Q$. Then for the $F$-polynomial of the dual representation $M^\ast$ we have
\[F_{M^\ast}(x)=x^{\udim  M}F_M(x')\]
where $x'_q=x_q^{-1}$ for every $q\in Q_0$.
\end{lemma}

\subsection{Counting admissible subsets}\label{admissible1}
In order to determine $F$-polynomials for any orientation of $\tilde D_n$, we first determine the $F$-poly\-nomials for representations of $\tilde D_n$ in subspace orientation. Applying BGP-reflections, we obtain the corresponding formula for every orientation. To do so and to fix notation, we proceed with recalling some well known procedure which can be used to obtain the generating functions explicitly. 

Let $f_0,\,f_1\in k[x_i\mid i\in I]$ and let $f_j$ for $j\geq 2$ be recursively defined by
\[\begin{pmatrix}f_{2n}\\f_{2n+1}\end{pmatrix}=\begin{pmatrix}a&b\\c&d\end{pmatrix}\begin{pmatrix}f_{2n-2}\\f_{2n-1}\end{pmatrix}\]
for some $a,b,c,d\in k[x_i\mid i\in I]$ and where $n\geq 1$. The eigenvalues of $A:=\begin{pmatrix}a&b\\c&d\end{pmatrix}$ are the zeroes of $\chi_A=(X-a)(X-d)-bc$, i.e.
\[\lambda_{\pm}=\frac{a+d}{2}\pm\sqrt{\frac{(a+d)^2}{4}-ad+bc}.\]
Define $z:=\sqrt{\frac{(a+d)^2}{4}-ad+bc}$. We have $\lambda_+\lambda_-=ad-bc$ and $\lambda_+-\lambda_-=2z$.
Assuming that $z\neq 0$ and $b\neq 0$, for the eigenspaces, we get
\[E_{\lambda_{\pm}}=\langle\begin{pmatrix}-b\\a-\lambda_{\pm}\end{pmatrix}\rangle.\]
For \[T:=\begin{pmatrix} -b&-b\\a-\lambda_{+}&a-\lambda_{-}\end{pmatrix}, \text{ we have } 
T^{-1}=\frac{1}{\det T}\begin{pmatrix} a-\lambda_-&b\\\lambda_+-a&-b\end{pmatrix}.\]
Then we have
{\allowdisplaybreaks\begin{eqnarray*}\begin{pmatrix}f_{2n}\\f_{2n+1}\end{pmatrix}&=&\begin{pmatrix}a&b\\c&d\end{pmatrix}^n\begin{pmatrix}f_{0}\\f_{1}\end{pmatrix}\\
&=&\frac{1}{-2bz}\begin{pmatrix} -b&-b\\a-\lambda_{+}&a-\lambda_{-}\end{pmatrix}\begin{pmatrix}\lambda_+&0\\0&\lambda_-\end{pmatrix}^n\begin{pmatrix} a-\lambda_-&b\\\lambda_+-a&-b\end{pmatrix}\begin{pmatrix}f_0\\f_1\end{pmatrix}\\
&=&\frac{1}{-2bz}\begin{pmatrix}b(\lambda_-^n(a-\lambda_+)-\lambda_+^n(a-\lambda_-))&b^2(\lambda_-^n-\lambda_+^n)\\(a-\lambda_+)(a-\lambda_-)(\lambda_+^n-\lambda_-^n)&b(\lambda_+^n(a-\lambda_+)-\lambda_-^n(a-\lambda_-))\end{pmatrix}\begin{pmatrix}f_0\\f_1\end{pmatrix}.
\end{eqnarray*}}
Thus we get
\begin{align}\label{genfunc} f_{2n}=\frac{1}{2z}((a(\lambda_+^n-\lambda_-^{n})-(ad-bc)(\lambda_+^{n-1}-\lambda_-^{n-1}))f_0-b(\lambda_-^n-\lambda_+^n)f_1)\end{align}
and
\begin{align}\label{genfunc2} f_{2n+1}=\frac{-1}{2bz}((a-\lambda_+)(a-\lambda_-)(\lambda_+^n-\lambda_-^n)f_0+b(\lambda_+^n(a-\lambda_+)-\lambda_-^n(a-\lambda_-))f_1)\end{align}

The quiver $A_m=0\leftarrow 1\leftarrow \ldots\leftarrow m$ appears as a subquiver of $\tilde D_n$ and the coefficient quiver of the real root representation of dimension $\mathds{1}_m=(1,1,\ldots,1)$ as a subset of the vertex set of the coefficient quivers under consideration. If $X_{\mathds{1}_m}$ is the corresponding indecomposable representation, it is straightforward to check that we have
\[F_m:=F_{X_{\mathds{1}_m}}=\sum_{i=-1}^m\prod_{j=0}^ix_j.\]
Moreover, let $F_{1,m}:=\sum_{i=0}^m\prod_{j=1}^ix_j$. By an easy induction, the following can be proved:

\begin{lemma}\label{gen}
We have
\begin{enumerate}
\item $\prod_{i=0}^m\begin{pmatrix}1&0\\1&x_i\end{pmatrix}=\begin{pmatrix}1&0\\F_{m-1}&\prod_{i=0}^mx_i\end{pmatrix}$
\item $\prod_{i=0}^{m-1}\begin{pmatrix}0&1\\-x_{m-i}&x_{m-i}+1\end{pmatrix}=\begin{pmatrix}-F_{1,m-1}+1&F_{1,m-1}\\-F_{1,m}+1&F_{1,m}\end{pmatrix}$
\end{enumerate}
\end{lemma} 

In order to determine generating functions of preprojectives of defect $-1$, we consider the following snake-shaped coefficient quiver $\mathcal Q(s,n)$ (where $t:=2n-2$) where we omit the vertices $q_i$ in the notation:

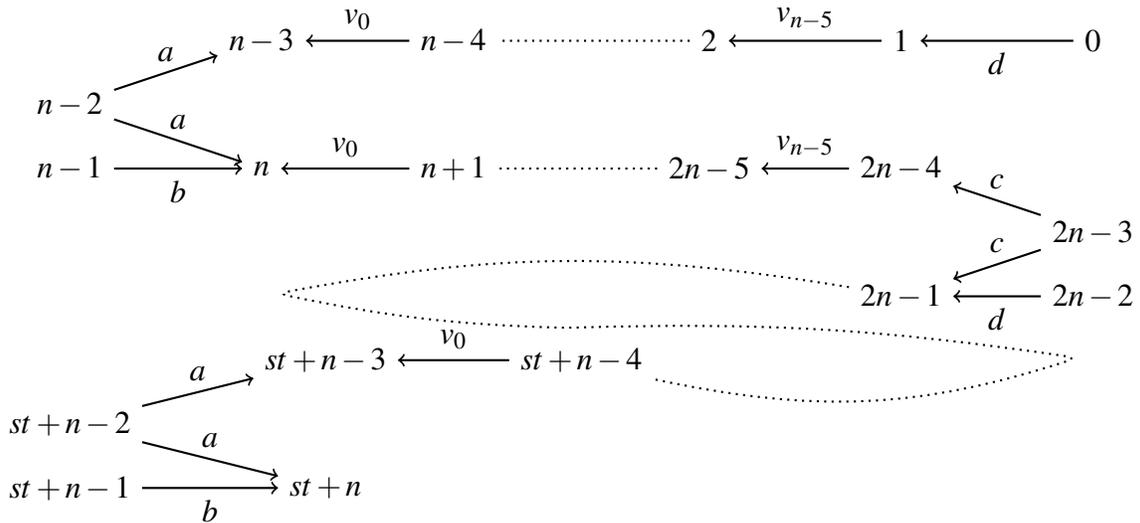
\begin{tikzpicture}[scale=0.85,thick]
\node (A) at (16,2) {$0$};
\node (B) at (13,2) {$1$};
\node (C) at (10,2) {$2$};
\node (D) at (6,2) {$n-4$};
\node (E) at (3,2) {$n-3$};
\node (F) at (0,1) {$n-2$};
\node (G) at (0,0) {$n-1$};
\node (H) at (3,0) {$n$};
\node (D1) at (6,0) {$n+1$};
\node (B1) at (13,0) {$2n-4$};
\node (C1) at (10,0) {$2n-5$};
\node (A2) at (16,-1) {$2n-3$};
\node (A1) at (16,-2) {$2n-2$};
\node (B2) at (13,-2) {$2n-1$};
\node (Z) at (3,-2){$~$};
\node (Z0) at (8,-2.5){$~$};
\node (Z1) at (16,-3){$~$};
\node (D2) at (8,-3) {$st+n-4$};
\node (E2) at (4,-3) {$st+n-3$};
\node (F2) at (0,-4) {$st+n-2$};
\node (G2) at (0,-5) {$st+n-1$};
\node (H2) at (4,-5) {$st+n$};
\draw[->](A) to node[below]{$d$}(B); 
\draw[->](B) to node[above]{$v_{n-5}$}(C);
\draw[-,dotted](C) to (D);
\draw[->] (D) to node[above]{$v_0$}  (E) ;
\draw[->](F) to node[above]{$a$}(E);
\draw[->](F) to node[above]{$a$}(H);
\draw[->](G) to node[below]{$b$}(H);
\draw[->](D1) to node[above]{$v_{0}$}(H);
\draw[->](A2) to node[above]{$c$}(B1); 
\draw[-,dotted](C1) to (D1);
\draw[->](A1) to node[below]{$d$}(B2);
\draw[->](A2) to node[above]{$c$}(B2);
\draw[->](B1) to node[above]{$v_{n-5}$}(C1);
\draw[->](D2) to node[above]{$v_{0}$}(E2) ;
\draw[->](F2) to node[above]{$a$}(E2) ;
\draw[->](F2) to node[above]{$a$}(H2) ;
\draw[->](G2) to node[below]{$b$}(H2) ;
\draw[rounded corners,dotted] (B2) to[bend right=10] (Z) to[bend right=5] (Z0)to[bend left=5]  (Z1) to[bend left=15]  (D2);
\end{tikzpicture}

We will see that we can basically restrict our calculations to this case. Also the case of the exceptional tubes can be reduced to this case. We refer to the corresponding preprojective representation by $M(s,n)$.

We call a subgraph a ramification subgraph if it is of the following form:

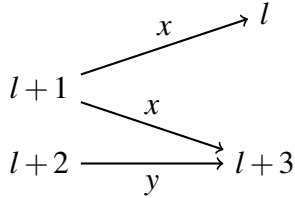
\begin{tikzpicture}[thick]
\node (E) at (3,2) {$l$};
\node (F) at (0,1) {$l+1$};
\node (G) at (0,0) {$l+2$};
\node (H) at (3,0) {$l+3$};
\draw[->](F) to node[above]{$x$}(E);
\draw[->](F) to node[above]{$x$}(H);
\draw[->](G) to node[below]{$y$}(H);
\end{tikzpicture}

Note that in our case we have $x\in\{a,c\}$ and $y\in\{b,d\}$. In this situation, the extremal arrows of $\mathcal Q(s,n)$ defined in section \ref{coeffquiver}  are all arrows but those of the form $l+1\xlongrightarrow{x} l$ contained in the ramification subgraphs. 

\begin{df}\label{admsub}
We call a subset $G_0$ of $\mathcal Q(s,n)_0$ admissible if the following holds:
\begin{enumerate}
\item $G_0$ is extremal successor closed, i.e. if the tail of an extremal arrow is contained in $G_0$, the head is also contained in $G_0$.
\item For all ramification subgraphs, we have: if $l+1,\,l+2\in G_0$, then $l\in G_0$.
\end{enumerate}
\end{df}

Note that we automatically have $l+3\in G_0$ if $G$ is extremal successor closed and if $l+1\in G_0$ or $l+2\in G_0$. Every subset induces a dimension vector $e\in\N Q_0$, called the type of $G_0$ in what follows. The next step is to determine the number of admissible subsets of $\mathcal Q(s,n)_0$ of a fixed type $e$ for the following reason. By Theorem \cite[Theorem 4.4]{LW}, for representations of defect $-1$, the non-empty Schubert cells correspond to the non-contradictory subsets $\beta$, see section \ref{noncontradictory}. In the present case, this translates into the notion of admissibility.

\begin{thm}\label{admissible}
Let $e\in\N Q_0$. Then $\chi(\Gr_e(M(s,n)))$ is the number of admissible subsets of $\mathcal Q(s,n)_0$ of type $e$. 
\end{thm}
\begin{rem}\label{transfer}
In order to determine the $F$-polynomials for indecomposables lying in exceptional tubes, we need to consider slight modifications of the coefficient quiver $\mathcal Q(s,n)$ for $n\leq 5$. The notions of extremal arrows and admissible subsets are the same as before, i.e. the extremal arrows are all but those of the form $l+1\xrightarrow{x} l$ contained in the ramification subgraphs and Definition \ref{admsub} can be transferred word by word.
\end{rem}

Consider $I:=\{0,\ldots,2s+1\}$ and $J:=\{0,\ldots,n-4\}$. If we delete the sources of $\mathcal Q(s,n)$ corresponding to the ramification subgraphs, we can think of the remaining graph as a matrix having entries which are vertices, i.e. with every index $(i,j)$ we associate the vertex in the $i$th row and $j$th column of the remaining graph. Note that we start the indexing by $(0,0)$. 

For $(i,j)\in I\times J$, let $\mathcal G(i,j)$ be the full (connected) subgraph of $\mathcal Q(s,n)$ which has vertices $\{(0,n-4), (0,n-5),...,(i,j)\}$ and where we
 add the subgraph $1\leftarrow 0$ and also all sources of ramification subgraphs whose remaining vertices are all contained in $\{(0,n-4), (0,n-5),...,(i,j)\}$.
Let 
\[\mathcal F_i^j=\sum_{e\in\N Q_0}\chi(i,j,e)x^e\] be the generating function counting the number $\chi(i,j,e)$ of admissible subsets of $\mathcal G(i,j)_0$ of type $e$. We define $\mathcal F_{-1}^{n-4}:=1$ and $\mathcal F_0^{n-4}:=1+x_{n-4}+x_{n-4}x_d$.

\begin{lemma}\label{genrecursion}
We have the following recursive relations:
\begin{enumerate}
\item For all $m\geq 0, \,j=n-5,\ldots,0$, we have $\mathcal F_{2m}^j=x_j\mathcal F^{j+1}_{2m}+\mathcal F_{2m-1}^{n-4}$.
\item For all $m\geq 0$, we have 
\[\mathcal F_{2m+1}^{0}=(1+x_{0}+x_{0}x_{a}+x_{0}x_{b}+x_{0}x_{a}x_{b})\mathcal F_{2m}^{0}-x_{0}x_{a}x_{b}\mathcal F_{2m-1}^{n-4}.\]
\item For all $m\geq 0$, we have $\mathcal F_{2m+1}^1=(x_1+1)\mathcal F^{0}_{2m+1}-x_1\mathcal F_{2m}^{0}$.
\item For all $m\geq 0,\,j=2,\ldots,n-4$, we have $\mathcal F_{2m+1}^j=(x_j+1)\mathcal F^{j-1}_{2m+1}-x_j\mathcal F_{2m}^{j-2}$.
\item For all $m\geq 1$, we have \[\mathcal F_{2m}^{n-4}=(1+x_{n-4}+x_{n-4}x_{c}+x_{n-4}x_{d}+x_{n-4}x_{c}x_{d})\mathcal F_{2m-1}^{n-4}-x_{n-4}x_{c}x_{d}\mathcal F_{2m-1}^{n-5}.\]
\end{enumerate}
\end{lemma}

\begin{proof}
An admissible subset of $\mathcal G(2m,j)_0$ is obtained from one of $\mathcal G(2m,j-1)_0$ by adding the vertex corresponding to the index $(2m,j)$ or it is given by an admissible subset $\mathcal G(2m-1,n-4)_0$. Note that if $(2m,i)$ is a vertex of an admissible subset, then $(2m,i-1)$ is forced to be part of the admissible subset because it is extremal successor closed for $i=n-4,\ldots,1$. Thus we obtain the first statement. The third and forth statements can be obtained similarly.

An admissible subset of $\mathcal G(2m+1,0)_0$ is obtained by adding an admissible subset of the ramification subgraph which is glued. This corresponds to the first summand in the second statement. But because of the second property we have to drop those subsets containing the vertices $sm+n-2,\,sm+n-1,sm+n$ but not containing $sm+n-3$. This gives the second summand. The last statement can be obtained by a similar argument. 
\end{proof}

Let $H(x,y,z)=1+x+xy+xz+xyz$. Together with the observations in the beginning of this section, Lemmas \ref{gen} and \ref{genrecursion} give rise to the following recursive description of the generating functions:

\begin{cor}
\begin{eqnarray*}\begin{pmatrix}\mathcal F_{2m+1}^{n-4}\\\mathcal F_{2m+2}^{n-4}\end{pmatrix}&=&\begin{pmatrix}0&1\\-x_{n-4}x_{c}x_{d}&H(x_{n-4},x_{c},x_{d})\end{pmatrix}\begin{pmatrix}- F_{1,n-5}+1& F_{1,n-5}\\- F_{1,n-4}+1& F_{1,n-4}\end{pmatrix}\\&&
\begin{pmatrix}0&1\\-x_{0}x_{a}x_{b}&H(x_{0},x_{a},x_{b})\end{pmatrix}\begin{pmatrix}1&0\\F_{n-6}& \prod_{i=0}^{n-5}x_i\end{pmatrix}\begin{pmatrix}\mathcal F_{2m-1}^{n-4}\\\mathcal F_{2m}^{n-4}\end{pmatrix}\end{eqnarray*}
For $n=4$, we get
\begin{eqnarray*}\begin{pmatrix}\mathcal F^4_{2m+1}\\\mathcal F^4_{2m+2}\end{pmatrix}&=&\begin{pmatrix}0&1\\-x_{0}x_{c}x_{d}&H(x_{0},x_{c},x_{d})\end{pmatrix}\begin{pmatrix}0&1\\-x_{0}x_{a}x_{b}&H(x_{0},x_{a},x_{b})\end{pmatrix}\begin{pmatrix}\mathcal F^4_{2m-1}\\\mathcal F^4_{2m}\end{pmatrix}\\&=&\begin{pmatrix}-x_0x_ax_b&H(x_0,x_a,x_b)\\-x_0x_ax_bH(x_0,x_c,x_d)&-x_0x_cx_d+H(x_0,x_c,x_d)H(x_0,x_a,x_b)\end{pmatrix}\begin{pmatrix}\mathcal F^4_{2m-1}\\\mathcal F^4_{2m}\end{pmatrix}\end{eqnarray*}
\end{cor}

With this tool in hand we can determine the $F$-polynomials of representations of $\tilde D_n$ explicitly using formulae (\ref{genfunc}), (\ref{genfunc2}). This is done in the following subsections. We begin with the representations of the tubes as it turns out that the proofs are less technical.

\subsection{The homogeneous tube}\label{homotube}
The $F$-polynomials of the representations of the homogeneous tubes play an important role in the following. The imaginary Schur root $\delta$ of $\tilde D_n$ is independent of the orientation of $\tilde D_n$. Using the methods of section \ref{redone} and \ref{bgp}, it is straightforward to check that the $F$-polynomial of a representation of dimension $\delta$ lying in one of the homogeneous tubes does not depend on the tube. Alternatively, one can apply \cite[Lemma 5.3]{Dupont11} or \cite[Theorem 4.4]{LW}. Thus we can fix a homogeneous tube without loss of generality. We denote the unique representation of dimension $r\delta$ in this tube by $M_{r\delta}$ and define $F_{r\delta}:=F_{M_{r\delta}}$. Moreover, for every $r\geq 1$ there exists an almost split sequence of the form
\[\ses{M_{r\delta}}{M_{(r-1)\delta}\oplus M_{(r+1)\delta}}{M_{r\delta}}\]
where $M_{0\delta}=M_0:=0$ and $F_{M_{0}}=1$. By applying Theorem \ref{almostsplit}, we obtain

\begin{lemma}\label{formhom}
The $F$-polynomial of representations lying in one of the homogeneous tubes depends only on the dimension vector and satisfies the recursion
\[F_{(r+1)\delta}F_{(r-1)\delta}=F_{r\delta}^2-x^{r\delta}\]
for $r\geq 1$. \hfill\qed
\end{lemma}
By the methods of section \ref{admissible1}, we also obtain an explicit formula for $F_\delta$. Recall Corollary \ref{homogform1}, saying that
\[F_{r\delta}=F_{\delta}F_{(r-1)\delta}-x^{\delta}F_{(r-2)\delta}\]
for $r\geq 1$ where $F_0=1$ and $F_{-\delta}=0$.
This yields
\[\begin{pmatrix}F_{r\delta}\\F_{(r+1)\delta}\end{pmatrix}=\begin{pmatrix}0&1\\-x^{\delta}&F_{\delta}\end{pmatrix}^{r+1}\begin{pmatrix}0\\1\end{pmatrix}
\]
Defining
\[z=\frac{1}{2}\sqrt{F_{\delta}^2-4x^{\delta}},\quad\lambda_{\pm}=\frac{F_{\delta}}{2}\pm z,\]
we thus get the following explicit formula:

\begin{cor}\label{homogform2}We have
\[F_{r\delta}=\frac{1}{2z}(\lambda_+^{r+1}-\lambda_-^{r+1}).\]
\end{cor}

\subsection{The exceptional tubes of rank two}
\noindent In this section, we apply the developed methods to representations lying in the  exceptional tubes of rank two. To do so we first restrict to $\tilde D_4$ in subspace orientation. Afterwards, we extend the results to $\tilde D_n$ in subspace orientation and, finally, to any orientation. 

If $\alpha$ is a real root let $F_{\alpha}:=F_{M_{\alpha}}$. Similar to the case of preprojective representations of defect $-1$, we obtain all coefficient quivers of representations lying in this tube by glueing the coefficient quivers
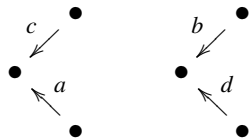
\[\begin{xy}\xymatrix@R10pt@C10pt{&&\bullet&&&\bullet\\&\bullet\ar^c@{<-}[ru]\ar^a@{<-}[rd]&&&\bullet\ar@{<-}^b[ru]\ar@{<-}^d[rd]\\&&\bullet&&&\bullet&}\end{xy}\]
This means that (up to permutation of the arrows $a,b,c,d$), we consider the coefficient quivers $\mathcal Q(s,4)$ with an extra arrow $-1\xrightarrow{b} 1$. Also the notion of non-contradictory subsets transfers exactly to the one of admissible subsets, see also Remark \ref{transfer}. We denote the representation on the left hand side by $T_1$ and the representation on the right hand side by $T_2$. Then we get for the generating functions
\[F_{T_1}=1+x_0+x_0x_a+x_0x_c+x_0x_ax_c,\quad F_{T_2}=1+x_0+x_0x_b+x_0x_d+x_0x_bx_d.\]
Without loss of generality we can assume that we start our glueing process with the coefficient quiver of $T_1$. Now we proceed completely analogous to subsection \ref{admissible1} and apply Theorem \cite[Theorem 4.4]{LW} to obtain $f_{-1}=0$, $f_0=1$ and 
\[f_{2r+1}=F_{T_1}f_{2r}-x_0x_ax_cf_{2r-1},\quad f_{2r+2}=F_{T_2}f_{2r+1}-x_0x_bx_df_{2r}\]
where $f_{2r+1}$ is the generating function corresponding to the unique indecomposable of dimension $t(r):=\udim  T_1+r\cdot\delta$ and $f_{2r+2}$ is the generating function corresponding to the unique indecomposable $M^1_{(r+1)\delta}$ of dimension $(r+1)\cdot\delta$ such that $T_1\subset M^1_{(r+1)\delta}$. Note that also the recursion is up to permutation of arrows basically the same as the one in subsection \ref{admissible1}. The only difference is that we start our glueing process in the present situation with the empty coefficient quiver while in subsection \ref{admissible1}, we start with the coefficient quiver $\bullet\xlongleftarrow{d} \bullet$.

\begin{prop}\label{genfuncfac}
For $r\geq 0$, we have
\[F_{t(r)}=f_{2r+1}=F_{T_1}F_{r\delta}=F_{t(0)}F_{r\delta},\]
\[F_{M^1_{(r+1)\delta}}=f_{2r+2}=F_{(r+1)\delta}+x_0x_ax_cF_{r\delta}\]
\end{prop}

\begin{proof}
Using the notation from subsection \ref{admissible1}, we have
\[a=-x_0x_ax_c,\quad b=F_{T_1},\quad c=-x_0x_ax_cF_{T_2},\quad d=-x_0x_bx_d+F_{T_1}F_{T_2}.\]
Then it is easy to check that we have
$$a+d=F_{\delta},\quad z=\frac{1}{2}\sqrt{F_{\delta}^2-4x^{\delta}},\quad ad-bc=\lambda_+\lambda_-=x^{\delta}.$$
Moreover, we get
$$\lambda_+=\frac{1}{2}(F_{\delta}+\sqrt{F_{\delta}^2-4x^{\delta}}),\quad \lambda_-=\frac{1}{2}(F_{\delta}-\sqrt{F_{\delta}^2-4x^{\delta}}).$$
Since $f_{-1}=0$ and $f_0=1$, equation (\ref{genfunc}) yields
\[ f_{2r+1}=\frac{1}{2z}(F_{T_1}(\lambda_+^{r+1}-\lambda_-^{r+1})).\]
Since $\lambda_{\pm}$ and $z$ take the same values as in section \ref{homotube}, by Corollary \ref{homogform2}, we get $f_{2r+1}=F_{T_1}F_{r\delta}$.
Using Equation (\ref{genfunc2}) together with Corollary \ref{homogform2}, we have
\begin{eqnarray*} f_{2r+2}&=&\frac{1}{2z}(\lambda_-^{r+1}(-x_0x_ax_c-\lambda_-)-(\lambda_+^{r+1}(-x_0x_ax_c-\lambda_+)))\\
&=& \frac{1}{2z}(\lambda_+^{r+2}-\lambda_-^{r+2})+\frac{1}{2z}x_0x_ax_c(\lambda_+^{r+1}-\lambda_-^{r+1})\\
&=&F_{(r+1)\delta}+x_0x_ax_cF_{r\delta},
\end{eqnarray*}
which completes the proof of the proposition.
\end{proof}

Let us consider the tubes of rank two for general $n$ with arbitrary orientation. For a fixed tube, we denote by $t_1(0)$ and $t_2(0)$ the quasi-simple roots. The real roots in this tube are given by $t_i(r)=t_i(0)+r\delta$. Finally, we denote the representation of dimension $r\delta$ with subrepresentation $M_{t_i(0)}$ by $M_{r\delta}^i$.

\begin{thm}\label{Fpolyrank2}
For the indecomposable representations $M_{t_i(r)}$ and $M_{r\delta}^i$ lying in one of the exceptional tubes of rank two of $\tilde D_n$, we have:
\begin{enumerate}
\item $F_{t_i(r)}=F_{t_i(0)}F_{r\delta}$
\item $F_{M_{r\delta}^i}=F_{r\delta}+x^{t_i(0)}F_{(r-1)\delta}$
\end{enumerate}
\end{thm}

\begin{proof}
Under consideration of Lemma \ref{verlaengern} it is straightforward to generalize Proposition \ref{genfuncfac} to arbitrary $\tilde D_n$ in subspace orientation.

Assume that  $M$ with $\udim M=t_i(r)+r\delta$ lies in one of the exceptional tubes of rank two of $\tilde D_n$ (with arbitrary orientation) and satisfies
$F_{M}=F_{t_i(0)}F_{r\delta}$. Applying Theorem \ref{refpoly}, we have 
\[F_{\sigma_qM}=F_{\sigma_qt_i(0)}F_{r\delta}.\]
Thus the first statement follows by induction.

For a fixed sink $q$, of $\tilde D_n$ with arbitrary orientation, it is straightforward to check that 
\[\sum_{p\in Q_0}a(p,q)t_i(0)_p=\delta_q.\]
Indeed, if $q\in\{a,b,c,d\}$, both sides are one. Otherwise both sides are two. Assume that $F_{M_{r\delta}^i}=F_{r\delta}+x^{t_i(0)}F_{(r-1)\delta}$. Then, again by Theorem \ref{refpoly} and Lemma \ref{hilfslemma}, we have
\[F_{\sigma_qM_{r\delta}^i}=F_{r\delta}+x^{\sigma_qt_i(0)}(1+x_q^{-1})^{\sum_{p\in Q_0}a(p,q)t_i(0)_p}(1+x_q^{-1})^{-\delta_q}F_{(r-1)\delta}=F_{r\delta}+x^{\sigma_qt_i(0)}F_{(r-1)\delta}.\]
Thus the second statement also follows by induction.
\end{proof}

\subsection{The exceptional tube of rank $n-2$}
Let us consider the tube of rank $n-2$. In this tube, there exist $n-2$ indecomposable representations of dimension $r\delta$, which we denote by $M_{r\delta}^i$ for $i=1,\ldots,n-2$. If it is clear which of these representations is considered, we drop the $i$.

We again first restrict to subspace orientation and treat the general case later. Let $d_m(r)$ be the root given by
\[d_m(r)_q=\begin{cases} r\text{ for } q=a,b\\
r-1\text{ for }q=c,d\\
 2r-1\text{ for }q=0,\ldots,m\\2r-2 \text{ for }i=m+1,\ldots,n-4\end{cases}\]
In this tube, there exists an infinite chain of irreducible inclusions 
\[M_{d_0(1)}\hookrightarrow
 M_{d_1(1)}\hookrightarrow\ldots\hookrightarrow M_{d_{n-4}(1)}\hookrightarrow
M_{\delta}\hookrightarrow M_{d_0(2)}\hookrightarrow\ldots\]
where the imaginary root representations $M_{r\delta}$ are uniquely determined by this chain. 

\begin{lemma}\label{n=5}
Let $n=5$. Then we have 
\begin{enumerate}
\item $F_{d_0(r)}=F_{d_0(1)}F_{r\delta}+x_0x_1x_ax_bF_{(r-1)\delta}$
\item $F_{d_1(r)}=F_{d_1(1)}F_{r\delta}$
\item $F_{M_{r\delta}}=F_{r\delta}+x_0x_ax_b(1+x_1)F_{(r-1)\delta}$
\end{enumerate}
\end{lemma}

\begin{proof}
First note that the coefficient quivers of the representations under consideration pointed out in \cite[Appendix B]{LW} are modifications of $\mathcal Q(s,5)$. Moreover, the notion of non-contradictory subsets and admissibility transfers again, see Remark \ref{transfer}. More precisely, we add an extra arrow $-1\rightarrow 1$ in all cases and, moreover, the last vertices of the quivers to consider are $st+1, st+2$ and $st-2$ respectively.

Then the last two statements follow in the same way as Proposition \ref{genfuncfac} where we additionally apply Lemma \ref{verlaengern}. Now we can apply (the methods of) Lemma \ref{genrecursion} in order to obtain
\[F_{d_0(r)}=F_{d_0(1)}F_{M_{r\delta}}-x_0x_ax_bF_{d_1(r-1)}=F_{d_0(1)}F_{r\delta}+x_0x_1x_ax_bF_{(r-1)\delta}.\]
Here we use that $F_{d_0(1)}(1+x_1)-x_1=F_{d_1(1)}$, see also Lemma \ref{genrecursion}.
\end{proof}

For $l\leq m$, we denote by $S_{l,m}$ the exceptional representation with dimension vector $\udim S_{l,m}=\sum_{i=l}^m q_i$. 
Note that
\[\udim S_{m+2,n-4}=d_{n-4}(1)-d_0(1)-\tau^{-1}d_m(1)=d_{n-4}(1)-d_{m+1}(1).\]
Moreover, we have 
$$\prod_{i=1}^{m+1}x_i=x^{\tau^{-1}d_m(1)},\quad x_0x_ax_b=x^{d_0(1)}.$$ 

\begin{lemma}\label{subspace}
For arbitrary $n$, we have 
\begin{enumerate}
\item $F_{d_m(r)}=F_{d_m(1)}F_{r\delta}+x_0x_ax_b\prod_{i=1}^{m+1}x_i F_{S_{m+2,n-4}}F_{(r-1)\delta}$ for $m=0,\ldots,n-5$.
\item $F_{d_{n-4}(r)}=F_{d_{n-4}(1)}F_{r\delta}$ 
\item $F_{M_{r\delta}}=F_{r\delta}+x_0x_ax_bF_{S_{1,n-4}}F_{(r-1)\delta}$
\end{enumerate}
\end{lemma}

\begin{proof}This is obtained when combining Lemma \ref{n=5} and Lemma \ref{verlaengern}.
\end{proof}

For $\tilde D_n$ with arbitrary orientation in the tube of rank $n-2$, there exist $n-2$ chains of irreducible morphisms of the form
\[M_{0,1}\hookrightarrow
 M_{0,2}\hookrightarrow\ldots\hookrightarrow M_{0,n-3}\hookrightarrow
M_{1,0}:=M_{\delta}\hookrightarrow M_{1,1}\hookrightarrow\ldots\]
where the $m_{l}(r):=\udim M_{r,l}$ are real roots and the imaginary root representations $M_{r,0}:=M_{r\delta}$ are uniquely determined by this chain. In particular, for every real root $\alpha$ in the tube of rank $n-2$ there exists an exceptional root $m_l(0)$ such that $\alpha=r\delta+m_l(0)$.
\begin{lemma}\label{hilfslemma2} For $1\leq l\leq n-4$, we have
$\delta=m_{n-3}(0)+\tau m_{l+1}(0)-m_l(0)$.
\end{lemma}

\begin{proof}
Considering the tube and its roots in detail, we obtain 
\[\delta=m_{n-3}(0)+\tau m_{n-3}(0)-m_{n-4}(0).\]
Since we also have
\[m_l(0)+\tau m_{l+2}(0)=m_{l+1}(0)+\tau m_{l+1}(0)\Leftrightarrow \tau m_{l+2}(0)-m_{l+1}(0)=\tau m_{l+1}(0)- m_l(0),\]
the claim follows by induction.
\end{proof}
Under the convention that $F_{\alpha}=0$ if $\alpha\in\Z Q_0$ has at least one negative coefficient, we obtain the following result: 

\begin{thm}\label{Fpolyrankn2}Set $m_0(r)=m_{n-2}(r-1)=r\delta$. Let $M$ be an indecomposable representation of $\tilde D_n$ lying in the tube of rank $n-2$ such that $\udim M=m_l(r)$ for some $l=0,\ldots,n-3$. Then we have  and, moreover, we have
\[F_{m_{l}(r)}=F_{m_l(0)}F_{r\delta}+x^{m_{l+1}(0)}F_{m_{n-3}(0)-m_{l+1}(0)}F_{(r-1)\delta}\]
for $l=0,\ldots, n-3$.
\end{thm}

\begin{proof}
We proceed by induction. For every representation $M$ lying in the tube of rank $n-2$,  there exists a sequence of reflections $\sigma_1,\ldots, \sigma_l$ at sinks and a representation $M_{d_m(r)}$ with $m\in\{-1,\ldots, n-4\}$ such that $M=\sigma_1\ldots \sigma_l M_{d_m(r)}$. Since the claim is true for $M_{d_{m}(r)}$, by Lemma \ref{subspace}, we can assume that $M=\sigma_q M_{r,l}$ and that the claim is true for $M_{r,l}$, i.e. we have
\[F_{m_{l}(r)}=F_{m_l(0)}F_{r\delta}+x^{m_{l+1}(0)}F_{m_{n-3}(0)-m_{l+1}(0)}F_{(r-1)\delta}\]

Then applying Theorem \ref{refpoly} and Lemma \ref{hilfslemma}, we have
\[F_{M}=F_{\sigma_qm_l(0)}F_{r\delta}+(1+x_q^{-1})^{a}(1+x_q^{-1})^{b}x^{\sigma_qm_{l+1}(0)}F_{\sigma_q(m_{n-3}(0)-m_{l+1}(0))}F_{(r-1)\delta}\]
with
\[a=-\delta_q-m_l(0)_q+m_{n-3}(0)_q-m_{l+1}(0)_q,\,b=\sum_{p\in Q_0} a(p,q)m_{l+1}(0)_p.\]
Again by Lemma \ref{hilfslemma} and, moreover, by Lemma \ref{hilfslemma2}, we have 
\[a+b=-\delta_q-m_l(0)_q+m_{n-3}(0)_q+\tau m_{l+1}(0)_q=0,
\]
which completes the proof of the theorem.
\end{proof}
Note that Theorems \ref{Fpolyrank2} and \ref{Fpolyrankn2} can be summarized as done in Theorem C.

\subsection{The preprojectives of small defect}\label{Fprej}
Finally, we consider the preprojective roots.  Also in this case, we obtain explicit formulae. Thanks to Theorem \ref{Fdef2}, the generating functions corresponding to the roots of defect $-2$ can easily be obtained from those of defect $-1$. Moreover, the generating functions for the preinjectives can be calculated when passing to the opposite quiver.

We again denote the projective representations corresponding to the outer vertices by $P_a,P_b,P_c$ and $P_d$. We follow the strategy of the last two subsections and first restrict to subspace orientation.  Up to permutation of the sources, the preprojective roots for $n=4$ are given by
\[d_1(r)=(2r+1,r+1,r,r,r)\text{ and }d_2(r)=(2r,r-1,r,r,r).\]

\begin{prop}\label{subspace2}
Let $n=4$. Then we have
\begin{enumerate}
\item $F_{d_1(r)}=F_{P_a}F_{r\delta}-x^{\udim \tau^{-1}P_a}F_{\delta-\udim \tau^{-1}P_a}F_{(r-1)\delta}$
\item $F_{d_2(r)}=F_{\tau^{-1}P_a}F_{(r-1)\delta}-x^{\delta}F_{(r-2)\delta}$
\end{enumerate}
\end{prop}

\begin{proof}
Initially, we consider the coefficient quiver of the preprojective representation of dimension $(1,1,0,0,0)$, i.e. $\bullet\xrightarrow{a}\bullet$. By glueing the coefficient quivers
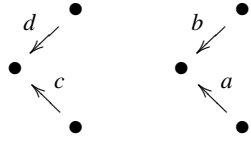
\[\begin{xy}\xymatrix@R10pt@C10pt{&&\bullet&&&\bullet\\&\bullet\ar^d@{<-}[ru]\ar^c@{<-}[rd]&&&\bullet\ar@{<-}^b[ru]\ar@{<-}^a[rd]\\&&\bullet&&&\bullet&}\end{xy}\]
in turn, up to permutation of the arrows, we obtain the coefficient quivers $\mathcal Q(s,4)$ and we can apply Theorem \ref{admissible}. We denote the corresponding representations by $T_1$ and $T_2$ respectively.
Using the notation and results from subsection \ref{admissible1}, we have
\[f_0=1, \,f_1=1+x_0+x_0x_a.\]
Moreover, we have $f_{2r}=F_{d_2(r)}$ and $f_{2r+1}=F_{d_1(r)}$ and 
\[a=-x_0x_cx_d,\quad b=F_{T_1},\quad c=-x_0x_cx_dF_{T_2},\quad d=-x_0x_ax_b+F_{T_1}F_{T_2}.\]
Using $\lambda_+\lambda_-=ad-bc=x^{\delta}$, we also have
\[(a-\lambda_+)(a-\lambda_-)=(a^2+(ad-bc)-a(a+d))=-bc.\]
Since $a,b,c,d$ and hence $z$ and $\lambda_{\pm}$ take the same values as in the proof of Proposition \ref{genfuncfac}, following Corollary \ref{homogform2}, we have
\[F_{r\delta}=\frac{1}{2z}(\lambda_+^{r+1}-\lambda_-^{r+1}).\]
We obtain
\[f_{2r}=-x_0x_cx_dF_{(r-1)\delta}-x^{\delta}F_{(r-2)\delta}+F_{T_1}F_{(r-1)\delta}f_1=F_{\tau^{-1}P_b}F_{(r-1)\delta}-x^{\delta}F_{(r-2)\delta},\]
\begin{eqnarray*}
f_{2r+1}&=&cF_{(r-1)\delta}-(aF_{(r-1)\delta}-F_{r\delta})f_1\\
&=&F_{P_a}F_{r\delta}-x^{\udim \tau^{-1}P_a}F_{\delta-\udim \tau^{-1}P_a}F_{(r-1)\delta}
\end{eqnarray*}
Note that $\delta-\udim \tau^{-1}P_a=\udim S_a$ is the simple root (which is in this case also the injective root respectively) corresponding to the vertex $a$. 
\end{proof}

\begin{rem}Applying the reflection functor to $M_{d_1(r)}$ (resp. $F_{d_1(r)}$), it is straightforward to check that
\[F_{d_2(r)}=(1+x_a^{-1})F_{r\delta}-x^{\udim P_a}F_{\delta-\udim P_a}F_{(r-1)\delta}.\]
Note that $\tau d_1(r)=\hat d_2(r)$, where $\hat d_2(r)$ is defined below.
\end{rem}

The following is straightforward to check:

\begin{lemma}\label{verlaengern2}
Considering the variable transformation of Lemma \ref{verlaengern}, we have
\[(1+x_1)F_{\delta-\udim P_b}(x')=F_{\delta-\udim P_b}(x)\] where the root $\delta - \udim P_b$ on the left hand side is the root $q_0+q_a+q_c+q_d$ of $\tilde D_4$ and on the right hand side it is the root $q_0+2q_1+q_a+q_c+q_d$ of $\tilde D_5$. 
\end{lemma}

For general $n$, the preprojective roots of defect $-1$ are given by:
\begin{itemize}
	\item $d^m_1(r)$, $m=0,\ldots,n-4$, with $d^m_1(r)_{q_i}=2r+1$ for $i=0,\ldots,m$, $d^m_1(r)_{q_i}=2r$ for $i=m+1,\ldots,n-4$, $d_1^m(r)_{a}=d_1^m(r)_{c}=d_1^m(r)_{d}=r$ and $d_1^m(r)_{b}=r+1$.	We denote by $\hat{d}^m_1(r)$ the resulting root obtained when permuting $q_{a}$ and $q_{b}$.
	\item $d_2(r)$ with $d_2(r)_{q_i}=2r$ for $i=0,\ldots,n-4$, $d_2(n)_{a}=d_2(r)_{c}=d_2(r)_{d}=r$ and $d_2(r)_{b}=r-1$. We denote by $\hat{d}_2(r)$ the resulting root obtained when permuting $q_{a}$ and $q_{b}$.
		
	\item $d^m_3(r)$, $m=0,\ldots,n-4$, with $d^m_3(r)_{q_i}=2r$ for $i=0,\ldots,m$, $d^m_3(r)_{q_i}=2r-1$ for $i=m+1,\ldots,n-4$, $d_3^m(r)_{a}=d_3^m(r)_{b}=d_3^m(r)_{c}=r$ and $d_3^m(r)_{d}=r-1$.	We denote by $\hat{d}^m_3(r)$ the resulting root obtained when permuting $q_{c}$ and $q_{d}$.
	
	\item $d_4(r)$ with $d_4(r)_{q_i}=2r+1$ for $i=0,\ldots,n-4$, $d_4(r)_{a}=d_4(r)_{b}=d_4(r)_{c}=r$ and $d_4(r)_{d}=r+1$.	We denote by $\hat{d}_4(r)$ the resulting root obtained when permuting $q_{c}$ and $q_{d}$.

\end{itemize}

Note that we have $\tau^{-1}d^m_1(r)=\hat d_1^{m+1}(r)$ for $m=0,\ldots,n-5$, $\tau^{-1}d^{n+4}_1(r)=\hat d_2(r+1)$ and $\tau^{-1} \hat d_2(r+1)=d^0_1(r+1)$. The analogous statement holds for $d_3^m(r)$ and $\hat d_4(r+1)$.

We denote by $P_m$ the projective corresponding to the vertex $q_m$ and by $s_q$ the simple root corresponding to $q$. We have $d_1^m(0)=\udim P_m+s_a$.

\begin{prop}\label{subspace3}
Let $\tilde D_n$ be in subspace orientation. Then we have
\begin{enumerate}
\item $F_{d^m_1(r)}=F_{d_1^m(0)}F_{r\delta}-x^{\tau^{-1}d^m_1(0)}F_{\delta-\tau^{-1}d^m_1(0)}F_{(r-1)\delta}$
\item $F_{d_2(r)}=F_{d_2(1)}F_{(r-1)\delta}-x^{\delta}F_{(r-2)\delta}$
\item $F_{d^m_3(r)}=F_{d^m_3(1)}F_{(r-1)\delta}-x^{\tau^{-1}d^m_3(1)}F_{\delta-\tau^{-1}d^m_3(1)}F_{(r-2)\delta}$
\item $F_{d^{n-4}_3(r)}=F_{d^{n-4}_3(1)}F_{(r-1)\delta}-x^{\delta}F_{(r-2)\delta}$
\item $F_{d_4(r)}=F_{d_4(0)}F_{r\delta}-x^{\tau^{-1}d_4(0)}F_{\delta-\tau^{-1}d_4(0)}F_{(r-1)\delta}$
\end{enumerate}
\end{prop}

\begin{proof}
Initially, we derive the case $n=5$ when applying Lemma \ref{verlaengern} to the results of Proposition \ref{subspace2}. We first obtain
\[F_{d_1^{1}(r)}=F_{d_1^1(0)}F_{r\delta}-x^{\tau^{-1}d^{1}_1(0)}F_{\delta-\tau^{-1}d^{1}_1(0)}F_{(r-1)\delta}.\]
We have $\tau \hat d_1^1(r)= d_1^0(r)$ where $\tau \hat d_1^1(r)$ is obtained from $ d_1^1(r)$ when reflecting successively at the sinks $q_0,q_a,q_b,q_1,q_c$ and $q_d$. Thus, keeping in mind Lemma \ref{hilfslemma} and Theorem \ref{refpoly}, a straightforward calculation yields
\begin{eqnarray*} F_{d_1^{0}(r)}&=&F_{\tau\hat d_1^{1}(r)}=F_{\tau \hat d_1^{1}(0)}F_{r\delta}-x^{\tau^{-1}\hat d_1^{0}(0)}F_{\delta-\tau^{-1}\hat d_1^{0}(0)}F_{(r-1)\delta}\\
&=&F_{d_1^0(0)}F_{r\delta}-x^{\tau^{-1}d^0_1(0)}F_{\delta-\tau^{-1}d^0_1(0)}F_{(r-1)\delta}.
\end{eqnarray*}
Again by Proposition \ref{subspace2} together with Lemma \ref{verlaengern}, we get
\[F_{d_2(r)}=F_{d_2(1)}F_{(r-1)\delta}-x^{\delta}F_{(r-2)\delta}.\]
Moreover, we get
\[F_{d_3^1(r)}=F_{d_3^1(1)}F_{(r-1)\delta}-x^{\delta}F_{(r-2)\delta}.\]
Similar to the case of $d_1^0(r)$, we have $\tau \hat d_3^1(r)=d_3^0(r)$. Let $\sigma:=\sigma_d\sigma_1\sigma_a\sigma_b\sigma_0$. It is straightforward to see that
\[F_{\sigma d_3^1(r)}=F_{\sigma d_3^1(1)}F_{(r-1)\delta}-x^{\delta}F_{(r-2)\delta}.\]
But since $(\sigma d_3^1(0))_c=0$ we get
\begin{eqnarray*}F_{\sigma_c \sigma d_3^1(r)}&=&F_{\sigma_c \sigma d_3^1(1)}F_{(r-1)\delta}-x^{\delta}(1+x_c^{-1})F_{(r-2)\delta}\\
&=&F_{d_3^0(1)}F_{(r-1)\delta}-x^{\tau^{-1}d_3^0(1)}F_{\delta-\tau^{-1}d_3^0(1)}F_{(r-2)\delta}.\end{eqnarray*}

Finally, also the formula for $F_{d_4(r)}$ for $n=5$ is obtained in this way.
In order to obtain the formulae for general $n\geq 4$, we can apply Lemma \ref{verlaengern} starting with the case $n=5$.
\end{proof}
For the remaining part of the subsection, we should keep the following remark in mind: 
\begin{rem}\label{bem1}
\begin{itemize}\item Using Lemma \ref{verlaengern2} together with Equation (\ref{genfunc}), we get
\[F_{d_2(r)}=(1+x_b^{-1})F_{r\delta}-x^{\udim P_b}F_{\delta-\udim P_b}F_{(r-1)\delta}.\]
It is likely that there is a similar formula for $d_3^{n-4}(r)$.
\item If $\alpha\leq \delta$ is a preprojective root such that $\sigma_q(\alpha)>\delta$, then $q$ is a source and $\delta-\alpha$ is the injective simple root corresponding to $q$. 
Indeed,  $\delta-\alpha$ is a preinjective root if $0<\alpha<\delta$ is preprojective. Since the positive non-simple roots are invariant under the Weyl group, $\delta-\alpha$ is forced to be simple.  
In particular, if $\alpha<\delta$ and $\tau^{-1}\alpha<\delta$, we have that $\sigma_q\tau^{-1}\alpha>\delta$ if and only if $\delta-\tau^{-1}\alpha$ is the simple root corresponding to the source $q$. The analogous statement holds if $\alpha$ is preinjective.
\item If $q$ is a sink, we have that $P_q=S_q$. If $t_M$ is preprojective with $\tau^{-1}t_M<\delta$, we have $\sigma_q\tau^{-1}t_M<\delta$ because otherwise $\delta-\tau^{-1}t_M=s_q$ were injective. In turn if $t_M<\delta$ and $\sigma_q\tau^{-1}t_M>\delta$, we already have $\tau^{-1}t_M>\delta$ in which case $\delta-t_M$ is injective.
\item If $\alpha$ is a preprojective root of defect $-1$ of $\tilde D_n$ in subspace orientation, we have that $\delta - \alpha$ is injective if and only if $\alpha=d_2(1)$ or $\alpha=d_3^{n-4}(1)$.

\item If $I_q$ is the injective representation corresponding to $q$ and $q'\neq q$ is a source, we have that $\sigma_{q'}I_q$ is also injective. Note that since $\tilde D_n$ is tree-shaped, we have $\udim (I_q)_{q'}=1$ if and only if there exists a (unique) path from $q'$ to $q$ and $\udim (I_q)_{q'}=0$ otherwise. 

\end{itemize}
\end{rem}

\begin{prop}\label{hilfslemma3}
Let $M$ be a preprojective representation of $\tilde D_n$ such that $t_M:=\udim M-r\delta\leq\delta$. Let $q$ be a sink.
\begin{enumerate}

\item 
Assume that 
\[F_M=F_{t_M}F_{r\delta}-x^{\tau^{-1}t_M}F_{\delta-\tau^{-1}t_M}F_{(r-1)\delta}.\]
If $\tau^{-1}t_M<\delta$ and $s_q\neq t_M\leq \delta$, we have
\[F_{\sigma_qM}=F_{\sigma_qt_M}F_{r\delta}-x^{\sigma_q\tau^{-1}t_M}F_{\delta-\sigma_q\tau^{-1}t_M}F_{(r-1)\delta}.\]
Then we also have $\sigma_q \tau^{-1}t_M<\delta$.
\item Assume that \[F_M=F_{t_M}F_{r\delta}-x^{\delta}F_{(r-1)\delta}\]
and that $\delta-t_M$ is the injective root corresponding to $q'$. Then we have $\delta-\tau^{-1}t_M=-\udim P_{q'}$. Moreover, if $q\neq q'$, we have  
\[F_{\sigma_qM}=F_{\sigma_qt_M}F_{r\delta}-x^{\delta}F_{(r-1)\delta}.\] 
 and if $q=q'$, (we have $\delta-t_M\neq s_q$ and) we have 
\[F_{\sigma_qM}=F_{\sigma_qt_M}F_{r\delta}-x^{\sigma_q\tau^{-1}t_M}F_{\delta-\sigma_q\tau^{-1} t_M}F_{(r-1)\delta}\]
where $\delta-\sigma_q\tau^{-1} t_M=s_q$.
\end{enumerate}
\end{prop}

\begin{proof}
For a sink $q$, by the second part of Lemma \ref{hilfslemma}, we have 
\[\sum_{p\in Q_0}a(p,q)\tau^{-1}(t_M)_p-(\tau^{-1}t_M)_q=(t_M)_q.\]
Thus we get
\[(t_M+r\delta)_q=\sum_{p\in Q_0}a(p,q)\tau^{-1}(t_M)_p+(\delta-\tau^{-1}t_M)_q+((r-1)\delta)_q.\]
Thus, since $t_M$ is not the simple root corresponding to $q$, the first part follows by Theorem \ref{refpoly}. 

For the second statement, note that $\sigma_{q}(\delta-t_M)$ is preinjective and $\udim I_q=s_{q}$ is the injective root of $\sigma_{q} Q$ corresponding to $q$. Indeed, $q$ is a sink of $Q$ and in turn a source of $\sigma_q Q$. For the first part, it suffices to show that
\[\sum_{p\in Q_0}a(p,q)\delta_p=(t_M)_q+\delta_q.\]
This can be deduced from 
\[\sum_{p\in Q_0}a(p,q)\delta_p-\delta_q=\sum_{p\in Q_0}a(p,q)(\tau^{-1}t_M)_p-(\tau^{-1}t_M)_q\]
which actually follows from $\tau^{-1}t_M-\delta=\udim P_{q'}$. Indeed, since $q$ is a sink and $Q$ tree-shaped, we have $\dim (P_{q'})_q=1$ if and only if there is exactly one neighbour $p$ such that $\dim (P_{q'})_p=1$. 

Since $q$ is a sink, we have $\udim P_q=s_q$. Thus the second part follows because 
\[x^{\delta}(1+x_q^{-1})=x^{\delta-s_q}F_{s_q}=x^{\sigma_q\tau^{-1}t_M}F_{\delta-\sigma_q\tau^{-1}t_M}\]
and $(\delta-\tau^{-1}t_M)_q=1$ in this case.
\end{proof}
Clearly the formulae hold in the other direction as well. This leads to the main result of this section:

\begin{thm}\label{Fpolydef-1}
Let $M$ be preprojective of defect $-1$ such that $t_M=\udim M-r\delta\leq\delta$. If $\delta-t_M$ is injective we have
\[F_M=F_{t_M}F_{r\delta}-x^{\delta}F_{(r-1)\delta}.\]
 If $\delta-t_M$ is not injective we have
\[F_M=F_{t_M}F_{r\delta}-x^{\tau^{-1}t_M}F_{\delta-\tau^{-1}t_M}F_{(r-1)\delta}.\]
\end{thm}

\begin{proof}
We proceed by induction. For $\tilde D_n$ in subspace orientation, the statement is true by Proposition \ref{subspace3} keeping in mind Remark \ref{bem1}. 

Now assume that $\alpha_1,\ldots,\alpha_l$ are all preprojective roots of defect $-1$ of $\tilde D_n$ with a fixed orientation such that $\alpha_i\leq\delta$. If $q$ is a sink, we can reflect at $q$ to obtain all preprojective roots $\sigma_q\alpha$ with $\sigma_q\alpha\leq\delta$ of $\sigma_qQ$ except $t:=\delta-s_q$. In particular, we can apply Proposition \ref{hilfslemma3} to obtain the generating functions corresponding to preprojectives of defect $-1$ except those of the form $r\delta+t$. 
Using the notation of Proposition \ref{hilfslemma3}, we are thus left with the case when $t_M=s_q$ where we can assume that
\[F_M=F_{t_M}F_{r\delta}-x^{\tau^{-1}t_M}F_{\delta-\tau^{-1}t_M}F_{(r-1)\delta}\]
because $\delta-s_q$ is not injective.
Since $t_M=s_q$, we have $\tau t_M=-\udim I_q$. In particular, since $\delta-(\delta+\tau t_M)$ is injective, by induction hypothesis, we have
\[F_{\tau M}=F_{\delta-\udim I_q}F_{(r-1)\delta}-x^{\delta}F_{(r-2)\delta}.\]
Now there exists an admissible sequence $$\sigma:=\prod_{\substack{q'\in (\tilde D_n)_0\\q'\neq q}}\sigma_{q'}$$ of reflections at sources such that $\sigma \tau M=\sigma_q M$. Since the first part of the second statement of Proposition \ref{hilfslemma3} clearly also holds in the opposite direction, keeping in mind the last part of Remark \ref{bem1}, we have
\[F_{\sigma_qM}=F_{\delta - \udim I_q}F_{(r-1)\delta}-x^{\delta}F_{(r-2)\delta}.\] 
Since $q$ is a source, we have $\udim I_q=s_q$ and the claim follows.

\end{proof}

 \enlargethispage{2\baselineskip}

\end{document}